\documentclass[11pt]{article}

\usepackage{amsmath,amssymb,amsfonts,amsthm,mathrsfs,bbm,bm, a4wide,leftidx}
\usepackage{enumerate}
\usepackage[all]{xy}
\usepackage[utf8]{inputenc}
\usepackage[T1]{fontenc}
\usepackage{palatino,color}
\usepackage[usenames,dvipsnames]{xcolor}
\usepackage{comment}

\usepackage{amssymb}
\usepackage{mathtools}
\usepackage{tikz}
\usetikzlibrary{chains}

\tikzset{node distance=2em, ch/.style={circle,draw,on chain,inner sep=2pt},chj/.style={ch,join},every path/.style={shorten >=4pt,shorten <=4pt},line width=1pt,baseline=-1ex}

\usepackage[
            breaklinks=true,
            pdftitle={},
            pdfauthor={Tamas Hausel, Michael Lennox Wong, Dimitri Wyss}]{hyperref}
\usepackage{color}

\newcommand{\mlabel}[1]{%
  \(#1\)
}

\let\dlabel=\alabel
\let\ulabel=\mlabel

\newcommand{\dnode}[2][chj]{%
\node[#1,label={below:\dlabel{#2}}] {};
}

\newcommand{\dnodea}[3][chj]{%
\dnode[#1,label={above:\ulabel{#2}}]{#3}
}

\newcommand{\dnodeanj}[2]{%
\dnodea[ch]{#1}{#2}
}

\newcommand{\QLeftarrow}{%
\begingroup
\tikz
\draw[shorten >=0pt,shorten <=0pt] (0,3pt) -- ++(-1em,0) (0,1pt) -- ++(-1em-1pt,0) (0,-1pt) -- ++(-1em-1pt,0) (0,-3pt) -- ++(-1em,0) (-1em+1pt,5pt) to[out=-105,in=45] (-1em-2pt,0) to[out=-45,in=105] (-1em+1pt,-5pt);
\endgroup
}

\definecolor{tocolor}{rgb}{.1,.1,.5}
\definecolor{urlcolor}{rgb}{.2,.2,.6}
\definecolor{linkcolor}{rgb}{.1,.4,.6}
\definecolor{citecolor}{rgb}{.6,.3,.1}
\hypersetup{colorlinks=true, urlcolor=urlcolor, linkcolor=linkcolor, citecolor=citecolor}

\newcommand{\FF}{\mathcal{F}}
\newcommand{\BF}{\mathbb{F}}
\newcommand{\BZ}{\mathbb{Z}}
\newcommand{\BN}{\mathbb{N}}
\newcommand{\BA}{\mathbb{A}}
\newcommand{\Z}{\mathbb{Z}}

\newcommand{\BC}{\mathbb{C}}
\newcommand{\C}{\mathbb{C}}
\newcommand{\R}{\mathbb{R}}
\newcommand{\BP}{\mathbb{P}}
\newcommand{\BL}{\mathbb{L}}
\newcommand{\Q}{\mathbb{Q}}
\newcommand{\BH}{\mathbb{H}}
\renewcommand{\O}{\mathcal{O}}

\newcommand{\MM}{\mathcal{M}}
\newcommand{\FM}{\mathcal{M}}
\newcommand{\NN}{\mathcal{N}}

\renewcommand{\O}{\mathcal{O}}
\newcommand{\FO}{\mathcal{O}}

\newcommand{\calP}{\mathcal{P}}
\renewcommand{\ll}{\mathfrak{l}}
\def\muhat{{\bm \mu}}

\def\r{{\bf r}}

\newcommand{\FP}{\mathcal{P}}
\newcommand{\W}{{\mathrm{W}}}
\newcommand{\T}{{\mathrm{T}}}
\renewcommand{\L}{\mathrm{L}}
\newcommand{\Ss}{\mathrm{S}}
\newcommand{\ZZ}{\mathrm{Z}}
\newcommand{\RR}{\mathrm{R}}
\newcommand{\FN}{\mathcal{N}}

\newcommand{\G}{{\mathrm{G}}}

\renewcommand{\P}{\mathbb{P}}
\newcommand{\F}{\mathcal{F}}
\newcommand{\OM}{\widehat{\Omega}}
\newcommand{\B}{\mathbf{B}}
\newcommand{\GL}{{\rm{GL}}}

\newcommand{\g}{\mathfrak{g}}

\newcommand{\x}{{\bf {x}}}

\renewcommand{\k}{\mathfrak{k}}

\renewcommand{\1}{\mathbbm{1}}
\newcommand{\unt}{\mathbbm{1}}
\newcommand{\s}{\mathfrak{s}}
\newcommand{\z}{\mathfrak{z}}
\newcommand{\gl}{\mathfrak{gl}}

\newcommand{\ft}{\mathfrak{t}}
\renewcommand{\t}{\mathfrak{t}}

\newcommand{\lla}{\left\langle }
\newcommand{\rra}{\right\rangle}
\newcommand{\CC}{\mathcal{C}}
\newcommand{\K}{\mathbb{K}}
\DeclareMathOperator{\ev}{ev}
\DeclareMathOperator{\Hom}{Hom}
\DeclareMathOperator{\End}{End}
\DeclareMathOperator{\Res}{Res}

\DeclareMathOperator{\spec}{\textnormal{Spec}}
\DeclareMathOperator{\Rep}{Rep}

\DeclareMathOperator{\tr}{tr}
\DeclareMathOperator{\res}{res}

\DeclareMathOperator{\im}{im}

\newcommand{\irr}{\textnormal{irr}}

\newcommand{\Ad}{\textnormal{Ad}}

\newcommand{\diag}{\textnormal{diag}}

\newcommand{\Id}{\textnormal{Id}}

\newcommand{\od}{\textnormal{od}}
\newcommand{\rc}{\textnormal{rc}}
\newcommand{\reg}{\textnormal{reg}}
\newcommand{\PGL}{\rm{PGL}}
\newcommand{\arsim}{\xrightarrow{\sim}}

\newcommand{\SL}{\textnormal{SL}}

\newcommand{\m}{\mathbf{m}}
\newcommand{\n}{\mathbf{n}}

\newcommand{\bfr}{\mathbf{r}}
\newcommand{\ucommtri}[6]{ \xymatrixrowsep{2.5pc}\xymatrixcolsep{1.25pc}
\xymatrix{
& #1 \ar[dl]_{#4} \ar[dr]^{#5} & \\
#2 \ar[rr]_-{#6} & & #3 } }

\newcommand{\kvar}{\mathop{\rm KVar}\nolimits}
\newcommand{\kexp}{\mathop{\rm KExpVar}\nolimits}
\newcommand{\expm}{\mathscr{E}xp\mathscr{M}}

\newcommand{\mat}[4]{\left[ \begin{array}{cc} #1 & #2 \\ #3 & #4 \end{array} \right]}

\theoremstyle{definition}
\newtheorem{thm}[equation]{Theorem}
\newtheorem{lem}[equation]{Lemma}
\newtheorem{lemma}[equation]{Lemma}
\newtheorem{corollary}[equation]{Corollary}

\newtheorem{prop}[equation]{Proposition}
\newtheorem{cor}[equation]{Corollary}
\newtheorem{ex}[equation]{Example}

\newtheorem{defn}[equation]{Definition}
\newtheorem{conj}[equation]{Conjecture}

\theoremstyle{remark}
\newtheorem{rmk}[equation]{Remark}
\newtheorem{rem}[equation]{Remark}

\numberwithin{equation}{subsection}
\begin{document}

\title{Arithmetic and metric aspects of open de Rham spaces}

\author{ Tam\'as Hausel \\ {\it IST Austria} \\{\tt tamas.hausel@ist.ac.at} \and Michael Lennox Wong \\ {\tt michael.lennox.wong@gmail.com } \and Dimitri Wyss \\ {\it EPFL} \\{\tt dimitri.wyss@epfl.ch}  
}

\maketitle

\begin{abstract} In this paper we determine the motivic class---in particular, the weight polynomial and conjecturally the Poincar\'e polynomial---of the open de Rham space, defined and studied by Boalch, of certain moduli spaces of irregular meromorphic connections on the trivial rank $n$ bundle on $\P^1$.  The computation is by motivic Fourier transform.  We show that the result satisfies the purity conjecture, that is, it agrees with the pure part of the conjectured mixed Hodge polynomial of the corresponding wild character variety.  We also identify the open de Rham spaces with quiver varieties with multiplicities of Yamakawa and Geiss--Leclerc--Schr\"oer.  We finish with constructing natural complete hyperk\"ahler metrics on them, which in the $4$-dimensional cases are expected to be of type ALF. \end{abstract}

\small
\tableofcontents
\normalsize

\section{Introduction} 
\noindent
In the paper \cite{hausel-villegas}  the Hodge structure on the cohomology of character varieties of  representations of the fundamental group of a Riemann surface was studied, using arithmetic harmonic analysis. It resulted in a conjecture   \cite[Conjecture 4.1]{hausel-villegas} on the mixed Hodge polynomial and observations \cite[Remark 4.4.2]{hausel-villegas} on the pure part.  This study was extended to character varieties of punctured Riemann surfaces in \cite{HLV11}, where a more geometric purity conjecture \cite[Remark 1.3.1]{HLV11} appeared. First we recall this conjecture. 

\subsection{Tame purity conjecture} \label{s:tamepureconj}
For $k,n\in \Z_{>0}$, let $\muhat=(\mu^1,\dots,\mu^k)\in \calP_n^k$ be a $k$-tuple of partitions of $n$.  Let $\G := \GL_n(\C)$.  An orbit $\O \subset \g := \gl_n(\C)$ for the adjoint action of $\G$ has type $\mu \in \calP_n$ if the partition given by the multiset of multiplicities of the eigenvalues of any element in $\O$ is $\mu$. Let 
$({\mathcal O}_1,\dots,{\mathcal O }_k) $ be a generic $k$-tuple of semisimple adjoint orbits in  $\g$ of type $\muhat$ in the sense of Definition~\ref{genericc}.  Then  the variety 
\begin{align} \label{e:mstartameintro}
\FM_\muhat^*=\{(A_1,A_2,\dots,A_k) \ | \ A_i\in {\mathcal O}_i,\ A_1+\dots+A_k=0\} /\! / \G,
\end{align}
constructed in  \cite[p.141-142]{Bo01} as an affine GIT quotient by the diagonal conjugation action of $\G$ is smooth. 

Fix distinct points $a_1,\dots, a_k \in {\mathbb P}^1 \setminus \{ \infty \}$. A point of $\FM_\muhat^*$ represented by $(A_1,A_2,\dots,A_k)$ yields a logarithmic connection $$\sum_{i=1}^k A_i \frac{dz}{z-a_i}$$ on the trivial rank $n$ bundle on $\P^1$  with residue in $\O_i$ at the point $a_i$ for $1 \leq i \leq k$.  We call $\FM_\muhat^*$ the {\em tame open de Rham space}, as it is open in the full moduli space of flat rank $n$ connections on ${\mathbb P}^1$ with logarithmic singularities and prescribed adjoint orbit of residues around $a_i$. 

Defining the conjugacy class $\CC_i=\exp(2\pi i \O_i) \subset \G$ we get a generic $k$-tuple $({\CC}_1,\dots,{\CC}_k)$ of semisimple conjugacy classes of type $\muhat$. We define the {\em tame character variety} of type $\muhat$ as the smooth affine GIT quotient \cite[p.141-142]{Bo01}
\begin{align} \label{e:tamecharvar}
\FM_\B^\muhat:= \{ (M_1,\dots,M_k) \, | \, M_i \in \CC_i, \, M_1\cdots M_k= \1_n\}/\! / \G,
\end{align}
where $\1_n$ is the $n \times n$ identity matrix.

The character variety parametrizes isomorphism classes of $n$-dimensional representations of the fundamental group of ${\mathbb P}^1\setminus \{a_1,\dots, a_k\}$, with monodromy around $a_i$ in ${\mathcal C}_i$.  One has the Riemann--Hilbert monodromy map $$\nu_{a} :\FM_\muhat^*\to \FM_\B^\muhat,$$ 
taking the flat connection $\sum_{i=1}^k A_i \frac{dz}{z-a_i}$ to the representation given by its monodromy along loops in ${\mathbb P}^1\setminus \{a_1,\dots, a_k\}$.  Although this monodromy map is not algebraic, we have the following:

\begin{conj}[Purity conjecture]\label{purity} The map 
	$\nu_a^*: H^*(\FM_{\B}^{\muhat},\Q)\to H^*(\FM_{\muhat}^*,\Q)$ is surjective, it preserves mixed Hodge structures and is an isomorphism on the pure parts.
	\end{conj}
	
As the mixed Hodge structure	of $\FM^*_\muhat$ is known to be pure \cite[Proposition 2.2.6]{HLV11}, the conjecture implies that the pure part of the cohomology of $\FM_\B^\muhat$ is isomorphic to the full cohomology of $\FM^*_\muhat$. The consistency of \cite[Conjecture 1.2.1]{HLV11} on the mixed Hodge polynomial of $\FM_\B^\muhat$ with Conjecture~\ref{purity} above was tested by checking in \cite[Theorem 1.3.1]{HLV11} that the pure part of the conjectured mixed Hodge polynomial of $\FM_\B^\muhat$ agrees with the weight polynomial of $\FM^*_\muhat$.
	
The proof of the purity of $H^*(\FM^*_\muhat;\Q)$ in \cite[Proposition 2.2.6]{HLV11} proceeds by recalling \cite{CB03} the identification of $\FM^*_\muhat$ as a certain star-shaped Nakajima quiver variety. In particular, $\FM^*_\muhat$ acquires a natural complete hyperk\"ahler metric of ALE type. For example, for the star-shaped affine Dynkin diagrams of $\tilde{D}_4$, $\tilde{E}_6$, $\tilde{E}_7$ and  $\tilde{E}_8$ together with an imaginary root we obtain the corresponding ALE gravitational instantons of Kronheimer \cite{kronheimer}.

\subsection{Irregular purity conjecture}

The tame moduli spaces and the Riemann-Hilbert monodromy map above were generalised to allow irregular singularities in \cite{Bo01}. The aim of the present paper  is to extend the above purity ideas to the case of meromorphic connections with irregular singularities. 

 Consider the following analogue of \eqref{e:mstartameintro}.  Fix $k, n, s \in \Z_{>0}$ and we will take $\muhat$ and $(\O_1, \ldots, \O_k)$ as in Section \ref{s:tamepureconj}.  In addition, for $1 \leq i \leq s$, let $m_i \in \Z_{>1}$ and consider the truncated polynomial ring $\RR_{m_i} := \C[z]/(z^{m_i})$, and the group $\GL_n(\RR_{m_i})$ of invertible matrices over $\RR_{m_i}$.  Let $\T \subset \G$ be the maximal torus of diagonal matrices and $\t$ denote its Lie algebra.  We will fix an element of the form
\begin{align} 
C^i = \frac{C_{m_i}^i}{z^{m_i}} + \frac{C_{m_i-1}^i}{z^{m_i-1}} + \cdots + \frac{C_1^i}{z},
\end{align}
with $C_j^i \in \t$, further assuming that $C_{m_i}^i$ has distinct eigenvalues.  An element $g \in \GL_n(\RR_{m_i})$ acts on $C^i$ by conjugation, viewing both $g$ and $C^i$ as matrices over the ring of Laurent polynomials over $z$; however, we will truncate any terms with non-negative powers of $z$.  We denote the $\GL_n(\RR_{m_i})$-orbit of $C^i$ under this action by $\O(C^i)$ (it is explained in Section \ref{s:jd} how we may view $C^i$ as an element of the dual of the Lie algebra of $\GL_n(\RR_{m_i})$, and $\O(C^i)$ as its coadjoint orbit).  Observe that $\G = \GL_n(\C)$ sits in each $\GL_n(\RR_{m_i})$ as a subgroup and so acts on each $\O(C^i)$; in fact, given an element 
\begin{align} \label{e:eltintro}
Y^i = \frac{Y_{m_i}^i}{z^{m_i}} + \frac{Y_{m_i-1}^i}{z^{m_i-1}} + \cdots + \frac{Y_1^i}{z} \in \O(C^i),
\end{align}
the $\G$-action is by conjugation on each term $Y_p^i$.  Now, for $1 \leq i \leq s$, we set $r_i := m_i - 1$, and write $\bfr := (r_1, \ldots, r_s)$ for the tuple.  We may then construct the \emph{(irregular) open de Rham space} as the smooth affine GIT quotient 
\begin{align*}
\FM^*_{\muhat, \bfr} := \left\{ (A_1, \ldots, A_k, Y^1, \ldots, Y^s) \in \prod_{j=1}^k \O_j \times \prod_{i=1}^s \O(C^i) \, : \, \sum_{j=1}^k A_j + \sum_{i=1}^s Y_1^i = 0 \right\} \bigg/\!\!\!\bigg/ \G.
\end{align*}

One likewise has an interpretation of $\FM^*_{\muhat, \bfr}$ as a moduli space of meromorphic connections (see Section \ref{modcon}), this time with poles of higher order, on the trivial rank $n$ vector bundle over $\P^1$. The class of $(A_j, Y^i)$ yields a connection
\begin{align*}
\sum_{j=1}^k A_j \frac{dz}{z - a_j} + \sum_{i=1}^s \sum_{p=1}^{m_i} Y_p^i \frac{dz}{(z-b_i)^p}
\end{align*}
for a set of distinct poles $\{ a_1, \ldots, a_k, b_1, \ldots, b_s \} \in \P^1 \setminus \{ \infty \}$.

The definition of the corresponding \emph{wild character variety} $\FM_\B^{\muhat, \bfr}$, which is the space of monodromy data for moduli spaces of irregular connections, is a little more involved than in the logarithmic case (see \cite[Equation (2)]{Boalch2007}).  For the poles of higher order $b_i$, in addition to the topological monodromy, one must also take into account the Stokes data, which distinguish analytic isomorphism classes of locally defined connections from formal ones.  However, even in this irregular case, the wild character variety retains certain similarities with the character variety defined above at \eqref{e:tamecharvar}:  it is a smooth affine variety defined only in terms of $\G$, certain algebraic subgroups of $\G$, and orbits in them, and may be viewed as a space of ``Stokes representations'' \cite[\S3]{Bo01}, though we will not be working with this space directly.  


There is again a Riemann--Hilbert monodromy map \cite[Corollary 1]{Boalch2007}
\begin{align}\label{irrRH}
\nu : \FM_{\muhat, \bfr}^* \to \FM_{\B}^{\muhat, \bfr}
\end{align}
in this irregular case, which takes a connection to its monodromy data.  An explanation of why the wild character variety takes the form that it does and a detailed description of the monodromy map $\nu$ may be found at \cite[\S 3]{Bo01}.  We then have

\begin{conj}[Irregular purity conjecture]\label{irrpurity} The irregular Riemann--Hilbert map $\nu$ \eqref{irrRH} induces a surjective map $\nu^*: H^*(\FM_{\B}^{\muhat, \bfr},\Q)\to H^*(\FM_{\muhat, \bfr}^*,\Q)$ which  preserves mixed Hodge structures and is an isomorphism on the pure parts.
\end{conj}

\subsection{Main results and layout of the paper}

To formulate our main result  we fix integers $g\geq 0$ and $k>0$. Let
$\x_1=\{x_{1,1},x_{1,2},\dots\}, \dots,
\x_k=\{x_{k,1},x_{k,2},\dots\}$ be $k$  sets of infinitely many independent
variables and let $\Lambda(\x_1,\ldots,\x_k)$ be the ring of functions
separately symmetric in each of the sets of variables.

We define the Cauchy kernel
$$\Omega_k(z,w):=\sum_{\lambda\in \calP}{\cal H}_\lambda(z,w)\prod_{i=1}^k \tilde{H}_\lambda(z^2,w^2;{\bf x}_i)\in \Lambda({\bf x}_1,\dots,{\bf x}_k)\otimes_\Z \Q(z,w),$$ where
$$
{\cal H}_{\lambda} (z,w):=\prod
\frac{(z^{2a+1}-w^{2l+1})^{2g}} {(z^{2a+2}-w^{2l})(z^{2a}-w^{2l+2})}
$$
is a $(z,w)$-deformation of the $(2g-2)$-th power of the standard hook
polynomial - where the product goes through the boxes in the Young tableaux of $\lambda$, and $a$ and $l$ are the arm length and leg length of the corresponding box -  and $$\tilde{H}_\lambda(z^2,w^2;{\bf x}_i)\in \Lambda({\bf x}_i)\otimes_\Z \Q(q,t)$$ is the modified Macdonald symmetric function defined in \cite[(11)]{garsia-haiman}; see \cite[\S 2.3.4]{HLV11} for more details. Finally, we let \small
\begin{align*} &\BH_{\muhat,r}(z,w)=\\ & (-1)^{rn} (z^2-1)(1-w^2)\left\langle {\rm Log} \left(\Omega_{k+r}\right),h_{\mu_1}({\bf x}_1)\otimes \cdots \otimes h_{\mu_k}({\bf x}_k)\otimes s_{(1^n)}({\bf x}_{k+1})\otimes \cdots \otimes s_{(1^n)}({\bf x}_{k+s})\right\rangle,
\end{align*} \normalsize
where
$h_\mu({\bf x}_i)$ are the
complete symmetric functions, $s_{(1^n)}({\bf x}_{j})$ are the Schur symmetric functions in the corresponding variables,  and $\langle\cdot,\cdot\rangle$ is the
extended Hall pairing and ${\rm Log}$ is the plethystic logarithm. The notation is explained in more detail in Section~\ref{s:wcvcomp}; see also \cite[\S 2.3]{HLV11} for detailed explanations of the formalism. 

Note that $\BH_{\muhat,r}(-z,w)$ as defined is a rational function in $\Q(z,w)$, but conjecturally, because of Conjecture~\ref{mainconj} below, it is a polynomial function in $\Q[z,w]$ with positive integer coefficients.

 One of our main results is the computation of the weight polynomial of the open de Rham space $\FM_{\muhat, \bfr}^*$.

\begin{thm} \label{main} $\sum_{k,i \geq 0} (-1)^i \dim_\BC \left( Gr^W_{2k} H^i_c(\FM_{\muhat, \bfr}^*,\BC) \right) q^{k}=q^{d_{\muhat,\r}/2} \BH_{\tilde{\muhat},r}(0,q^{1/2})$
\end{thm}

The main conjecture of \cite[Conjecture 0.2.2]{HMW16} - formulated here for compactly supported cohomology by Poincar\'e duality for the smooth variety $\FM^{\muhat, \bfr}_\B$ - claims that  \begin{conj} \label{mainconj} $$\sum_{i,k \geq 0}\dim_\BC \left(Gr^W_{2k} H^i_c(\FM^{\muhat, \bfr}_\B,\BC) \right) q^{k}t^i=(qt^2)^{d_{\muhat,\r}/2} \BH_{\tilde{\muhat},r}(-q^{-1/2},tq^{1/2}).$$\end{conj}

Thus, if $H^*(\FM_{\muhat, \bfr}^*,\Q)$ were pure, our main Theorem \ref{main} would be a consequence of our purity conjecture Conjecture~\ref{irrpurity} and Conjecture~\ref{mainconj}.  However, we were unable to prove that $H^*(\FM_{\muhat, \bfr}^*,\Q)$ is always pure and we only state it as Conjecture \ref{purcon}. 

The proof of Theorem~\ref{main} is first performed, as Theorem \ref{mainodr}, in the case of $k=0$, i.e., only irregular punctures. In this case, we can proceed by motivic Fourier transform, as in \cite{WY16}, and the result will be a motivic extension of Theorem~\ref{main}. The general case---Theorems \ref{finalcount} and \ref{maincor}---is then proved via the arithmetic harmonic analysis technique of \cite{HLV11}. 

As an analogue of Crawley-Boevey's result \cite{CB03} for the irregular case, in Section~\ref{s:qmv} we prove that the open de Rham spaces $\FM_{\muhat, \bfr}^*$ are  isomorphic to quiver varieties with multiplicities. These varieties have been considered by Yamakawa \cite{YA10} in the rank $2$ case and their defining equations, in terms of certain preprojective algebras, have been studied by Geiss--Leclerc--Schr\"oer \cite{GeissLeclercSchroeer2017} in general.  Then, in Section~\ref{s:nonslDD}, we consider the star-shaped non-simply laced affine Dynkin diagrams which correspond to open de Rham spaces of dimension $2$. Here we discuss the main results of the paper in these special toy cases. 

Finally, in Section~\ref{s:HK} we prove the existence of natural complete hyperk\"ahler metrics on $\FM_{\muhat, \bfr}^*$ when the irregular poles have order $2$.  Existence of such metrics was discussed in \cite[\S3.1]{Boalch2012}), but as there seem to be no complete proofs given in the literature, we provide the details here.  In the (real) four-dimensional toy example cases, e.g., those appearing in Section~\ref{s:nonslDD}, we expect the resulting metrics to be of type ALF.

\paragraph{Acknowledgements} We would like to thank Gergely B\'erczy, Roger Bielawski, Philip \linebreak Boalch, Sergey Cherkis, Andrew Dancer, Brent Doran, Elo\"{i}se Hamilton, Frances Kirwan, Bernard Leclerc, Emmanuel Letellier, Alessia Mandini, Maxence Mayrand, Andr\'as N\'emethi, Szil\'ard Szab\'o, Daisuke Yamakawa for discussions related to the paper. We  especially thank the referee for an extensive  list of very careful comments.   At various stages of this project the authors were supported   by the  Advanced Grant ``Arithmetic and physics of Higgs moduli spaces'' no. 320593 of the European Research Council, by grant no. 153627 and NCCR SwissMAP,  both funded by the Swiss National Science Foundation as well as by EPF Lausanne and IST Austria.  In the final stages of this project, MLW was supported by SFB/TR 45 ``Periods, moduli and arithmetic of algebraic varieties'', subproject M08-10 ``Moduli of vector bundles on higher-dimensional varieties''.  DW was also supported by the Fondation Sciences Math\' ematiques de Paris, as well as public grants overseen by the Agence national de la recherche (ANR) of France as part of the « Investissements d’avenir » program, under reference numbers ANR-10-LABX-0098 and ANR-15-CE40-0008 (D\'efig\'eo). 

\section{Preliminaries} 

\subsection{Groups, Lie algebras and their duals over truncated polynomial rings} \label{s:jd}

Let us fix a perfect base field $\K$ and an integer $n\geq 1$. We will set $\G := \GL_n(\K)$ and denote by $\g :=\mathfrak{gl}_n(\K)$ its Lie algebra.  Furthermore, $\T \subset \G$ will denote the standard maximal torus consisting of the invertible diagonal matrices, $\ft \subset \g$ its Lie algebra and $\ft^{\reg} \subset \ft$ the subset of elements with distinct eigenvalues.  Fix another integer $m \geq 1$ and let $\RR_m := \K[\![z]\!]/ (z^m) = \K[z]/(z^m)$.  Then we may also consider these groups over $\RR_m$,
\begin{align} \label{e:Gmelts} \G_m & := \GL_n(\RR_m) = \left\{ g_0 + zg_1+\dots +z^{m-1}g_{m-1} \ |\ g_0 \in \G, \ g_1,\dots,g_{m-1} \in \g \right\},\\
\g_m & := \mathfrak{gl}_n(\RR_m) = \left\{X_0+zX_1+\dots + z^{m-1}X_{m-1}\ |\ X_i \in \g \right\}, \nonumber
\end{align}
and we define $\T_m$ and $\ft_m$ similarly.  We will regard $\G_m$ and $\T_m$ as algebraic groups over $\K$:  from the description above, $\G_m = \G \times \g^{m-1}$ as a $\K$-variety; writing out the components of each $z^i$ under the group law in $\G_m$, it is easy to see that the operation is well-defined on tuples and so one indeed gets an algebraic group over $\K$.  Of course, $\g_m$ is a vector space over $\K$.

Observe that $\GL_n(\RR_m)$ is not reductive; its unipotent radical $\G_m^1 \subset \G_m$ and Lie algebra $\g_m^1$ are, respectively,
\begin{align*} 
\G_m^1 = \GL_n^1(\RR_m) &= \{ \1_n + zb_1 +z^2b_2 +\dots + z^{m-1}b_{m-1}\ |\ b_i \in \g \},\\
\g_m^1 = \gl_n^1(\RR_m) & = \{ zX_1 +z^2X_2 +\dots + z^{m-1}X_{m-1} \ |\ X_i \in \g\},
\end{align*}
where $\1_n$ denotes the $n \times n$ identity matrix.

There is a semi-direct product decomposition 
\begin{align} \label{e:Gsemdir}
\G_m = \G_m^1 \rtimes \G, 
\end{align}
where we identify $\G$ with the subgroup of those elements satisfying 
\begin{align*}
g_1 = \cdots = g_{m-1} = 0,
\end{align*}
in the notation of \eqref{e:Gmelts}; we will often refer to $\G$ identified as such as the subgroup of constant elements in $\G_m$.  We thus obtain a direct sum decomposition 
\begin{align} \label{e:gmdecomp}
\g_m = \g_m^1 \oplus \g; 
\end{align}
this decomposition is preserved by the adjoint action of $\G$ but not of $\G_m$.  We will write $\T_m^1 := \T_m \cap \G_m^1$ and $\t_m^1 := \t_m \cap \g_m^1$.

It will be convenient to identify the dual vector space $\g_m^\vee$ with 
\begin{align} \label{e:gmvee} 
z^{-m}\g_m = \left\{ z^{-m}Y_m + z^{-(m-1)}Y_{m-1}+\dots + z^{-1}Y_1 \ \bigg| \ Y_i \in \g \right\}
\end{align}
via the trace residue pairing. This means that for $X \in \g_m$ and $Y \in z^{-m}\g_m$ we set
\begin{equation} \label{trp} \lla Y,X \rra := \Res_{z=0} \tr YX = \sum_{i=1}^{m} \tr Y_{i}X_{i-1}.\end{equation}
Under this identification, the dual $\g^\vee$ of the subgroup $\G \subseteq \G_m$ corresponds to the subspace $z^{-1}\g \subset z^{-m}\g_m$ and $(\g_m^1)^\vee$ to those elements in $z^{-m}\g_m$ having zero residue term ,i.e., $Y_1=0$. We write
\begin{align}\label{pis} & \pi_{\res} : \g_m^\vee \rightarrow \g^\vee, & & \pi_{\irr}: \g_m^\vee \rightarrow (\g_m^1)^\vee \end{align}
for the natural projections.  The latter projection may be identified with
\begin{align} \label{e:irrproj}
z^{-m} \g_m \to z^{-m} \g_m \big/ z^{-1} \g_m,
\end{align}
so we are simply truncating the residue term.

The adjoint and coadjoint actions of $\G_m$ on $\g_m$ and $\g_m^\vee$ will both be denoted by $\Ad$ and are defined by the same formula: for $g \in \G_m$, 
\begin{align*}
\Ad_g X & = gXg^{-1}, \quad X \in \g_m & \Ad_g Y & = g Y g^{-1}, \quad Y \in \g_m^\vee.
\end{align*}
What we mean in the latter case is that we consider $g$, $g^{-1}$ and $Y$ as matrix-valued Laurent polynomials in $z$ and we truncate all terms of non-negative degree in $z$ after multiplying.  With this convention we have 
\begin{align*}
\lla \Ad_gY ,X\rra = \lla Y,\Ad_{g^{-1}}X\rra.
\end{align*}

Consider a tuple\footnote{We will use $\BN$ and $\BZ_{>0}$ interchangeably.} $\lambda = (\lambda_0, \ldots, \lambda_l) \in \BN^{l+1}$ with $\sum_{i=1}^{l+1} \lambda_i = n$. Let $\L_\lambda \subset \G$ be the subgroup of block diagonal matrices, with blocks of sizes $\lambda_0, \ldots, \lambda_l$, i.e.,
\begin{align} \label{e:Llambda}
\L_\lambda = \left\{ \diag(f_0, \ldots, f_l) \in \G \, : \, f_i \in \GL_{\lambda_i}(\K) \right\}. 
\end{align}
Let $\ll_\lambda$ denote its Lie algebra.  The centre of $\ll_\lambda$ is given by
\begin{align*}
\z(\ll_\lambda) = \left\{ \diag( c_0 \1_{\lambda_0}, \ldots, c_l \1_{\lambda_l}) \, : \, c_0, \ldots, c_l \in \K \right\} \subseteq \t.
\end{align*}
We will write $\z(\ll_\lambda)^\reg$ for the subset of $\z(\ll_\lambda)$ for which the $c_i$ are pairwise distinct.  The center of $\L_\lambda$ satisfies $\ZZ(\L_\lambda) = \z(\ll_\lambda) \cap \G$, and can be described as the subset of $\z(\ll_\lambda)$ with all $c_i \in \K^\times$.  An important special case is when $\lambda = (1, \ldots, 1)$, so one has $\ll_\lambda = \t$ and $\z(\ll_\lambda)^\reg = \t^\reg$.  We will also use the notation as above for these groups, namely
\begin{align*}
\L_{\lambda, m} & := \L_\lambda(\RR_m) & \L_{\lambda, m}^1 & := \L_{\lambda, m} \cap \G_m^1 & \ll_{\lambda, m} & := \ll_\lambda(\RR_m) & \ll_{\lambda, m}^1 & := \ll_{\lambda, m} \cap \g_m^1.
\end{align*}

\subsection{Coadjoint orbits}

Coadjoint orbits for groups of the form $\G_m$ will play a prominent role in this paper, so here we will set some notation and record some results that will be used later. 

\subsubsection{Conventions and variety structure}

By a \emph{diagonal element} or \emph{formal type of order $m$} we will mean an element  in $\g_m^\vee$ of the form
\begin{align} \label{e:diagelt}
C = \frac{C_m}{z^m} + \frac{C_{m-1}}{z^{m-1}} + \cdots + \frac{C_1}{z} \in \t_m^\vee = z^{-m} \t_m, \quad \text{ with } C_\ell \in \t(\K), \quad 1 \leq \ell \leq m.
\end{align}
By permuting diagonal entries if necessary, $C$ may be written
\begin{align} \label{e:C}
C = \diag(c^0 \1_{\lambda_0}, c^1 \1_{\lambda_1}, \ldots, c^l \1_{\lambda_l})
\end{align}
for some partition $(\lambda_0, \ldots, \lambda_l)$ of $n$, with the $c^i \in z^{-m} \RR_m$ distinct.  In fact, we will make the stronger assumption that for all $0 \leq i \neq j \leq l$ 
\begin{align} \label{e:cassum}
z^m(c^i - c^j) \in \RR_m^\times.
\end{align}

This assumption is stronger than what is required for example in \cite[Main assumption]{BiquardBoalch}, but many of our arguments will rely crucially on \eqref{e:cassum}.

If we write $c^i = \frac{c_m^i}{z^m} + \cdots + \frac{c_1^i}{z}$ with $c_\ell^i \in \K$, $1 \leq \ell \leq m$, then \eqref{e:cassum} is equivalent to 
\begin{align*} 
c_m^i \neq c_m^j
\end{align*}
for all $0 \leq i \neq j \leq l$; that is, all the leading coefficients are distinct.  We fix such a diagonal element $C$ and write $\O(C)$ for the coadjoint orbit of $C$ i.e. the image of the orbit map
\begin{align*} \eta: \G_m &\to \gl_m^\vee \\
												g	& \mapsto \Ad_g(C).
\end{align*}

\begin{lem} \label{l:coadjorb} Let $C$ be a diagonal element satisfying \eqref{e:cassum}.

\begin{enumerate}[(a)]
\item \label{l:coadjaff} The coadjoint orbit $\O(C)$ is an affine variety.
\item \label{l:coadjhomog} Let $\ZZ_{\G_m}(C)$ denote the centralizer of $C$ in $\G_m$.  There is a $\G_m$-equivariant isomorphism $\phi : \G_m/\ZZ_{\G_m}(C) \xrightarrow{\sim} \O(C)$ such that the diagram
\begin{align*}
\ucommtri{\G_m }{ \G_m/Z_{\G_m}(C) }{ \O(C) }{}{\eta}{\phi}
\end{align*}
commutes. In particular $\O(C)$ is a homogeneous space for $\G_m$ and $\eta$ is a categorical quotient.
\end{enumerate}
\end{lem}

We will give a description of the centralizers $Z_{\G_m}(C)$ for certain $C$ in Lemma \ref{intor}\eqref{l:centralizers} below.

\begin{proof}
To see \eqref{l:coadjaff}, we note that $\O(C)$ will be the set of elements satisfying the appropriate minimal polynomial over $\RR_m$, which may then be written down as equations over $\K$.  To be explicit about this, let $Y \in \g_m^\vee$ be written as in \eqref{e:gmvee}; we consider $z^m Y$ as an element of $\g_m$ by writing
\begin{align*}
z^m Y = Y_m + z Y_{m-1} + \cdots + z^{m-1} Y_1.
\end{align*}
Similarly for $1 \leq i \leq l$ we consider $z^m c^i \in \RR_m$ by writing
\begin{align*}
z^m c^i = c_m^i + z c_{m-1}^i + \cdots + z^{m-1} c_1^i.
\end{align*}
Then the minimal polynomial as mentioned gives defining equations for $\O(C)$ in the sense that the following matrix over $\RR_m$ is zero if and only if $Y \in \O(C)$:
\begin{align*}
(z^m Y - z^m c^0 \1_n) \cdots (z^m Y - z^m c^l \1_n).
\end{align*}
Of course, this can be expanded and the coefficient of each power $z^i$ gives a matrix equation over $\K$; the individual entries of all of these matrices give the algebraic equations for $\O(C)$.  The fact that these are indeed defining equations for $\O(C)$ can be proved in the same way one proves the theorem which asserts that a matrix is diagonalizable if and only if its minimal polynomial has distinct roots; for this we need to use the assumption \eqref{e:cassum}.

For \eqref{l:coadjhomog} notice that both $\G$ and $\FO(C)$ are smooth over $\K$ and for every closed point $A \in \FO(C)$ one has $\dim \eta^{-1}(A) = \dim \ZZ_{\G_m}(C)$. Since $\eta$ is surjective it is thus faithfully flat by \cite[Theorem 23.1]{Ma89} and \cite[\href{https://stacks.math.columbia.edu/tag/00HQ}{Tag 00HQ}]{stacks-project} and $\phi$ an isomorphism by \cite[Proposition 7.11]{Mi17}.
\end{proof}

\subsubsection{Computational lemmata}

Fix a tuple $\lambda = (\lambda_0, \ldots, \lambda_l) \in \BN^{l+1}$ with $\sum \lambda_i = n$ and consider the group $\L_\lambda \subset \G$ as in \eqref{e:Llambda}.  Let $\ll_\lambda^\od \subseteq \g$ denote the subspace consisting of matrices with zeros on the diagonal blocks (of course, the superscript can be read, "off diagonal"), so that we have an obvious $\L_\lambda$-invariant decomposition
\begin{align} \label{e:llamdirsum}
\g = \ll_\lambda \oplus \ll_\lambda^\od.
\end{align}
In the case $\lambda = (1,\ldots, 1)$, we set $\g^{\od} := \ll_\lambda^\od$, so that this is the space of matrices with zeroes along the main diagonal.

Let $b = \1_n + \sum_{i=1}^{m-1} z^i b_i \in \G_m^1$.  We may write 
\begin{align*}
b^{-1} = \1_n + \sum_{i=1}^{m-1} z^i w_i,
\end{align*}
where the $w_i \in \g$ are given by
\begin{equation}\label{explinvx} w_i = \sum_{p=1}^i (-1)^p \sum_{\substack{(\lambda_1, \lambda_2, \dots, \lambda_p) \in \BN^p \\ \lambda_1 + \dots + \lambda_p = i}} b_{\lambda_1} \cdots b_{\lambda_p}. \end{equation}

Suppose $A = \sum_{j=1}^m z^{-j} A_j \in \g_m^\vee$.  Using the fact that $w_i = -b_i + \dots$, we have the explicit formula
\begin{align} \label{e:coadjointaction}
\Ad_b A &= \left( \1_n + \sum_{i=1}^{m-1} z^i b_i\right) A \left(\1_n + \sum_{i=1}^{m-1} z^i w_i \right) \nonumber \\
&= A + \sum_{p=1}^{m-1} z^{-m+p} \left( \sum_{i=1}^p [b_i, A_{m+i-p}] + \sum_{i=1}^{p-1} \sum_{j=1}^{p-i} [b_i, A_{m-p+i+j}] w_j \right).
\end{align}

\begin{lemma} \label{intor} \begin{enumerate}[(a)]
\item If $X \in \g$ and $Y \in \z(\ll_\lambda)$, then $[X,Y] \in \ll_\lambda^\od$.
\item \label{l:intorb} If $X \in \g$, $Y \in \z(\ll_\lambda)^\reg$ and $[X, Y] \in \ll_\lambda$ (or equivalently, by \eqref{e:llamdirsum}, $[X, Y] = 0$), then $X \in \ll_\lambda$.
 
\item \label{secnd} Let 
\begin{align*}
C = \sum_{i=1}^m z^{-i} C_i \in \z(\ll_\lambda)_m^\vee \subseteq \t_m^\vee
\end{align*}
be given as in \eqref{e:C} with $C_m \in \z(\ll)^{\reg}$ and suppose $g \in \G_m$ is such that $\Ad_g C - C \in \ll_{\lambda, m-1}^\vee$. Then $g \in \L_{\lambda, m}$ and $\Ad_g C = C$. 

\item \label{l:centralizers} Let $C$ be as in \eqref{secnd}.  Then the centralizers of $C$ and $C^1 := \pi_{\irr}(C) \in (\t_m^1)^\vee$ are 
\begin{align*}
Z_{\G_m} (C) & = \L_{\lambda, m} & Z_{\G_m^1} (C^1) & = G_{\lambda, m}^{1,\rc} := \left\{ b = \1_n + \sum_{j=1}^{m-1} z^j b_j \in \G_m^1\ \bigg| \ b_1,\ldots,b_{m-2} \in \ll_\lambda \right\}.
\end{align*}
The superscript $\rc$ stands for "regular centralizer" which is justified by the statement.

\item \label{intorr} Let $g \in Z(\L_\lambda)$ and $h \in \G_m^1$ such that $hgh^{-1} \in \L_{\lambda, m}$. Then $hgh^{-1} = g$.
\end{enumerate}
\end{lemma}

\begin{proof} The first two statements can be seen by writing everything out as block matrices. For part \eqref{secnd}, by \eqref{e:Gsemdir}, we may write $g = g_0 b$ for some $g_0 \in G$, $b \in \G_m^1$.  Thus, $\Ad_g C$ can be obtained by applying $\Ad_{g_0}$ to the expression in \eqref{e:coadjointaction}.  The $z^{-m}$ term is then $\Ad_{g_0} C_m$; the hypothesis is that this is $C_m$.  Then part \eqref{l:intorb} implies that $g \in \ll_\lambda \cap \G = \L_\lambda$.  Already this shows that $\Ad_{g_0} C = C$, since all $C_i \in \z(\ll_\lambda)$.

Now, the $z^{-(m-1)}$ term in $\Ad_g C - C$ is then $\Ad_{g_0} [b_1, C_m]$.  The assumption is that this lies in $\ll_\lambda$; by $\L_\lambda$-invariance, $[b_1, C_m] \in \ll_\lambda$ and again by \eqref{l:intorb}, $b_1 \in \ll_\lambda$.  By induction, we may assume $b_1, \ldots, b_r \in \ll_\lambda$.  We will show that $b_{r+1} \in \ll_\lambda$.  Then from \eqref{e:coadjointaction}, the $z^{-(m-r-1)}$-term of $\Ad_g C - C$ is
\begin{align*}
\Ad_{g_0} \left( \sum_{i=1}^{r+1} [b_i, C_{m+i-r-1}] + \sum_{i=1}^r \sum_{j=1}^{r-i+1} [b_i, C_{m-r+i+j-1}] w_j \right).
\end{align*}
The induction hypothesis implies that the only commutator that does not vanish is $[b_{r+1}, C_m]$ and then by assumption, this lies in $\ll_\lambda$, and again we conclude by \eqref{l:intorb}.  This shows that all $b_i \in \ll_\lambda$ and hence $h \in \L_{\lambda, m}^1$; as $g_0 \in \L_\lambda$, $g = g_0 b \in \L_{\lambda,m}$.  Furthermore, as above, we can show that all commutators in \eqref{e:coadjointaction} vanish, and we have already noted $\Ad_{g_0} C = C$; hence $\Ad_g C = C$.

The first assertion in part \eqref{l:centralizers} follows from \eqref{secnd}.  The second assertion uses the same argument, and one simply needs to observe that we are omitting the residue term when computing in $\g_m^1$ and hence no condition is imposed on $b_{m-1}$.

Part \eqref{intorr} is proved with an inductive argument similar to that of \eqref{secnd}, using \eqref{explinvx}.
\end{proof}

Next, we study regular semisimple coadjoint orbits.  We take $C \in \t_m^\vee$, let $C^1 := \pi_{\irr}(C) \in (\t_m^1)^\vee$ with $C_m \in \ft^{\reg}$.  We will write $\FO(C)$ for the $\G_m$-coadjoint orbit through $C$ and $\FO(C^1)$ for the $\G_m^1$-coadjoint orbit through $C^1$.  These are coadjoint orbits for \emph{different} groups.

Let
\begin{align*}
\G_m^\od & := \left\{ b = \1_n + \sum_{j=1}^{m-1} z^j b_j \in \G_m^1\ \bigg| \ b_1,\ldots,b_{m-1} \in \g^{\od} \right\}, \\
\G_m^{1,\od} & := \left\{ b = \1_n + \sum_{j=1}^{m-1} z^j b_j \in \G_m^\od \ \bigg| \ b_{m-1}= 0 \right\}. 
\end{align*}

Observe that $\G_m^{1,\od} \subseteq \G_m^\od$ are subvarieties of $\G_m^1$, though not subgroups, isomorphic to affine spaces of dimensions $(m-2)n(n-1)$ and $(m-1)n(n-1)$, respectively.  

\begin{lemma}\label{codecomp} \begin{enumerate}[(a)]
\item \label{l:G1fact} The restriction of the multiplication map yields isomorphisms
\begin{align*}
\G_m^\od \times \T_m^1 & \arsim \G_m^1 & \G_m^{1,\od} \times \G_m^{1,\rc} & \arsim \G_m^1 .
\end{align*}
In other words, every $b \in \G_m^1$ has unique factorizations
\begin{align*}
b & = b^\od b^\T & b & = b^{1,\od} b^\rc
\end{align*}
with $b^\od \in \G_m^\od$, $b^\T \in \T_m^1$, $b^{1,\od} \in \G_m^{1,\od}$, $b^\rc \in \G_m^{1,\rc}$, and the map taking $b$ to any one of these factors is a morphism.

\item \label{l:C1aff} The morphism $\G_m^{1,\od} \to \O(C^1)$ 
\begin{align*}
b \mapsto \Ad_b C^1
\end{align*}
is an isomorphism.  In particular, $\O(C^1) \cong \mathbb{A}^{(m-2)n(n-1)}$.

\item\label{ztfib} 
There is an isomorphism
\begin{align*} 
\Gamma: \left(G \times \G_m^{\od} \right) \big/\T &\rightarrow \FO(C) & (g,b) &\mapsto \Ad_{gb} C,
\end{align*}
where the action of $\T$ on $\G \times \G_m^{\od}$ is given by $(g_0,b)t_0= (g_0t_0, \Ad_{t_0^{-1}}b)$. In particular, $\G \times \G_m^{\od} \rightarrow (\G \times \G_m^{\od})/\T$ is a Zariski locally trivial principal $\T$-bundle.
\end{enumerate}
\end{lemma}
\begin{proof} 
To prove part (a), consider elements 
\begin{align*}
x & := \1_n + \sum_{j=1}^{m-1} z^j x_j, & y & := \1_n + \sum_{j=1}^{m-1} z^j y_j, & b & := \1_n + \sum_{j=1}^{m-1} z^j b_j \in \G_m^1
\end{align*}
Then the product expression $xy = b$ imposes the relations
\begin{align} \label{e:bxy}
b_j = x_j + y_j + \sum_{\ell=1}^{j-1} x_\ell y_{j-\ell} 
\end{align}
for $1 \leq j \leq m-1$.  Assuming that $b \in \G_m^1$ is arbitrary, the equation $b_1 = x_1 + y_1$ allows us to take $x_1 \in \g^\od$, $y_1 \in \t$ to be the off-diagonal and diagonal parts of $b_1$, respectively.  The relations \eqref{e:bxy} allow us to continue this inductively, so that all $x_j \in \g^\od$ and $y_j \in \t$, noting that at each stage, one has an algebraic expression in terms of the $b_j$.  This gives the first factorization.  The second one is obtained in exactly the same way, except that the conditions are that $x_{m-1} = 0$, but there is no condition on $y_{m-1}$, which makes up for it.

Part \eqref{l:C1aff} follows directly from Lemma \ref{intor}\eqref{l:centralizers} and part \eqref{l:G1fact}.

For part \eqref{ztfib}, note that Lemma \ref{l:coadjorb}\eqref{l:coadjhomog}, together with Lemma \ref{intor}\eqref{l:centralizers}, gives an isomorphism $\phi : \G_m/\T_m \xrightarrow{\sim} \O(C)$. Observe that $\G_m = \G \ltimes \G_m^1$ and hence, as a variety, one has
\begin{align*}
\G_m = \G \times \G_m^1 = \G \times \G_m^\od \times \T_m^1,
\end{align*}
by part \eqref{l:G1fact}. Taking the quotient by $\T_m = \T \times \T_m^1$ then gives the desired isomorphism. One obtains the indicated $\T$-action on $\G \times \G_m^{\od}$ by identifying $ \G \times \G_m^1$ with $\G_m$ via multiplication. The final statement follows since $\G \rightarrow \G/\T$ is a Zariski locally trivial principal $\T$-bundle, as $\T$ is a special group \cite[\S4.3]{serre58}.
\end{proof}

\subsection{Grothendieck rings with exponentials}\label{groexp}

Following \cite{CL} and \cite{CLL}, in this section we introduce Grothendieck rings with exponentials, a naive notion of motivic Fourier transform and convolution. Similar techniques were used in \cite{WY16} to compute the motivic classes of Nakajima quiver varieties.  Throughout this section, $\K$ will denote an arbitrary field, by a variety we mean a separated scheme of finite type over  $\K$,  and by a morphism of varieties we will mean a $\K$-morphism. 

\subsubsection{Definitions}
 The \textit{Grothendieck ring of varieties}, denoted by $\kvar$, is the quotient of the free abelian group generated by isomorphism classes of varieties modulo the relation
 \[ X - Z - U,\] 
for $X$ a variety, $Z\subset X$ a closed subvariety and $U=X\setminus Z$. The multiplication is given by $[X]\cdot [Y] = [X\times Y]$, where we write $[X]$ for the equivalence class of a variety $X$ in $\kvar$. 

The \textit{Grothendieck ring with exponentials} $\kexp$ is defined similarly. Instead of varieties we consider pairs $(X,f)$, where $X$ is a variety and $f:X \rightarrow \BA^1=\spec(\K[T])$ is a morphism. A morphism of pairs $u:(X,f) \rightarrow (Y,g)$ is a morphism $u:X \rightarrow Y$ such that $f = g \circ u$. Then $\kexp$ is defined as the free abelian group generated by isomorphism classes of pairs modulo the following relations.

\begin{enumerate}
\item[(i)] For a variety $X$, a morphism $f:X\rightarrow \BA^1$, a closed subvariety $Z\subset X$ and $U=X\setminus Z$ the relation
\[ (X,f) - (Z,f|_{Z})- (U,f|_{U}).\]
\item[(ii)] For a variety $X$ and $pr_{\BA^1}:X\times \BA^1\rightarrow \BA^1$ the projection onto $\BA^1$ the relation
\[ (X\times \BA^1,pr_{\BA^1}).\]
\end{enumerate} 
The class of $(X,f)$ in $\kexp$ will be denoted by $[X,f]$. We define the product of two generators $[X,f]$ and $[Y,g]$ as 
\[ [X,f]\cdot [X,g] = [X\times Y, f\circ pr_X + g\circ pr_Y],\]
where $f\circ pr_X + g\circ pr_Y: X\times Y \rightarrow \BA^1$ is the morphism sending $(x,y)$ to $f(x)+g(y)$. This gives $\kexp$ the structure of a commutative ring with identity $[\mathrm{pt}, 0]$.

Denote by $\BL$ the class of $\BA^1$ in $\kvar$, resp.~$(\BA^1,0)$ in $\kexp$.  The localizations of $\kvar$ and $\kexp$ with respect to the the multiplicative subset generated by $\BL$ and $\BL^n-1$, where $n\geq 1$,  are denoted by $\mathscr{M}$ and $\expm$.

For a variety $S$ there is a straightforward generalization of the above construction to obtain the \textit{relative Grothendieck rings} $\kvar_S,\kexp_S,\mathscr{M}_S$ and $\expm_S$. For example, generators of $\kexp_S$ are pairs $(X,f)$ where $X$ is an $S$-variety (i.e., a variety with a morphism $X\rightarrow S$) and $f:X \rightarrow \BA^1$ a morphism. The class of $(X,f)$ in $\kexp_S$ will be denoted by $[X,f]_S$ or simply $[X,f]$ if the base variety $S$ is clear from the context. 

The usual Grothendieck ring $\kvar_S$ can be identified with the subring of $\kexp_S$ generated by pairs $(X,0)$, see \cite[Lemma 1.1.3]{CLL}. 

For a morphism of varieties $u: S \rightarrow T$ we have induced maps
\begin{align*} 
u_!&:\kexp_S \rightarrow \kexp_T, & [X,f]_S & \mapsto [X,f]_T\\
u^*&:\kexp_T \rightarrow \kexp_S, & [X,f]_T & \mapsto [X \times_T S, f \circ pr_X]_S. 
\end{align*}
In general, $u^*$ is a morphism of rings and $u_!$ a morphism of additive groups. 

\subsubsection{Realization morphisms} 

The rings $\kvar$ and $\kexp$ and their localizations $\mathscr{M}, \expm$, although easy to define, are quite hard to understand concretely. To circumvent this, one usually considers realization morphisms to simpler rings. 

If $\K= \BF_q$ is a finite field there is a ring homomorphism $\kvar \to \mathbb{Z}$ sending the class of a variety $X/\BF_q$ to the number of $\BF_q$-rational point of $X$. More generally for $S$ a variety over $\BF_q$ we can construct a realization from $\kexp_S$ to the ring $\mathrm{Map}(S(\BF_q),\BC)$ of complex valued functions on $S(\BF_q)$ by sending the class of $[X,f]$ to the function
\begin{equation}\label{ffreal} s \in S(\BF_q) \mapsto \sum_{x \in X_s(\BF_q)} \Psi\big(f(x)\big),\end{equation}
where $\Psi:\BF_q \rightarrow \BC^\times$ is a fixed non-trivial additive character and $X_s$ the fiber over $s$. Under this realization, for a morphism $u: S \rightarrow T$, the operations $u_!$ and $u^*$ correspond to summation over the fibers of $u$ and composition with $u$, respectively. 

If $\K$ is a field of characteristic $0$, whose transcendence degree over $\mathbb{Q}$ is at most the one of $\BC/\mathbb{Q}$, we can embed $\K$ into $\BC$ and consider any variety over $\K$ as a variety over $\BC$. By the work of Deligne \cite{De71, De74} the compactly supported cohomology $H^*_c(X,\BC)$ of any complex algebraic variety $X$ carries two natural filtrations, the weight and the Hodge filtration. Taking the dimensions of the graded pieces we obtain the compactly supported mixed Hodge numbers of $X$
\[ h_c^{p,q;i}(X) = \dim_\BC \left(Gr^H_p Gr^W_{p+q} H^i_c(X,\BC)\right). \]

From these numbers we define the \textit{E-polynomial} as
\begin{equation}\label{epd} E(X;x,y) = \sum_{p,q,i \geq 0} (-1)^i  h_c^{p,q;i}(X) x^py^q.\end{equation}
This way we obtain a morphism (see for example \cite[Appendix]{hausel-villegas})
\begin{align}\label{reale} 
\kvar & \rightarrow \BZ[x,y] & [X] &\mapsto E(X;x,y).
\end{align}
It is not hard to see that $E(\BL;x,y) = xy$ and thus this realization extends to a morphism 
\[\mathscr{M} \to \BZ[x,y]\left[ \frac{1}{xy}; \frac{1}{(1-(xy)^n)}, n\geq 1\right].\]
These two realizations of $\kvar$ are related by the following theorem of Katz. We refer to \cite[Appendix]{hausel-villegas} for the precise definition of a \textit{strongly polynomial count} variety over $\BC$.

\begin{thm}\cite[Theorem 6.1.2(3)]{hausel-villegas} \label{Katzthm} If $X$ over $\BC$ is strongly polynomial count with counting polynomial $P_X(t) \in \mathbb{Z}[t]$, then
\[ E(X;x,y) = P_X(xy).\]
\end{thm}

In particular in the situation of Theorem \ref{Katzthm}, the $E$-polynomial of $X$ is a polynomial in one variable, which we also call the \textit{weight polynomial}
\[ E(X;q) = E(X;q^{\frac{1}{2}},q^{\frac{1}{2}}).\]

\subsubsection{Computational tools}

We now introduce several tools for our computations in $\kvar$ and $\kexp$, which are mostly inspired by similar constructions over finite fields through the realization \eqref{ffreal}.

\paragraph{Fibrations.} We say that $f:X \rightarrow Y$ is a \textit{Zariski locally-trivial fibration} if $Y$ admits an open covering $Y=\cup_{j} U_j$ such that $f^{-1}(U_j) \cong F\times U_j$ for some fixed variety $F$. In this case we have the product formula
\begin{eqnarray}\label{trif} [X]=[F][Y] \end{eqnarray}
in $\kvar$ as we can compute directly
\[[X] = \sum_{j} [f^{-1}(U_j)] - \sum_{j_1 < j_2} [f^{-1}(U_{j_1} \cap U_{j_2})] +... = [F][Y].\]

\paragraph{Character sums.}
When computing character sums over finite fields, one has the following crucial identity 
\begin{eqnarray*} 
\sum_{v\in V} \Psi\big( f(v) \big) = \begin{cases} q^{\dim V} &\text{ if } f=0\\
0 &\text{ otherwise,}\end{cases}\end{eqnarray*}
where $V$ is a finite-dimensional vector space over $\BF_q$ and $f\in V^\vee$ a linear form.  To establish an analogous identity in the motivic setting, we let $V$ be a finite-dimensional vector space over $\K$ and $S$ a variety. We replace the linear form above with a family of affine linear forms, i.e., a morphism $g=(g_1,g_2) : X \rightarrow V^\vee\times \K$, where $X$ is an $S$-variety. Then we define $f$ to be the morphism
\begin{align*} 
f & :X\times V \rightarrow \K & (x,v) & \mapsto \left\langle g_1(x), v \right\rangle + g_2(x).
\end{align*}
Finally we put $Z= g_1^{-1}(0)$. 

\begin{lemma}\cite[Lemma 2.1]{WY16}\label{orth} With the notation above we have the relation
\[ [X\times V,f] = \BL^{\dim V}[Z,{g_2}|_{Z}] \]
in $\kexp_S$. In particular, if $X = \spec \K$ and $f \in V^\vee$, we have $[V,f] = 0$ unless $f = 0$.
\end{lemma}

\paragraph{Fourier transforms.} We now define the \textit{naive motivic Fourier transform} for functions on a finite-dimensional $\K$-vector space $V$ and the relevant inversion formula. All of this is a special case of \cite[Section 7.1]{CL}.

\begin{defn}\label{nft} Let $p_V:V\times V^\vee \rightarrow V$ and $p_{V^\vee}:V\times V^\vee \rightarrow V^\vee$ be the obvious projections. 
\textit{The naive Fourier transformation} $\mathcal{F}_V$ is defined as 
\begin{align*} \mathcal{F}_V& :\kexp_{V} \rightarrow \kexp_{V^\vee} &
\phi & \mapsto p_{V^\vee!} \left( p_V^*\phi \cdot[V\times V^\vee,\lla,\rra] \right). \end{align*}
Here $\left\langle \, , \right\rangle:V\times V^\vee \rightarrow \K$ denotes the natural pairing.  We will often write $\mathcal{F}$ instead of $\mathcal{F}_V$ when no ambiguity will arise from doing so.
\end{defn}

Of course, the definition is again inspired by the finite field version, where one defines for any function $\phi: V \rightarrow \BC$ the Fourier transform at $w \in V^\vee$ by
\[ \FF(\phi)(w) = \sum_{v \in V } \phi(v) \Psi(\lla w,v\rra).\]

Notice that $\mathcal{F}$ is a homomorphism of groups and thus it is worth spelling out the definition in the case when $\phi = [X,f]$ is the class of a generator in $\kexp_V$. Letting $u:X\rightarrow V$ be the structure morphism we simply have 
\begin{eqnarray}\label{ftgen}
\mathcal{F}([X,f]) = [X \times V^\vee, f\circ pr_X+\left\langle u\circ pr_X,pr_{V^\vee}\right\rangle].
\end{eqnarray}

We have the following version of Fourier inversion.
\begin{prop}\cite[Proposition 2.2]{WY16}\label{finv} For every $\phi \in \kexp_{V}$  we have the identity
\[\mathcal{F}(\mathcal{F}(\phi)) = \BL^{\dim V} \cdot i^*(\phi),\]
where $i:V\rightarrow V$ is multiplication by $-1$.
\end{prop}

\paragraph{Convolution.} Finally, we introduce a motivic version of convolution. 

\begin{defn} Let $R:\kexp_V\times \kexp_V \rightarrow \kexp_{V\times V}$ be the natural morphism sending two varieties over $V$ to their product, and $s: V \times V \rightarrow V$ the sum operation. The \textit{convolution product} is the associative and commutative operation
\begin{align*} 
\ast : & \kexp_V\times \kexp_V \rightarrow \kexp_V & (\phi_1,\phi_2) &\mapsto \phi_1 \ast \phi_2 = s_! R(\phi_1,\phi_2).
\end{align*}
\end{defn}
As expected the Fourier transform interchanges product and convolution product.
\begin{prop} \label{fouconv} For $\phi_1,\phi_2 \in \kexp_V$ we have
\[ \FF(\phi_1 \ast \phi_2) = \FF(\phi_1)\FF(\phi_2).\]
\end{prop}
\begin{proof} As both $\FF$ and $\ast$ are bilinear, it is enough prove the identity for two generators $[X,f],[Y,g] \in \kexp_V$ with respective structure morphisms $u:X\rightarrow V$, $v:Y \rightarrow V$. Using \eqref{ftgen} we can then directly compute
\begin{align*} 
\FF \left( [X,f] \ast [Y,g] \right) &= \FF(s_![X\times Y,f\circ pr_X + g\circ pr_Y])\\
&= [X\times Y \times V^\vee, f\circ pr_X + g\circ pr_Y + \lla s \circ (u \times v) \circ pr_{X\times Y}, pr_{V^\vee}\rra ]
\end{align*}
On the other hand using the natural isomorphism 
\[ (X\times V^\vee) \times_{V^\vee} (Y \times V^\vee) \cong X \times Y \times V^\vee, \]
we get
\begin{align*} 
\FF[X,f]\FF[Y,g] &= [X \times V^\vee,f\circ pr_X + \lla u\circ pr_X,pr_{V^\vee}\rra][Y \times V^\vee, g\circ pr_Y+ \lla v\circ pr_Y,pr_{V^\vee}\rra] \\
&=  [X\times Y \times V^\vee, f\circ pr_X + g\circ pr_Y + \lla s \circ (u \times v) \circ pr_{X\times Y}, pr_{V^\vee}\rra ],
\end{align*}
and thus $\FF \left( [X,f] \ast [Y,g] \right) = \FF[X,f]\FF[Y,g]$.
\end{proof}

\begin{rem}
We will use the convolution product to study equations in a product of varieties, i.e., consider $V$-varieties $u_i:X_i \rightarrow V$ say for $i=1,2$. Then it follows from the definition of $*$ that for any $v: \spec \K \rightarrow V$ the class of $\{ (x_1,x_2) \in X_1 \times X_2 \ |\ u_1(x_1) +u_2(x_2) = v\}$ is given by $v^*([X_1] *[X_2])$. Proposition \ref{fouconv} allows us to compute the latter by understanding the Fourier transforms $\FF(X_1)$  and $\FF(X_2)$ separately.
\end{rem}

\section{Open de Rham spaces} \label{s:odr} 

In this section we define open de Rham spaces as an additive fusion of coadjoint orbits, similar as in \cite[\S2]{HiroeYamakawa}. They were first introduced in \cite[Section 2]{Bo01} as certain moduli spaces of connections on $\BP^1$. We recall this viewpoint briefly in Section \ref{modcon}.

\subsection{Additive fusion products of coadjoint orbits} \label{s:addfusprod}

As before, let $\K$ be a fixed base field. Fix a diagonal element $C \in \g_m^\vee$ as in \eqref{e:diagelt} and consider its $\G_m$-coadjoint orbit $\O(C) \subseteq \g_m^\vee$.  In the case that $\K = \R$ or $\C$, $\G_m$ is a real or complex Lie group, accordingly, and the coadjoint orbit $\O(C)$ admits a canonical symplectic form (the Kirillov--Kostant--Souriau form) for which the coadjoint action is Hamiltonian with moment map given by the inclusion $\mu_{\O(C)} : \O(C) \hookrightarrow \g_m^\vee$ \cite[\S II.3.3.5, \S II.3.3.8]{Audin}.  We may restrict the action on $\O(C)$ to that of the constant subgroup $\G \subset \G_m$, using \eqref{e:Gsemdir}.  The moment map $\mu_{\res} = \pi_{\res} \circ \mu_{\O(C)} : \O(C) \to z^{-1} \g = \g^\vee$ for the restricted action is the composition of the inclusion and the projection $\pi_{\res} : \g_m^\vee \to \g^\vee$ \eqref{pis}, so simply takes the residue term
\begin{align} \label{e:muresMM}
\frac{Y_m}{z^m} + \frac{Y_{m-1}}{z^{m-1}} + \cdots + \frac{Y_1}{z} \mapsto \frac{Y_1}{z}.
\end{align}

Now, for $d \in \Z_{>0}$, let $m_i \in \Z_{>0}$ for $1 \leq i \leq d$, and let $C^i \in \g_{m_i}^\vee$ be diagonal elements and $\O(C^i)$ their coadjoint orbits.  We may form the product $\prod_{i=1}^d \O(C^i)$ which has the product symplectic structure.  It also carries a diagonal $\G$-action for which the moment map is

\begin{equation}\label{e:mud} \mu : \prod_{i=1}^d \O(C^i) \to \g^\vee, \ \ \ \ (Y^1, \ldots, Y^d) \mapsto \sum_{i=1}^d \frac{Y_1^i}{z} \end{equation}

In this case where $\K = \R$ or $\C$, we can then form the symplectic (Marsden--Weinstein) quotient in the usual way.  We now observe that if $\K$ is any field, then both \eqref{e:muresMM} and \eqref{e:mud} are defined over $\K$.  

\begin{defn} \label{d:odr} Let $\mathbf{C} = (C^1,\dots,C^d)$ be a tuple of diagonal elements satisfying \eqref{e:cassum}, so that the product $\prod_{i=1}^d \O(C^i)$ is an affine variety (Lemma \ref{l:coadjorb}\eqref{l:coadjaff}).  The \textit{open de Rham space} $\FM^*(\mathbf{C})$ is the affine GIT quotient
\[   \FM^*(\mathbf{C}) = \prod_{i=1}^d \O(C^i) \bigg/\!\!\!\!\bigg/_0 \G := \spec \left( \K[\mu^{-1}(0)]^{\G} \right). \]
\end{defn}

Except possibly in Section \ref{s:qmv} we will typically impose two more conditions on $\mathbf{C}$, regularity and genericity.

\begin{defn}  A diagonal element $C \in \g^\vee_m$ of order $m\geq 2$ is \emph{regular} if $C_{m} \in \ft^{\reg}$ i.e. has distinct eigenvalues. A $d$-tuple $\mathbf{C}=(C^1,\dots,C^d)$ is called \emph{regular} if for all  $1\leq i \leq d$ with $m_i \geq 2$, $C^i$ is regular.
\end{defn}

We now define genericity of the tuple $\mathbf{C}$ in a similar manner as in  \cite[2.2.1]{HLV11}. Our more explicit formulation will be used in later computations, see the proof of Theorem \ref{mainodr} and Lemma \ref{finalc}.  Define for $I \subset \{1,2,\dots,n\}$ the matrix $E_I \in \g$ by
 \begin{eqnarray}\label{eij} (E_I)_{ij} = \begin{cases} 1 & \text{ if } i=j\in I \\ 0 & \text{ otherwise.}	\end{cases} \end{eqnarray}

\begin{defn}\label{genericc}  We call $\mathbf{C}$ \emph{generic} if $\sum_{i=1}^d \tr C_1^i = 0$ and for every integer $n'<n$ and subsets $I_1,\dots,I_d \subset \{1,\dots,n\}$ of size $n'$ we have 
\begin{eqnarray}\label{gen}\sum_{i=1}^d \lla C_1^i,E_{I_i}\rra \neq 0,\end{eqnarray}
where $\lla,\rra$ is the pairing defined in \eqref{trp}.
In other words, if $\mathbf{C}$ is generic there are no non-trivial subspaces $V_1,\dots,V_d \subsetneq \K^n$ of the same dimension such that $V_i$ is invariant under $C_i$ for $1 \leq i\leq d$ and $\sum_i \tr C^i_1|_{V_i} = 0$.
\end{defn}

Now, let $\mathbf{C}$ be a regular generic tuple of formal types.  We adapt the following notation so as to later (Section \ref{s:wcvcomp}) make comparisons with \cite{HMW16}.  First, we order $C^1,\dots,C^d$ in a way such that $m_1=\dots = m_{k} =1$ and $m_{k+1},\dots,m_{k+s} \geq 2$, with $k+s = d$. We write $\muhat = (\mu^1,\dots,\mu^k) \in \FP_n^k$ for the $k$-tuple of partitions of $n$ defined by the multiplicities of the eigenvalues of $C^i_1$ for $1\leq i \leq k$. 

When one is talking about a meromorphic connection having a formal type of order $m$ at a pole with a semisimple leading order term (of course, here we have only discussed such types of poles), then it is standard terminology to refer to the number $m-1$ as the \emph{Poincar\'e rank} of the pole.  For our moduli spaces, since $C^{k+i}_{m_{k+i}} \in \ft^{\reg}$ for $1 \leq i\leq s$, we will write $r_i := m_{k+i}-1$ for the Poincar\'e rank of $C^i$ and record these in the $s$-tuple $\bfr = (r_1,\dots, r_s)$.  Finally, we write $r= \sum_{i=1}^s r_i$ and call this the \emph{total Poincar\'e rank}.
	
\begin{defn}\label{genodr}
For a regular generic $\mathbf{C}$ we write $\FM^*_{\muhat,\bfr}$ instead of $\FM^*(\mathbf{C})$ and refer to it as the \textit{generic open de Rham space} of type $(\muhat,\bfr)$. If, furthermore, $k=0$ we write $\FM^*_{n,\bfr}$ for $\FM^*(\mathbf{C})$.
\end{defn}

\begin{rmk}\label{larchar}
This notation is justified since the invariants we compute in Sections \ref{s:motcls} and \ref{cac} will depend only on $(\muhat,\bfr)$ and not on the actual eigenvalues of the formal types. 

We will always assume $s \geq 1$ in this paper, in which case a generic $\mathbf{C}$ always exists if $\K$ is algebraically closed. This follows from \cite[Lemma 2.2.2]{HLV11}, because in our case when $s \geq 1$ in loc. cit. both $D=d=1$. If $\K= \BF_q$ is a finite field one needs an additional lower bound on $q$ depending on $n$ and $d$ to make sure that the Zariski-open subvariety of $\BA^{nd}$ defined by \eqref{gen} has an $\BF_q$-rational point. We will not spell out an explicit bound here, as we are only interested in sufficiently large $q$.
\end{rmk}	

\begin{prop}\label{odrspa}  If non-empty, $\FM^*_{\muhat,\bfr}$ is smooth and equidimensional of dimension 
\begin{align} \label{e:dimform}
d_{\muhat,\bfr} = dn^2 - sn +r(n^2-n) - \sum_{i=1}^k N(\mu^i) - 2(n^2 - 1),
\end{align}
where for $\mu= (\mu_1\geq \dots \geq \mu_l) \in \FP_n$ we define $N(\mu) = \sum_{j=1}^l \mu_j^2$.
\end{prop}

\begin{proof} We take $\mu$ as in \eqref{e:mud}.  Clearly the scalars $\K^\times \hookrightarrow \G$ act trivially on $\mu^{-1}(0)$, hence the $\G$-action factors though $\PGL_n = G/\K^\times$. We show first that this $\PGL_n$-action is free on $\mu^{-1}(0)$.

Let $(A^1,\dots,A^d) \in \mu^{-1}(0)$ and $g\in \G$ such that $\Ad_g A^i=A^i$ for $1 \leq i \leq d$. We show now that $g$ is scalar, by looking at some non-zero eigenspace $V$ of $g$. Then clearly $A^i_1$ will preserve $V$ for all $i$ and by the moment map condition $\sum_i A^i_1 = 0$ we deduce $\sum_i \tr A^i_1|_V = 0$. The point is then that for each $i$ there is a subspace $V'_i$ of the same dimension as $V$ such that 
\begin{equation}\label{painf}\tr A^i_1|_V = \tr C^i|_{V'_i}.\end{equation}
By the genericity of $\mathbf{C}$ (Definition \ref{genericc}), this implies then $V = V'_i = \K^n$ and hence $g$ is scalar.

For $1 \leq i \leq k$, \eqref{painf} follows simply because $A^i$ and $C^i$ are conjugate in $\G$. To prove (\ref{painf}) for $k+1\leq i \leq d$, we write $A^i =\Ad_h C^{i}$ for some $h \in \G_{m_i}$. By conjugating $A^i$ and $g$ with the constant term $h_0$ of $h$ we can assume without loss of generality $h \in \G_{m_i}^1$, i.e., $h_0=\unt$. Then $A^i_{m_i} = C^{i}_{m_i} \in \ft^{\reg}$ and thus $g\in \T$. Next consider $\bar{h} = hgh^{-1}$, which satisfies $\Ad_{\bar{h}} C^i = C^i$. By Lemma \ref{intor}\eqref{secnd} we have $\bar{h} \in \T_{m_i}$ and then by Lemma \ref{intor}\eqref{intorr} $hgh^{-1} = g$. This implies that $h_j$ preserves $V$ for every $0\leq j\leq m_i-1$ and hence we have $\tr A|_{V} = \tr (\Ad_{h}C^i)|_{V} = \tr \Ad_{h|_{V}}C^i|_{V} = \tr C^i|_{V}$. This proves \eqref{painf} and hence $\PGL_n$ acts freely on $\mu^{-1}(0)$.

In particular, all the $\G$-orbits in $\mu^{-1}(0)$ are closed and hence they are in bijection with the points of the GIT quotient \cite[Theorem 6.1]{Do03}. Furthermore, as $\mu$ is a moment map, freeness of the $\PGL_n$ action implies that $0$ is a regular value of $\mu$, which in turn implies smoothness of $\mu^{-1}(0)$ and hence of $\FM^*(\mathbf{C})$.  Looking at tangent spaces, we see that 
\begin{align*}
\dim \FM^*(\mathbf{C}) = \dim \prod_{i=1}^d \FO(C^i) - 2 \dim \PGL_n.
\end{align*}
The formula now follows since $\dim \FO(C^i) = n^2 - N(\mu^i)$ for $1\leq i \leq k$ and $\dim \FO(C^{s+i}) = (r_i+1)(n^2-n)$ for $1\leq i \leq s$.
\end{proof}

We will see in Corollary \ref{nonemp} that $\FM^*_{\muhat,\bfr}$ is non-empty and connected if $d_{\muhat,\bfr} \geq 0$. For more general (i.e. non-generic) $\mathbf{C}$ the non-emptyness of $\FM^*(\mathbf{C})$ has been determined in \cite[Theorem 0.3]{Hiroe}.

\subsection{Moduli of connections}\label{modcon}

We now work over the field $\K = \C$, otherwise adopting the notation of Section \ref{s:jd} for groups and Lie algebras.  Let $X$ be a Riemann surface and fix a point $x \in X$. Let $\O$ and $\Omega$ be the sheaves of holomorphic functions and differentials on $X$ respectively and $\widehat{\O}_x, \OM_x$ the completions of their stalks at $x$.

If $m \in \mathbb{Z}_{>0}$, we will let $\Omega(m \cdot x)$ denote the sheaf of meromorphic differentials with a pole of order $\leq m$ at $x$; we let $\Omega(\ast x) = \bigcup_{m \in \mathbb{Z}_{>0}} \Omega(m \cdot x)$ be their union, i.e., the sheaf of meromorphic differentials with a pole at $x$ of arbitrary order. Finally we write  $\OM(m \cdot x)_x$, $\OM(\ast x)_x$ for their respective completions. If $z$ is a choice of coordinate centered at $x$ one has isomorphisms

\begin{align} \label{e:localisos}
\OM_x & \arsim \C[\![ z ]\!] \cdot dz & \OM(m \cdot x)_x & \arsim z^{-m} \C[\![ z ]\!] \cdot dz & \OM(\ast x)_x & \arsim \C(\!( z )\!) \cdot dz.
\end{align}

Consider the space $(\OM(\ast x)_x / \OM_x) \otimes_\C \t$.  It is clear that one has a well-defined notion of the order of the pole of such an element.  Further, under the isomorphisms \eqref{e:localisos}, a given $C \in (\OM(\ast x)_x / \OM_x) \otimes_\C \t$ with a pole of order $m$ has a unique representative in $z^{-1} \t[z^{-1}] \subseteq \t(\!(z)\!)$ of the form
\begin{align} \label{e:formaltype}
\left( \frac{C_m}{z^m} + \cdots + \frac{C_1}{z} \right) dz
\end{align}
with $C_i \in \t$.

\begin{defn} \label{d:formaltype}
A \emph{formal type of order $m$ at $x$} is an element $C \in (\OM(\ast x)_x / \OM_x) \otimes_\C \t$ with a pole of order $m$.  We will call such a formal type $C$ of order $m \geq 2$ \emph{regular} if, upon some choice of coordinate $z$ at $x$, in the expression \eqref{e:formaltype}, one has $C_m \in \t^\reg$.
\end{defn}

\begin{defn}
Let $V$ be a holomorphic vector bundle over $X$ and $\nabla$ a meromorphic connection with a pole (only) at $x$.  Choose any holomorphic trivialization of $V$ in a neighbourhood of $x$ and let $A$ be the connection matrix of $\nabla$ with respect to this trivialization; if $\nabla$ has a pole of order $m$ at $x$, then $A$ yields an element of $\Omega(m \cdot x)_x \otimes \g$.  Let $C$ be a formal type at $x$.  We say that $(V, \nabla)$ \emph{has formal type $C$ at $x$} if there exists a formal gauge transformation $g \in G(\widehat{\O}_x)$ such that the class of $\Ad_g A - dg \cdot g^{-1} \in (\OM( \ast x)_x / \OM_x) \otimes \g$ agrees with that of $C$ under the inclusion
\begin{align*}
\left( \OM( \ast x)_x / \OM_x \right) \otimes \t \hookrightarrow \left(\OM( \ast x)_x / \OM_x \right) \otimes \g.
\end{align*}
\end{defn}

With these definitions, the open de Rham spaces admit the following moduli description.  Set $X = \P^1$ with an effective divisor $D = m_1a_1+m_2a_2 + \cdots + m_da_d$ and for $1 \leq i\leq d$, a formal type $C^i$ of order $m_i$ at $a_i$. By 'forgetting' $dz$ in \eqref{e:formaltype} we obtain a tuple $\mathbf{C}=(C^1, \ldots, C^d)$ of diagonal elements in $\g^\vee$ (depending on a fixed coordinate on $\BP^1$).

\begin{prop} \cite[Proposition 2.1]{Bo01}, \cite[Proposition 2.7, Corollaries 2.14, 2.15]{HiroeYamakawa} For $\mathbf{C}$ regular and generic, the open de Rham space $\FM^*_{\muhat,\bfr}$ is isomorphic to the moduli space of  meromorphic connections $\nabla$ on the trivial bundle of rank $n$ on $\BP^1$, where $\nabla$ has poles bounded by $D$ and prescribed formal type $C^i$ at $a_i$.
\end{prop}

\begin{rmk}
More generally, fixing the polar divisor $D$ and the local parameters $\mathbf{C}$, one may construct a moduli space $\MM(\mathbf{C})$ of meromorphic connections $(V, \nabla)$ on $\P^1$, where $V$ is a degree $0$ vector bundle (though not necessarily trivial) and $\nabla$ has formal type $C^i$ at $a_i$.  An analytic construction is given by \cite[Proposition 4.5]{Bo01}, which is generalized to higher genus curves in \cite{BiquardBoalch}.  Algebraic constructions of these moduli spaces are given in \cite{InabaSaito2013}.

After \cite{Si94}, such moduli spaces are typically referred to as "de Rham moduli spaces".  Forgetting the connection, these spaces yield, a fortiori, families of vector bundles of degree $0$.  As the only semi-stable bundle of degree $0$ on $\P^1$ is the trivial bundle, and since the semi-stable locus of a family is always Zariski open, the moduli spaces $\MM^*(\mathbf{C})$ sit as open subvarieties
\begin{align*}
\MM^*(\mathbf{C}) & \subseteq \MM(\mathbf{C}). 
\end{align*} 
It is for this reason that we call them ``open'' de Rham spaces.
\end{rmk}

\section{Motivic classes of open de Rham spaces} \label{s:motcls}

Throughout this section $\K$ will always be an algebraically closed field of $0$ or odd characteristic.

The main result in this section is Theorem \ref{mainodr}, a formula for the motivic class of $\FM_{n,\bfr}$ in the localized Grothendieck ring $\mathscr{M}$ for any $d$-tuple $\bfr = (r_1,\dots,r_d) \in \BN^d$. By definition $\FM_{n,\bfr}$ is an additive fusion of coadjoint orbits and thus the computation of $[\FM_{n,\bfr}]$ can be split up in two parts.

First we determine in Theorem \ref{ftpuncturet} the motivic Fourier transform of the composition 
\[\FO(C) \hookrightarrow \g^\vee_m \rightarrow \g^\vee,\]
under the assumption $m \geq 2$. In particular we prove that $\FF(\FO(C)) \in \kexp_{\g}$ is supported on semisimple conjugacy classes, which is not true for $m=1$. The motivic convolution formalism then allows us to deduce the formula for $[\FM_{n,\bfr}]$ from these local computations.

\subsection{Some notation for partitions} \label{s:partnote}

For $n \in \Z_{> 0}$, we denote by $\FP_n$ the set of partitions of $n$. For $\lambda = (\lambda_1 \geq \lambda_2 \geq \dots \geq \lambda_l) \in \FP_n$ we use the following abbreviations
\begin{align*}
\begin{split}
l(\lambda) &:= l \\
N(\lambda) &:= \sum_{i=1}^l \lambda_i^2\\
\end{split} \quad \begin{split}
\lambda! &:= \prod_{i=1}^l \lambda_i!\\
\binom{n}{\lambda} &:= \frac{n!}{\prod_{i=1}^l \lambda_i!}.\\
\end{split} 
\end{align*}
Further we write $m_k(\lambda)$ for the multiplicity of $k \in \Z_{>0}$ in $\lambda$ and 
\[u_\lambda := \prod_{k \in \Z_{>0} } m_k(\lambda)!.\]
The polynomial $\phi_\lambda \in \Z[t]$ is defined as
\[\phi_\lambda(t) := \prod_{i=1}^l (1-t)(1-t^2) \cdots (1-t^{\lambda_i}).\]

Then if $\L_\lambda \cong \prod_{i=1}^l \GL_{\lambda_i}$ denotes the subgroup of block diagonal matrices in $\GL_n$ we have
\[ [\L_\lambda] = (-1)^n \BL^{\frac{N(\lambda)-n}{2}} \phi_\lambda(\BL) \in \kvar,\]

since for any $n \in \Z_{> 0}$ we have for example by \cite[Proposition 1.1]{EK09}
\[ [\GL_n] = (\BL^n-1)(\BL^n-\BL) \cdots(\BL^n-\BL^{n-1}). \]

\subsection{Fourier transform at a pole} \label{ss:ftp}
In this section we compute the Fourier transform $\FF(\FO(C)) \in \kexp_\g$ of a coadjoint orbit  $\pi_{\res}:\FO(C) \rightarrow \g^\vee$ of an element $C \in \t_m^\vee$, say in the notation of \eqref{pis}, where we use the language of Section \ref{groexp}. Assuming $m\geq 2$, we can give an explicit formula for $\FF([\FO(C)])$, but to do so we will first need to introduce some more notation.

A semisimple element $X \in \g$ has \textit{type} $\lambda = (\lambda_1\geq \dots \geq\lambda_l) \in \FP_n$ if $X$ has $l$ distinct eigenvalues with multiplicities $\lambda_1,\dots,\lambda_l$. We write $\g_\lambda := \{ X \in \g \ |\ X \text{ has type } \lambda \}$ and $i_\lambda: \g_\lambda \hookrightarrow \g$.

\begin{thm}\label{ftpuncturet} Let $C \in \g^\vee_m$ be a regular diagonal element of order $m\geq 2$. For any partition $\lambda \in \FP_n$, in $\expm_{\g^\lambda}$ we have the formula
\begin{equation}\label{ftpuncture} i_\lambda^*\FF([\FO(C)]) =\frac{ \BL^{n+\frac{1}{2}\left(m(n^2-2n) + (m-2)N(\lambda)\right)}}{ (\BL-1)^{n}}\left[\ZZ_\lambda, \phi^C\right], \end{equation}
where $\ZZ_\lambda := \{ (g,X) \in \G \times \g_\lambda \ |\ \Ad_{g^{-1}}X \in \ft\}$ and 
\begin{align*} 
\phi^C: & \ZZ_\lambda \to \BA^1 & (g,X) & \mapsto \lla C_1,\Ad_{g^{-1}} X \rra.
\end{align*}
Furthermore, the pullback of $\FF([\FO(C)])$ to the complement $\g \setminus \bigsqcup_\lambda \g_\lambda$ equals $0$. 
\end{thm}

\begin{proof}
 By the formula  (\ref{ftgen}) we have
\[ \FF([\FO(C)]) = [\FO(C) \times \g, \lla \pi_{\res} \circ pr_{\FO(C)} , pr_{\g} \rra] = [\FO(C) \times \g, \lla pr_{\FO(C)} , pr_{\g} \rra],\]
where for the second equality sign we used the definition of $\lla,\rra$, see (\ref{trp}).  By Lemma \ref{codecomp}\eqref{ztfib} the natural map $\G \times \G_m^{\od} \to \FO(C)$ is a Zariski-locally trivial $T$-bundle and thus we can rewrite this in $\expm_\g$ as
\[\FF([\FO(C)]) = (\BL-1)^{-n} \left[\G \times \G_m^{\od} \times \g, \lla \Gamma \circ pr_{\G \times \G_m^{\od}}, pr_{\g} \rra \right].\]

Now, notice that for all $(g,b,X) \in \G \times \G_m^{\od} \times \g$ we have 
\[\lla \Gamma(g,b),X \rra = \lla \Ad_{gb}C, X \rra =  \lla \Ad_b C, \Ad_{g^{-1}} X\rra.\]
Thus  we finally obtain
\begin{equation}\label{blabla} \FF([\FO(C)]) = (\BL-1)^{-n} \left[\G \times \G_m^{\od} \times \g, \lla \Ad_{pr_{\G_m^{\od}}}C, \Ad_{(pr_{\G})^{-1}}pr_{\g} \rra \right].\end{equation}
 
We will simplify this by applying Lemma \ref{orth}. Consider the decomposition $\G_m^{\od} = \G_+^{\od} \oplus \G_-^{\od}$ with
\begin{align*}
\G_+^{\od} & =  \{ b \in \G_m^{\od} \ |\ b_{\lfloor\frac{m+1}{2}\rfloor} = \dots = b_{m-1} = 0 \} & \G_-^{\od} & =  \{ b \in \G_m^{\od} \ |\ b_1 = \dots = b_{\lfloor\frac{m-1}{2}\rfloor} = 0 \}. 
\end{align*}

It follows from Lemma \ref{keycomp} below that there are functions 
\[(h_1,h_2): \G \times \g \times \G_+^{\od} \rightarrow (\G_-^{\od})^\vee \times \K\]
such that $\lla \Ad_b C, \Ad_{g^{-1}} X \rra = \lla h_1(g,X,b_+), b_- \rra + h_2(g,X,b_+)$ for all $g \in \G, X \in \g$ and $b = (b_+,b_-) \in \G_+^{\od} \oplus \G_-^{\od}$. More explicitly $h_1$ and $h_2$ are given by
\begin{align*} 
h_2(g,X,b_+) &= \lla \Ad_{b_+} C, \Ad_{g^{-1}} X \rra, &
\lla h_1(g,X,b_+), b_- \rra &= \lla \Ad_b C- \Ad_{b_+} C, \Ad_{g^{-1}} X\rra.
\end{align*}

Applying Lemma \ref{orth} to this decomposition formula (\ref{blabla}) becomes
\begin{equation}\label{somecanc} \FF([\FO(C)]) = (\BL-1)^{-n}\BL^{\dim \G_-^{\od}} [ h_1^{-1}(0), h_2].\end{equation}

Now assume first $m$ is even. In this case $\lfloor \frac{m-1}{2}\rfloor =\lfloor \frac{m-2}{2}\rfloor$ and Lemma \ref{keycomp} implies
 \[h_1^{-1}(0) = \{ (g,X,b_+) \in \G \times \g \times \G_+^{\od} \ |\ \Ad_{g^{-1}}X \in \ft, [b_+,\Ad_{g^{-1}}X]=0\},\]
and $h_{2|h_1^{-1}(0)}$ is independent of $b_+$. For any $\lambda \in \FP_n$ the pullback $i_\lambda^*h_1^{-1}(0) \rightarrow \ZZ_\lambda$ is the kernel of the vector bundle endomorphism $b_+ \mapsto [b_+,\Ad_{g^{-1}}X]$ on $\ZZ_\lambda \times \G_+^{\od}$. The rank of the kernel is constant and equals $\frac{m-2}{2}(N(\lambda) - n)$, thus we finally get
\begin{equation}\label{someform} i_\lambda^*[ h_1^{-1}(0), h_2] = \BL^{\frac{m-2}{2}\left(N(\lambda) - n\right)} [\ZZ_\lambda,\phi^C].\end{equation}
Together with $\dim \G_-^{\od} = \frac{m}{2}\left(n^2-n\right)$ the theorem follows when $m$ is even.

If $m$ is odd we have 
 \[h_1^{-1}(0) = \{ (g,X,b_+) \in \G \times \g \times \G_+^{\od} \ |\ \Ad_{g^{-1}}X \in \ft, [b_+,\Ad_{g^{-1}}X]= [b_{\frac{m-1}{2}},\Ad_{g^{-1}}X]\}.\]

If we decompose $b_{\frac{m-1}{2}} = b^l + b^u$ into strictly lower and upper diagonal parts, then Lemma \ref{keycomp} implies that $h_{2|h_1^{-1}(0)}$ is affine linear in $b^u$ and independent of $b^u$ if and only if $b^l$ commutes with $\Ad_{g^{-1}}X$. As before we can now apply Lemma \ref{orth} and see that \eqref{someform} also holds for $m$ odd, which finishes the proof of \eqref{ftpuncture}.

Finally we see from \eqref{somecanc} and the description of $h^{-1}(0) \subset  \G \times \g \times \G_+^{\od} $ in both the even and odd case, that the structure morphism $ h^{-1}(0) \to \g$, which is the projection onto $\g$, has image in the semisimple elements of $\g$. Thus the pullback of $\FF([\FO(C)])$ to the complement $\g \setminus \bigsqcup_\lambda \g_\lambda$ equals $0$.
\end{proof}
 
We are left with proving Lemma \ref{keycomp}, for which we need the explicit formula \eqref{explinvx} for the inverse of an element $b = \unt + zb_1 +\dots +z^{m-1}b_{m-1} \in \G_m^1$.  Notice that for $\left\lfloor \frac{m+1}{2}\right\rfloor \leq p \leq m-1$, $b_p$ can appear at most once in each summand on the right hand side of \eqref{explinvx}. This is the crucial observation in the proof of Lemma \ref{keycomp}.  

\begin{lemma}\label{keycomp} For $X \in \g$, the function 
\begin{align*} 
\phi_X: \G_m^{\od} &\rightarrow \K & b & \mapsto \phi_X(b_1,b_2,\dots,b_{m-1}) =  \lla \Ad_b C,X \rra
\end{align*}
is affine linear in $b_{\left\lfloor \frac{m+1}{2}\right\rfloor}, \dots, b_{m-1}$. It is independent of those variables if and only if $X \in \ft$ and $b_1, b_2,\dots,b_{\left\lfloor \frac{m-2}{2}\right\rfloor}$ commute with $X$.

In this case, if $m$ is odd and we decompose $b_{\frac{m-1}{2}} = b^l + b^u$, where $b^l$ and $b^u$ are strictly lower and upper triangular respectively, then $\phi_X$ is affine linear in $b^u$ and independent of $b^u$ if and only if $b^l$ commutes with $X$.
\end{lemma}
  
\begin{proof} It follows directly from the observation above that $\phi_X$ depends linearly on $b_i$ for $\left\lfloor \frac{m+1}{2}\right\rfloor \leq i \leq m-1$. 

For $b \in \G_m^{\od}$, using the notation (\ref{explinvx}) we have 
\begin{equation} \label{evrytn} \lla \Ad_b C, X\rra = \tr \sum_{i=1}^m \sum_{j=0}^{i-1} b_j C_i w_{i-j-1}X,\end{equation}
where we use the convention $b_0 = w_0 = \unt$. We start by looking at the dependence of $ \lla \Ad_b C, X\rra$ when varying $b_{m-1}$. The terms in (\ref{evrytn}) containing $b_{m-1}$ are given by 
\[\tr(b_{m-1}C_mX - C_mb_{m-1}X) = \tr [b_{m-1},C_m]X.\]

As $C_m \in \ft^{\reg}$, the commutator $[b_{m-1},C_m]$ can take any value in $\g^{\od}$, thus $\tr [b_{m-1},C_m]X$ is independent of $b_{m-1}$ if and only if $X \in (\g^{\od})^\perp = \ft$.

Assume from now on $X \in \ft$. We show now inductively that $\phi_X$ is independent of $b_{\left\lfloor \frac{m+1}{2}\right\rfloor  }, \dots, b_{m-2}$ if and only if $b_1, b_2,\dots,b_{\left\lfloor \frac{m-2}{2}\right\rfloor}$ all commute with $X$.

To do so, fix $\left\lfloor \frac{m+1}{2}\right\rfloor \leq p \leq m-2$ and assume that $b_1,\dots, b_{m-2-p}$ commute with $X$. Consider the element 
\[ b' = \unt + zb_1+\dots + z^{p-1}b_{p-1} + z^{p} b_{p}X+z^{p+1} b_{p+1} + \dots + z^{m-1} b_{m-1} \in \G_m^{\od}.\]

The point now is that the $b_p$-parts of the explicit formulas for $\lla \Ad_b C, X\rra$ and 
\begin{align*}
\Res_0 \tr (b' C b'^{-1}) = \Res_0\tr( C)= \tr (C_1)
\end{align*}
are very similar. Indeed, from (\ref{evrytn}) we see, that all the terms containing $b_p$ in $\lla \Ad_b C,X\rra$ are contained in 
\begin{equation}\label{lala1} \tr \sum_{i=p+1}^m b_pC_i w_{i-p-1}X + \sum_{r=p}^{m-1} \sum_{i=r+1}^m b_{i-r-1} C_i w_r X.\end{equation}

To write a formula for $\Res_0 \tr( b' C b'^{-1})$ we write $b'^{-1} = \unt +z w'_1+\dots + z^{m-1} w'_{m-1}$. Then we can use a similar expression as (\ref{evrytn}) to conclude that all the terms containing $b_p$ in $\Res_0 \tr (b' C b'^{-1})$ are contained in 
\begin{equation}\label{lala2} \tr \sum_{i=p+1}^m b_p X C_i w'_{i-p-1} + \sum_{r=p}^{m-1} \sum_{i=r+1}^m b_{i-r-1} C_i w'_r.\end{equation}

Next we want to study the dependence of the difference (\ref{lala1}) $-$ (\ref{lala2}) on $b_p$. Notice first, that since $[b_i,X]=0$ for $1\leq i \leq m-2-p$ also $[w_i,X]=0$ for $1\leq i \leq m-2-p$ and furthermore $[w_{m-p-1},X] = [X,b_{m-p-1}]$. From (\ref{explinvx}) we also see $w_i' = w_i$ for all $1 \leq i < p$. Finally, we remark that $w_rX -w_r'$ is independent of $b_p$ for $p \leq r \leq m-2$ and the terms containing $b_p$ in $w_{m-1}X-w'_{m-1}$ are given by $b_p [b_{m-p-1},X]$. Combining all this we see that the terms containing $b_p$ in (\ref{lala1}) $-$ (\ref{lala2}) are just 
\[ \tr \left(b_p C_m [X,b_{m-p-1}] + C_m b_p [b_{m-p-1},X]\right) = \tr [b_p,C_m][X,b_{m-p-1}].\]
Since $C_m \in \ft^{\reg}$, the commutator $[b_p,C_m]$ can take any value in $\g^{\od}$ as we vary $b_p \in \g^{\od}$. Hence in order for $\tr [b_p,C_m][X,b_{m-p-1}]$ to be constant, we need $[X,b_{m-p-1}] \in \ft$. Since $X \in \ft$ this is only possible if $[X,b_{m-p-1}] = 0$, which finishes the induction step.

Finally, we consider the special case when $m$ is odd. Take $p = \frac{m-1}{2}$. Then by the same argument as before, we obtain that all the terms in $\lla \Ad_bC,X \rra$  which depend on $b_p$ are $\tr [b_p,C_m][X,b_{p}]$. Now using the decomposition $b_p = b^l+b^u$ we have 
\[\tr [b_p,C_m][X,b_{p}] = 2 \tr(b^u[C_m,[X,b^l]]).\]
Since the orthogonal complement of strictly upper triangular matrices are the upper triangular matrices we see that $ \tr(b^u[C_m,[X,b^l]])$ is independent of $b^u$ if and only if $[X,b^l] = 0$.
\end{proof}

\subsection{Motivic classes of open de Rham spaces}\label{mcodr}

In this section we compute the motivic class of the generic open de Rham space $[\FM^*_{n,\bfr}] \in \mathscr{M}$ as defined in Definition \ref{genodr}.

\begin{thm}\label{mainodr} The motivic class $[\FM^*_{n,\bfr}] \in \mathscr{M}$ is given by 
\begin{align} \label{totalct}  
\frac{\BL^{\frac{d_\mathbf{r}}{2}} }{(\BL-1)^{nd-1}} &\sum_{\lambda \in \FP_n} (-1)^{n(d-1)+l(\lambda)-1} \frac{(l(\lambda)-1)!}{u_\lambda} \binom{n}{\lambda}^d \BL^{\frac{(r-1)}{2}\left(N(\lambda)-n\right)} \phi_\lambda(\BL)^{d-1},\end{align}
where $d_\mathbf{r}= (n^2-n)(r+d)-2(n^2-1)$ denotes the dimension of $\FM^*_{n,\bfr}$.
\end{thm}

We start by simplifying $[\FM^*_{n,\bfr}]$ in the following standard way.

\begin{lemma}\label{printrick} We have the following relation in $\mathscr{M}$ 
\begin{equation}\label{firsta} [\FM^*_{n,\bfr}] =  \frac{(\BL-1) [\mu^{-1}(0)]}{[\G]}. \end{equation}
\end{lemma}

\begin{proof} Define $\G_\mathbf{m}= \prod_{i=1}^d \G_{m_i}$ and $\T_\mathbf{m} = \prod_{i=1}^d \T_{m_i}$. It follows from Lemmata \ref{l:coadjorb}\eqref{l:coadjhomog} and \ref{intor}\eqref{l:centralizers} that the natural map 
\[\alpha:\G_\m \rightarrow \prod_{i=1}^d \FO(C^i),\]
is a principal $\T_\m$-bundle. Notice that $\alpha$ is $\G$-equivariant with respect to the free $\G$-action on $\G_\m$ given by diagonal left multiplication. 

By restriction we obtain a $\G$-equivariant principal $\T_\m$-bundle $X \rightarrow \mu^{-1}(0)$. Taking the (affine GIT-)quotient by $\G$, we obtain a principal $\K^\times \setminus \T_\m$-bundle $\G \setminus X \rightarrow \FM^*_{n,\bfr}$. Also, $X \rightarrow \G \setminus X$ is a principal $\G$-bundle, as it is the restriction of $\G_\m\rightarrow \G\setminus \G_\m$. As the groups $\T_\m, \G_\m$ and $\K^\times \setminus \T_\m$ are special \cite[\S4.3]{serre58}, all the principal bundles here are Zariski locally trivial and we get
\[ [\FM^*_{n,\bfr}] = \frac{[\G\setminus X]}{[\K^\times \setminus \T_\m]} = \frac{[X]}{[\G][\K^\times \setminus \T_\m]} = \frac{[\mu^{-1}(0)][\T_\m](\BL-1)}{[\G][\T_\m]} = \frac{(\BL-1) [\mu^{-1}(0)]}{[\G]} . \qedhere \]
\end{proof}

By (\ref{firsta}) it is enough to determine $[\mu^{-1}(0)] = 0^*[\FO(C^1)\times \cdots \times \FO(C^d)]$, where we consider $0:\spec \K \rightarrow \g^{*}$ as a morphism. 

Since the motivic class of $\prod_{i=1}^d  \FO(C^i)$ relative to $\g$ is the convolution of the individual classes $[\FO(C^i)]$ we have by Proposition \ref{fouconv} the equality
\begin{equation}\label{foffs} \FF\left(\left[\prod_{i=1}^d  \FO(C^i)\right] \right)= \FF\left([\FO(C^1)] \ast \cdots \ast [\FO(C^d)] \right) = \prod_{i=1}^d \FF([\FO(C^i)]) \in \kexp_\g.\end{equation}
Notice that the last product is relative to $\g$, hence we have by Theorem \ref{ftpuncturet} for every $\lambda \in \FP_n$ 

\begin{align*} i_\lambda^*\prod_{i=1}^d \FF([\FO(C^i)]) &= (\BL -1)^{-nd} \BL^{\frac{1}{2}(r(n^2-2n)+dn^2 +N(\lambda)(r-d))}\prod_{i=1}^d\left[\ZZ_\lambda, \phi^{C^i}\right] \\
&= (\BL -1)^{-nd} \BL^{\frac{1}{2}(r(n^2-2n)+dn^2 +N(\lambda)(r-d))} \left[Z^d_\lambda,\sum_{i=1}^d \phi^{C^i} \right], 
\end{align*}
with the notation $r = \sum r_i$, $\ZZ_\lambda^d$ for the $d$-fold product $\ZZ_\lambda \times_\g \dots \times_\g \ZZ_\lambda$ and $\sum_i \phi^{C^i}$ for the function taking $(z_1,\dots,z_d)\in \ZZ_\lambda^d$ to $\sum_i \phi^{C^i}(z_i)$.
By Fourier inversion (Proposition \ref{finv}) we thus get
\begin{align} \nonumber [\mu^{-1}(0)]& =\BL^{-n^2} 0^* \FF\left(\prod_{i=1}^d \FF([\FO(C^i)])\right) \\
	\label{seca}&=  (\BL -1)^{-nd} \BL^{\frac{1}{2}(r(n^2-2n) +dn^2) - n^2} \sum_{\lambda \in \FP_n} \BL^{\frac{1}{2}N(\lambda)(r-d)} \left[Z^d_\lambda,\sum_i \phi^{C^i} \right]
\end{align}

This leaves us with understanding $\left[Z^d_\lambda,\sum_i \phi^{C^i} \right]$ as an element of $\kexp$. We start by taking a closer look at $\ZZ_\lambda = \{ (g,X) \in G \times \g_\lambda \ |\ \Ad_{g^{-1}}X \in \ft\}$. If we put $\ft_\lambda = \ft \cap \g_\lambda$ we have an isomorphism 
\begin{align} \label{decoup} 
\ZZ_\lambda &\xrightarrow{\sim} \ft_\lambda \times \G & (g,X) &\mapsto (\Ad_{g^{-1}} X, g).
\end{align}

Next, we need to fix some notation to describe $\ft_\lambda$ combinatorially. To parametrize the eigenvalues of elements in $\ft_\lambda$ define for any $e \in \BN$ the open subvariety $\BA_{\circ}^e \subset \BA^e$ as the complement of $\cup_{i \neq j} \{x_i = x_j\}$.

Furthermore we need some discrete data. A \textit{set partition of $n$} is a partition 
\begin{align*}
I = (I_1,I_2,\dots ,I_l)
\end{align*}
of $\{ 1,2,\dots ,n\}$, i.e., $I_i \cap I_j = \emptyset$ for $i \neq j$ and $\cup_i I_i = \{ 1,2,\dots ,n\}$. For $\lambda = (\lambda_1 \geq\dots \geq \lambda_l) \in \FP_n$ we write $\FP_\lambda$ for the set of set partitions $I=(I_1,\dots,I_l)$ of $n$ such that $( |I_1|,\dots,|I_l|) = (\lambda_1,\dots,\lambda_l)$. Notice that $I = (I_1,I_2,\dots ,I_l)$ is ordered and hence we have 
\begin{align*}
|\FP_\lambda| = \frac{n!}{\prod_{i=1}^l\lambda_i!}= \binom{n}{\lambda}. 
\end{align*}

\begin{lemma}\label{param}  The morphism 
\begin{align*}
p  & : \FP_\lambda \times \BA_\circ^l \rightarrow \ft_\lambda, & (I, \alpha) & \mapsto \sum_{j=1}^l \alpha_j E_{I_j},
\end{align*}
is a trivial covering of degree $u_\lambda$, where $E_{I_j}$ is defined as in (\ref{eij}).
\end{lemma}

\begin{proof} For $e \in \BN$ let $\Ss_e$ be the symmetric group on $e$ elements. Then the lemma follows from the fact that the subgroup $\prod_{j \geq 1} \Ss_{m_j(\lambda)}$ of $\Ss_l$ acts simply transitively on the fibers of $p$.
\end{proof}

\begin{lemma}\label{thirda} The following relation holds in $\kexp$
\[ \left[\ZZ_\lambda^d, \sum_{i=1}^d \phi^{C^i} \right] = \frac{[\G \times \L_\lambda^{d-1}]}{u_\lambda}  \sum_{(I^1,\dots,I^d)\in (\FP_\lambda)^d} \left[\BA_\circ^{l}, \sum_{i=1}^d \lla C_1^i,p_{I^i}\rra\right], \]
where $p_{I^i}= p_{|\{I^i\} \times \BA_\circ^l}: \BA_\circ^l \rightarrow \ft_\lambda.$
\end{lemma}
\begin{proof} For $\alpha \in \BA^l_\circ$ the map $p_{|\FP_\lambda \times \{\alpha \}}:\FP_\lambda \to \ft_\lambda$ from Lemma \ref{param} is injective and its image are exactly the elements in $\ft_\lambda$ with eigenvalues given by $\alpha$. Combining this with \eqref{decoup} we see
\[ \ZZ_\lambda \times_\g \ZZ_\lambda \cong \left(\ft_\lambda \times \G\right) \times_\g  \left(\ft_\lambda \times \G\right) \cong \ft_\lambda \times \G \times \L_\lambda \times \FP_\lambda.\]

Applying this reasoning $d-1$ times and then using Lemma \ref{param} we obtain a trivial covering of degree $ u_\lambda$
\[ \G \times \L_\lambda^{d-1} \times \FP_\lambda^d  \times \BA^l_\circ \to   \ZZ_\lambda^d.\]
Keeping track of the isomorphisms gives the desired equality.
\end{proof}

\begin{proof}[Proof of Theorem \ref{mainodr}] Combining (\ref{firsta}), (\ref{seca}) and Lemma \ref{thirda} we are left with computing the character sum $\left[\BA_\circ^{l}, \sum_{i=1}^d \lla C_1^i,p_{I^i}\rra \right]$ for a fixed $d$-tuple $(I^1,\dots,I^d)$ of set partitions of $n$. For $\alpha \in \BA_\circ^l$ we can write 
\[ \sum_{i=1}^d \lla C_1^i,p_{I^i}(\alpha)\rra = \sum_{i=1}^d \lla C_1^i, \sum_{j=1}^l \alpha_j E_{I_j^i} \rra = \sum_{j=1}^l \alpha_j \sum_{i=1}^d \lla C_1^i,E_{I^i_j}\rra. \]

Now, for a fixed $1\leq j \leq l$ we have by definition $|{I_j^1}|=\dots =|{I_j^d}|$. Thus by our genericity assumption (\ref{gen}) the numbers $\beta_j = \sum_{i=1}^d \lla C_1^i,E_{I^i_j}\rra$ satisfy the assumptions of Lemma \ref{finalc} below, and we deduce

\begin{equation} \label{tadaa} \left[\BA_\circ^{l}, \sum_{i=1}^d \lla C_1^i,p_{I^i}\rra\right] = (-1)^{l-1} (l-1)! \BL,
\end{equation}
which proves Theorem \ref{mainodr}.
\end{proof}

\begin{lemma} \label{finalc} Let $\beta_1,\dots,\beta_e \in \K$ be such that $\sum_{j=1}^e \beta_j = 0$ and for $J \subset \{1, 2,\dots, e\}$ a proper subset,  $\sum_{j\in J} \beta_j \neq 0$. Then for the function 
\begin{align*}
& \lla \cdot, \beta\rra: \BA^e_\circ \rightarrow \K, &  \alpha &\mapsto \sum_{j=1}^e \alpha_j\beta_j, 
\end{align*}
we have
\[ [\BA^e_{\circ}, \lla \cdot, \beta\rra ] = (-1)^{e-1} (e-1)! \BL \in \kexp.\]
\end{lemma}

\begin{proof} We use induction on $e$. For $e=1$ we have $\beta_1 =0$ and $\BA^1_\circ = \BA^1$, hence the statement is clear. For the induction step consider $\BA^{e}_\circ$ as a subvariety of $\BA^{e-1}_\circ \times \BA^1$. As $\beta_e\neq 0$ we have $[\BA^{e-1}_\circ \times \BA^1, \lla \cdot, \beta\rra] = 0$ by Lemma \ref{orth}, hence 
\[ [\BA^e_{\circ}, \lla \cdot, \beta\rra ] = - [\BA^{e-1}_\circ \times \BA^1 \setminus \BA^{e}_\circ,  \lla \cdot, \beta\rra ].\] 
Now notice that the complement $\BA^{e-1}_\circ \times \BA^1 \setminus \BA^{e}_\circ$ has $e-1$ connected components, each of which is isomorphic to $\BA^{e-1}_\circ$, which implies the formula.
\end{proof}

\section{Tame poles and finite fields} \label{cac}

Here, the notation will be the same as in Sections \ref{s:odr} and \ref{s:motcls}, especially Definition \ref{genodr} and Section \ref{s:partnote}.  In this section, we replace the motivic computations over an algebraically closed field with arithmetic ones over a finite field $\BF_q$. This allows us to study the number of $\BF_q$-rational points of $\FM^*_{\muhat,\bfr}$ for any pair $(\muhat,\bfr)$ with $s \geq 1$.

First, we derive in Theorem \ref{finalcount} a closed formula for $|\FM^*_{\muhat,\bfr}(\BF_q)|$ using techniques similar to those of Section \ref{s:motcls}. Then, in Theorem \ref{maincor} we give a second description of $|\FM^*_{\muhat,\bfr}(\BF_q)|$ in terms of symmetric functions, which allows us to identify the $E$-polynomial of $\FM^*_{\muhat,\bfr}$ with the pure part of the conjectural mixed Hodge polynomial of the corresponding character variety \cite[Conjecture 0.2.2]{HMW16}. This gives strong numerical evidence for the purity Conjecture \ref{purity}, which was one of the main motivations of this paper. 

\subsection{de Rham spaces and finite fields}

Let $\BF_q$ be a finite field of characteristic coprime to $n$ and $q$ large enough (see Remark \ref{larchar}) so that there exists a regular generic tuple $\mathbf{C}=(C^1,\dots,C^d)$ of formal types over $\BF_q$ for a given pair $(\muhat,\bfr)$ with $s\geq 1$. Then also $\FM^*_{\muhat,\bfr}$ is defined over $\BF_q$ and we can prove a version of Theorem \ref{mainodr} which also includes tame poles. 

For a variety $X$ defined over $\BF_q$ we sometimes abbreviate $|X(\BF_q)| = |X|$. We write $\g'(\BF_q) \subset \g(\BF_q)$ for the matrices in $\g(\BF_q)$ whose eigenvalues are in $\BF_q$ and for $\lambda \in \FP_n$ we write $\g'_\lambda(\BF_q)= \g'(\BF_q) \cap \g_\lambda(\BF_q)$. The proof of Theorem \ref{ftpuncturet} then implies that the Fourier transform of the count function $\#_C: \g^\vee(\BF_q) \rightarrow \BZ$, $Y \mapsto |\pi_{\res}^{-1}(Y)|$ associated to the coadjoint orbit $\pi_{\res} : \FO(C) \rightarrow \g^\vee$ is supported on $\bigsqcup_{\lambda \in \FP_n} \g'_\lambda(\BF_q)$. Given such an $X \in \g'_\lambda(\BF_q)$, the $\BF_q$-version of formula (\ref{ftpuncture}) reads 
\begin{equation}\label{ffft} \FF(\#_C)(X) = \frac{q^{n+\frac{1}{2}\left(m(n^2-2n)+(m-2)N(\lambda)\right)}}{(q-1)^n} \left| \L_\lambda \right| \sum_{t \in  \ft \cap \FO(X)(\BF_q)} \Psi(\lla C_1,t\rra),
\end{equation}
where $\FO(X) \subset \g$ denotes the orbit of $X$ under the adjoint action and $\FF$ and $\Psi$ are defined as in Section \ref{groexp}.

We now spell out a similar formula for the Fourier transform of a tame pole. Let $\mu = (\mu_1 \geq \dots \geq \mu_l) \in \FP_n$ be a partition, $C \in \ft(\BF_q) \cap \g_\lambda(\BF_q)$ and $1_{C}:\g(\BF_q) \to \BZ$ the characteristic function of the adjoint orbit of $C$. Furthermore, for  $e \in \BN$ we write $\Ss_e$  for the symmetric group on $e$ elements and $\Ss_\mu = \prod_{i=1}^{l} \Ss_{\mu_i} \subset \Ss_n$.  
Now conjugacy classes of maximal tori over $\BF_q$ in $G=\GL_n$ are in natural bijection with the conjugacy classes of the Weil group $W= \Ss_n$ \cite[Section 2.5.3]{HLV11}. For any $w \in \Ss_n$ we fix a maximal torus $\T_w$ corresponding to the conjugacy class of $w$ under this bijection and define
\[ Q_{\T_w}^{\L_\lambda} =   (-1)^{n-rk_q(\T_w)} \frac{|\L_\lambda|}{q^{\frac{1}{2}\left(N(\lambda)-n\right)}|\T_w|}, \] 
where $rk_q$ denotes the $\BF_q$-rank. This is a shorthand notation for the value at $1$ of the Green function \cite[Theorem 7.1]{DL76} and we can be even more explicit. If we write $\mu(w) \in \FP_n$ for the partition defined by the cycles of $w \in \Ss_n$, we have 
\begin{equation} \label{tw} 
|\T_w| = \prod_{i=1}^{l(\mu(w))} (q^{\mu(w)_i}-1). 
\end{equation}

\begin{lemma}\label{fttame}For $\lambda \in \FP_n$ and any $X \in \g'_\lambda(\BF_q)$ we have
\begin{equation}\label{fttame2} \F\left(1_C\right)(X) = \frac{q^{\frac{1}{2}(n^2-N(\mu))}}{\mu!} \sum_{w \in \Ss_\mu} Q^{\L_\lambda}_{\T_w}   \sum_{Y \in \t^w \cap \FO(X)(\BF_q)} \Psi(\lla C, Y \rra), \end{equation} where $\t^w \subset \t$ denotes the locus fixed by $w \in S_\mu$.  If $C$ is regular, then $\F\left(1_C\right)$ is supported on $\g'(\BF_q)$.
\end{lemma}
\begin{proof} The proof consists of writing out Equation (2.5.5) in \cite{HLV11} explicitly. In their notation we have 
\begin{align*}
\F\left(1_C\right)= \epsilon_G \epsilon_L |\W_\L|^{-1} \sum_{w \in \W_\L} q^{d_L/2} \mathscr{R}_{\t_w}^\g \big( \F^{\t_w}( 1^{\T_w}_{C}) \big).
\end{align*}

Here $\L \cong \prod_{i=1}^l \GL_{\mu_i}$ denotes the centralizer of $C$ and $\W_\L \cong \prod_{i=1}^l \Ss_{\mu_i}$ the normalizer of the standard maximal torus $\T$ in $\L$.  We have $d_\L = \dim \G - \dim \L = n^2 - N(\mu)$ and since $\T$ is split over $\BF_q$ also $\epsilon_\G \epsilon_\L = (-1)^n(-1)^n = 1$. The formula then becomes 
\[\F\left(1_C\right) = \frac{q^\frac{n^2-N(\mu)}{2}}{\mu!} \sum_{w \in \W_\L} \mathscr{R}_{\t_w}^\g \big( \F^{\t_w}( 1^{\T_w}_{C}) \big).\]

The $\T_w$-coadjoint orbit of $C$ is just $C$ and hence $\F^{\t_w}( 1^{\T_w}_{C})(Y) = \Psi(\lla C, Y \rra)$ for all $Y \in \ft_w$. As $X \in \g'_\lambda(\BF_q)$ is semisimple \cite[Equation (2.5.4)]{HLV11} gives
\begin{align} \label{e:indR}
\mathscr{R}_{\t_w}^\g \big( \FF^{\t_w}( 1^{\T_w}_{C} ) \big) (X) = \left| \L_\lambda\right|^{-1} \sum_{ \substack{ \{h \in \G(\BF_q) | \\ X \in \Ad_h \t_w \} }} Q_{h\T_wh^{-1}}^{C_\G(X)} \big(1 \big) \Psi(\lla C,\Ad_{h^{-1}} X\rra).
\end{align}

By  \cite[Theorem 7.1]{DL76}, for any $h \in \G(\BF_q)$, $Q_{h\T_wh^{-1}}^{C_G(X)}(1) = Q_{\T_w}^{\L_\lambda}$ and thus 
\begin{align*}
\mathscr{R}_{\t_w}^\g \big( \F^{\t_w}( 1^{\T_w}_{C} ) \big) (X) &=  \frac{Q_{\T_w}^{\L_\lambda}}{|\L_\lambda|} \sum_{ \substack{ \{h \in \G(\BF_q) | \\ X \in \Ad_h \t_w \} }} \Psi(\lla C,\Ad_{h^{-1}} X\rra) = Q_{\T_w}^{\L_\lambda} \sum_{Y \in \t_w \cap \FO(X)(\BF_q)} \Psi(\lla C, Y \rra).
\end{align*}

Using \cite[2.5.3]{HLV11} we see that for every $w \in W_{C_G(C)}$ there is a $g \in C_G(C)(\overline{\BF}_q)$ such that $(\t_w \cap \FO(X))(\BF_q) = g (\t^w\cap \FO(X)(\BF_q))g^{-1}$ and hence 
\[ \sum_{Y \in \t_w \cap \FO(X)(\BF_q)} \Psi(\lla C, Y \rra) = \sum_{Y \in \t^w \cap \FO(X)(\BF_q)} \Psi(\lla C,gYg^{-1} \rra) = \sum_{Y \in \t^w \cap \FO(X)(\BF_q)} \Psi(\lla C, Y \rra). \]

If $C$ is regular we have $\W_\L = \{1\}$ and for $X\in \g(\BF_q)$  clearly $\{h \in \G(\BF_q) | X_s \in \Ad_h(\t)\} = \emptyset$ unless $X\in \g'(\BF_q)$, which finishes the proof of the lemma.
\end{proof}

Recall that for $\lambda \in \FP_n$ we denote by $\FP_\lambda$ the set of ordered set partitions of $\lambda$. The natural action of $S_n$ on $\{1,\dots,n\}$ induces an action on $\FP_\lambda$ and for $w \in \Ss_n$ we write $\FP^w_\lambda \subset \FP_\lambda$ for the set partitions invariant under $w$ i.e.
 \[\FP_\lambda^w=\{ I=(I_1,\dots,I_l) \in \FP_\lambda \ | \ wI_j = I_j \text{ for all } 1\leq j\leq l\}.\]
For two partitions $\lambda,\mu \in \FP_n$ we then define 
\[ \Delta(\lambda,\mu) = \frac{1}{\mu!}\sum_{w \in S_\mu} Q_{\T_w}^{\L_\lambda}|\FP^w_\lambda|. \]

\begin{thm} \label{finalcount} Let $d=k+s$ and $\mathbf{C} = (C^1,\dots,C^k,C^{k+1},\dots,C^d)$ generic regular formal types. If $s\geq 1$ we have
\begin{align}  |\FM^*_{\muhat,\bfr}(\BF_q)| &= \label{finalcountf} \\ \nonumber
\frac{q^{ \frac{d_{\muhat,\mathbf{r}}}{2}}}{(q-1)^{ns-1} } &\sum_{\lambda \in \FP_n}(-1)^{n(s-1)+	l(\lambda)-1} \frac{(l(\lambda)-1)!}{u_\lambda } \binom{n}{\lambda}^s q^{\frac{(r-1)}{2}\left(N(\lambda)-n \right) } \phi_\lambda(q)^{s-1}\prod_{i=1}^k \Delta(\lambda,\mu^i), \end{align}
where $d_{\muhat,\bfr}$ denotes the dimension of $\FM^*_{\muhat,\bfr}$.
\end{thm}

\begin{proof} The argument is similar to the proof of Theorem \ref{mainodr}. First by repeating the proof of Lemma $\ref{printrick}$ with $\FM^*_{n,\bfr}$ replaced by $\FM^*_{\muhat,\bfr}$ we see that
\[ |\FM^*_{\muhat,\bfr}(\BF_q)| = \frac{(q-1) |\mu^{-1}(0)|}{|\G|},\]
where $\mu$ is the moment map defined in \eqref{e:mud}. Using Fourier inversion and convolution over finite fields we further have \small
\begin{align*} |\mu^{-1}(0)|  &=  q^{-n^2}\FF\left( \prod_{i=1}^k \FF(1_{C^i}) \prod_{j=1}^s \FF(\#_{C^{k+j}})\right)(0) = q^{-n^2} \sum_{X \in \g(\BF_q)}\prod_{i=1}^k \FF(1_{C^i})(X) \prod_{j=1}^s \FF(\#_{C^{k+j}})(X)\\
&=  q^{-n^2} \sum_{\lambda \in \FP_n} \sum_{X \in \g'_\lambda(\BF_q)} \prod_{i=1}^k \FF(1_{C^i})(X) \prod_{j=1}^s \FF(\#_{C^{k+j}})(X).   \end{align*} \normalsize
For the last equation we used that $\FF(\#_{C^{k+j}})$ is supported on $\bigsqcup_{\lambda \in \FP_n} \g'_\lambda(\BF_q)$ for all $1\leq j \leq s$ and our assumption $s\geq 1$. 
We now parametrize $\g'_\lambda(\BF_q)$ using the same notation as in Section \ref{mcodr}, i.e., we have a surjective map 
\begin{align*}
p & : \G(\BF_q) \times \FP_\lambda \times \BA^l_\circ(\BF_q) \to \g'_\lambda(\BF_q) & 
(g,I,\alpha) &\mapsto g\left( \sum_{j=1}^l \alpha_j E_{I_j} \right)g^{-1},
\end{align*}
with each fiber of $p$ having cardinality $u_\lambda \binom{n}{\lambda} |\L_\lambda|$. Since $\FF(1_{C^i})$  and $\FF(\#_{C^{k+j}})$ are class functions they depend only on $\alpha \in \BA_\circ^l(\BF_q)$. In particular if $X= p(g,I,\alpha)$ we have for every $w \in \Ss_n$ 
\[ \ft^w\cap \FO(X)(\BF_q) = p\left( \{\Id\} \times \FP_\lambda^w \times \{\alpha\} \right),       \]
with $\FP_\lambda^w=\{ I=(I_1,\dots,I_l) \in \FP_\lambda \ | \ wI_j = I_j \text{ for all } 1\leq j\leq l\}$. Using these parametrizations of $ \g'_\lambda(\BF_q)$ and $\ft^w\cap \FO(X)(\BF_q)$ as well as  \eqref{ffft} and Lemma \ref{fttame}  we can write

\small
\begin{align*}
& \sum_{X \in \g'_\lambda(\BF_q)} \prod_{i=1}^k \FF(1_{C^i})(X) \prod_{j=1}^s \FF(\#_{C^{k+j}})(X)  \\
& = \frac{q^{\Omega}|L_\lambda|^s}{(q-1)^{ns}} \sum_{X \in \g'_\lambda(\BF_q)}\prod_{i=1}^k \left( \sum_{w \in \Ss_{\mu^i}} \frac{Q^{\L_\lambda}_{\T_w}}{\mu^i!}   \sum_{Y \in \t^w \cap \FO(X)(\BF_q)} \Psi(\lla C^i, Y \rra) \right) \prod_{j=1}^s \left( \sum_{t \in  \ft \cap \FO(X)(\BF_q)} \Psi(\lla C^{k+j}_1,t\rra) \right)\\
& = \frac{q^{\Omega}|L_\lambda|^{s-1}|G|}{(q-1)^{ns}u_\lambda} \sum_{\alpha \in \BA_\circ^l(\BF_q)}\prod_{i=1}^k \left( \sum_{w \in \Ss_{\mu^i}}  \frac{Q^{\L_\lambda}_{\T_w}}{\mu^i!}   \sum_{I \in \FP_\lambda^w } \Psi\lla C^i, \sum_{a = 1}^l \alpha_a E_{I_a} \rra \right) \prod_{j=1}^s \left( \sum_{I \in  \FP_\lambda} \Psi\lla C^{k+j}_1,\sum_{a = 1}^l \alpha_a E_{I_a}\rra \right) ,
\end{align*} \normalsize
where we abbreviated $\Omega = \frac{1}{2}\left( n^2(d+r)-2nr -\sum N(\mu^i)\right) + \frac{N(\lambda)}{2}(r-s)$. Now multiplying out the products over $i$ and $j$ we are left with computing sums of the form
\[ \sum_{\alpha \in \BA_\circ^l(\BF_q)} \Psi \left( \sum_{a=1}^l \alpha_a \lla \sum_{i=1}^d C^i_1,E_{I^i_a}\rra \right),   \]
with $I^i \in \FP_\lambda$. Because of the genericity assumption \eqref{gen} and (the finite field version of) Lemma \ref{finalc} this sum always equals $(-1)^{l-1} (l-1)!q$ independently of the $I^i$'s. With this simplification we finally deduce 

\small
\begin{align*}
 \sum_{X \in \g'_\lambda(\BF_q)} & \prod_{i=1}^k \FF(1_{C^i})(X) \prod_{j=1}^s \FF(\#_{C^{k+j}})(X)  \\
& = \frac{q^{\Omega} |\L_\lambda|^{s-1} |\G|}{(q-1)^{ns} u_\lambda } \prod_{i=1}^k\left( \frac{1}{\mu^i!}\sum_{w \in \Ss_{\mu^i}} Q_{\T_w}^{\L_\lambda} |\FP^w_\lambda |\right) |\FP_\lambda|^s (-1)^{l-1}(l-1)!q \\
& = \frac{q^{ \Omega+1 } |\L_\lambda|^{s-1} |\G|}{(q-1)^{ns} u_\lambda } \prod_{i=1}^k \Delta(\lambda,\mu^i) \binom{n}{\lambda}^s (-1)^{l-1}(l-1)!.
\end{align*} \normalsize
Summing up over all $\lambda \in \FP_\lambda$ finishes the proof.
\end{proof}

Theorem \ref{finalcount} immediately implies that $ |\FM^*_{\muhat,\bfr}(\BF_q)|$ is a rational function in $q$ as we vary the finite field $\BF_q$. Because it is the count of an algebraic variety we deduce that it is in fact a polynomial in $q$, namely its weight polynomial by Theorem \ref{Katzthm}. By analyzing the combinatorics of \eqref{finalcountf} we obtain the following 

\begin{cor}\label{nonemp} For every pair $(\muhat,\bfr)$ with $d_{\muhat,\bfr} \geq 0$ and $s \geq 1$ the counting polynomial for $\FM^*_{\muhat,\bfr}$ is monic of degree $d_{\muhat,\bfr}$. In particular  $\FM^*_{\muhat,\bfr}$ is non-empty and connected,  both over $\BF_q$ and $\mathbb{C}$.
\end{cor}

\begin{proof} The statement about the counting polynomial implies the second one, as the leading coefficient gives the number of components of $\FM^*_{\muhat, \bfr}$; see for example \cite[Lemma 5.1.2]{HLV13}. 

For $n=1$ we have $d_{\muhat,\bfr}  = 0$ and Theorem \ref{finalcount} gives indeed $ |\FM^*_{\muhat,\bfr}(\BF_q)| =1$. For $n\geq 2$, $d_{\muhat,\bfr} \geq 0$ implies either $r\geq 2$ or $r=s=1, k\geq 1$ and there exists an $1\leq i \leq k$ with $\mu^i \neq (n)$. In both cases it follows from Lemma \ref{degest} below that there is only one summand in \eqref{finalcountf} of highest degree, namely the one for $\lambda = (n)$. From the explicit formula for $\Delta((n),\mu)$ in Lemma \ref{degest} we also see that the coefficient of the highest $q$-power equals $1$ and the corollary follows.
\end{proof}

\begin{lemma}\label{degest} For any $\lambda,\mu \in \FP_n$ we have 
\begin{enumerate}
\item $  \Delta((n),\mu) =  (-1)^n \frac{ |\G|}{q^{\frac{1}{2}\left(n^2-n\right)} \phi_\mu(q)}$. 
\item $\deg \Delta((n),\mu) \geq \deg \Delta(\lambda,\mu),$ and equality holds if and only if $\lambda = (n)$ or $\mu = (n)$.
\end{enumerate}
\end{lemma}

\begin{proof} Let $\alpha \in \BN$ and write $\T^\alpha$ for the standard maximal torus of $\GL_\alpha$.The lemma boils down to the following combinatorial identity:  
\[\sum_{w \in \Ss_\alpha} \frac{(-1)^{\alpha-rk_q(T^{\alpha}_{w})}}{|T^{\alpha}_{w}|}	= (-1)^\alpha \alpha !\sum_{\nu \in \FP_\alpha} \frac{1}{u_\nu\prod_i \nu_i 	 \prod_i (1- q^{\nu_i})}= (-1)^\alpha \frac{\alpha !}{\phi_{(\alpha)}(q)}.    \]

Here the first equality follows from \eqref{tw}, while the second one can be deduced from \cite[(2.14') and Example I.2.4]{Mac}.

Now every $I  = (I_1, \ldots, I_\ell) \in \FP_\lambda$ defines a refinement $\mu^I$ of $\mu$ as follows. Write $\mu = (\mu_1 \geq \ldots \geq \mu_m)$ and $M_1 = \{ 1, \ldots, \mu_1 \}$, $M_i = \{ \sum_{p=1}^{i-1} \mu_p + 1, \ldots, \sum_{p=1}^i \mu_p \}$ for $2 \leq i \leq m$, so that $\{ 1, \ldots, n \} = \coprod_{i=1}^m M_i$.  Then $\mu^I$ is obtained by writing the numbers
\begin{align*}
| M_i \cap I_j |, \quad 1 \leq i \leq m, 1 \leq j \leq \ell
\end{align*}
in (some) non-increasing order. In particular $\{ w \in \Ss_\mu \ | \ wI = I  \} = \Ss_{\mu^I}$. With this we have
\begin{align*}   \Delta(\lambda,\mu) &= \frac{1}{\mu!} \sum_{w \in S_\mu} Q_{\T_w}^{\L_\lambda}|\FP^w_\lambda| =\frac{1}{\mu!}  \sum_{I \in \FP_\lambda} \sum_{w \in \Ss_{\mu^I}} Q_{\T_w}^{\L_\lambda} = \frac{|\L_\lambda|}{\mu!q^{\frac{1}{2}\left(N(\lambda)-n\right)}}\sum_{I \in \FP_\lambda} \sum_{w \in \Ss_{\mu^I}} \frac{ (-1)^{n-rk_q(\T_w)}}{|\T_w|} \\
&= (-1)^n \frac{|\L_\lambda|}{\mu! q^{\frac{1}{2}\left(N(\lambda)-n\right)}}\sum_{I \in \FP_\lambda} \frac{\mu^I!}{\phi_{\mu^I}(q)}.
\end{align*}

From this we get the formula for $ \Delta((n),\mu)$ as well as
\[\deg \Delta(\lambda,\mu) = \frac{1}{2} (N(\lambda) +n) - \min_{I \in \FP_\lambda} \frac{1}{2}(N(\mu^I)+n)  = \frac{1}{2}(N(\lambda) -   \min_{I \in \FP_\lambda} N(\mu^I)).\]
The inequality $\deg \Delta((n),\mu) \geq \deg \Delta(\lambda,\mu) $ is thus equivalent to 
\begin{align} \label{e:degineq}
 \min_{I \in \FP_\lambda} N(\mu^I) \geq N(\lambda) + N(\mu)-n^2.
 \end{align}
We prove this inequality by induction on $n$. Let $I \in \mathcal{P}_\lambda$ be such that $N(\mu^I)$ is minimal and suppose $n \in I_j$.  We may assume that $\lambda_{j+1} < \lambda_j$ since otherwise we could interchange $I_j$ and $I_{j+1}$ and the resulting $I'$ would still minimize $N(\mu^I)$. Hence
\begin{align*}
\tilde{\lambda} := (\lambda_1 \geq  \ldots \geq \lambda_{j-1} \geq \lambda_j - 1 \geq \lambda_{j+1} \geq \ldots \geq \lambda_\ell)
\end{align*}
is still a partition (of $n-1$).  Further let
\begin{align*}
\hat{\mu} := (\mu_1 \geq \ldots \geq \mu_{m-1} \geq \mu_m - 1),
\end{align*}
and $\hat{M}_i = M_i$ for $1 \leq i \leq m-1$, $\hat{M}_m = M_m \setminus \{ n \}$.
We have
\begin{align*}
N(\lambda) & = N(\tilde{\lambda}) + 2\lambda_j - 1 & N(\mu) & = N(\hat{\mu}) + 2\mu_m - 1.
\end{align*}
By induction, we know
\begin{align*}
\min_{J \in \mathcal{P}_{\tilde{\lambda}}} N(\hat{\mu}^J) \geq N(\tilde{\lambda}) + N( \hat{\mu}) - (n-1)^2.
\end{align*}
Therefore it suffices to show that
\begin{align} \label{e:tildelam}
\min_{J \in \mathcal{P}_{\tilde{\lambda}}} N(\hat{\mu}^J) \leq N(\mu^I) + 2(n- \lambda_j - \mu_m) + 1.
\end{align}
Consider
\begin{align*}
J = (I_1, \ldots, I_{j-1}, I_j \setminus \{n \}, I_{j+1}, \ldots, I_\ell) \in \mathcal{P}_{\tilde{\lambda}}.
\end{align*}
Clearly if $i \neq m$ or $k \neq j$ we have $|\hat{M}_i \cap J_k| = |M_i \cap I_k|$ and so 
\begin{align*}
N(\hat{\mu}^J) - N(\mu^I) & = \sum_{i=1}^m \sum_{k=1}^\ell | \hat{M}_i \cap J_k |^2 - \sum_{i=1}^m \sum_{k=1}^\ell | M_i \cap I_k |^2 = |\hat{M}_m \cap J_j|^2 - |M_m \cap I_j|^2 \\
& = (|M_m \cap I_j| - 1)^2 - |M_m \cap I_j|^2 = -2|M_m \cap I_j| + 1 \\
& = 2( |M_m \cup I_j| - \lambda_j - \mu_m) + 1 \leq 2( n - \lambda_j - \mu_m) + 1
\end{align*}
and taking the minimum we see that the inequality \eqref{e:tildelam} holds. The equality cases follows from a direct inspection.
\end{proof}

\begin{ex} Let us spell out explicitly formula \eqref{finalcountf} for $n=2,3$. Recall that $k$ denotes the number of simple poles, $s$ the number of higher order poles, $d=k+s$ the total number of poles, and $r$ the sum of the Poincar\'e ranks of the poles. We will assume $\mu_i \neq (n)$ for all $1 \leq i \leq k$, since the corresponding coadjoint orbits are points and do not contribute to the count. 

For $n=2$,  we have by assumption $\mu^i = (1^2)$ for $1 \leq i \leq k$ and thus 
\begin{eqnarray}\label{formn=2}  |\FM^*_{\muhat,\bfr}(\BF_q)|= \frac{q^{r+d-3} (q^{r-1}(q+1)^{d-1}-2^{d-1})}{q-1}.  \end{eqnarray}

For $n=3$, we write $k=k_1+k_2$ with $k_1 = |\{i \ | \ \mu^i = (1^3)\}|$ and $k_2 = |\{i\ |\ \mu^i = (2,1) \}|$.  Then $|\FM^*_{\muhat,\bfr}(\BF_q)|$ is given by \small
\begin{align} 
 \frac{q^{3r+3d-k_2-8} \left( q^{3r-3}(q^2+q+1)^{d-1}(q+1)^{s+k_1-1}-3^{s+k_1}q^{r-1}(q+1)^{s+k_1-1}(q+2)^{k_2}+3^{d-1}2^{s+k_1}   \right)   }{(q-1)^2}.\label{formn=3}\end{align} \normalsize
\end{ex}

\begin{rmk} If we take $k=0$, and thus $s=d$, formula \eqref{finalcountf} agrees with \eqref{totalct} if we replace $q$ by $\BL$. In fact by Theorem \ref{Katzthm}, \eqref{finalcountf} is a consequence of \eqref{totalct} provided there exists a suitable spreading out of $\FM^*_{n,\bfr}$. We refer to \cite[Appendix A]{HLV11}, where a similar construction has been carried out in detail.
\end{rmk}

\subsection{Comparison with wild character varieties} \label{s:wcvcomp}

\begin{lemma}\label{fourier} Let $m\geq 1$ and $C$ any formal type of order $m$ with $C_1$ regular semisimple. Then we have an equality of Fourier transforms
\[ \FF(\#_C) = \FF(1_{C_1}) \FF(1_\FN)^{m-1},\]
where $\FN \subset \g(\BF_q)$ denotes the nilpotent cone and $1_{\FN}$ its characteristic function on $\g(\BF_q) \cong \g^\vee(\BF_q)$.
\end{lemma}
\begin{proof}

By \cite[Theorem 3.6]{Le15}, the Fourier transform $\FF(1_\NN)$ is given by
\[ \FF(1_\NN) = \mathcal{U}_{1^n} = q^{\frac{1}{2}\dim \FO_{(n)}}\text{St}_n,\]
where $\text{St}_n$ denotes the Steinberg character and $\FO_{(n)}$ the regular nilpotent orbit, hence 
\begin{align*}
\dim \FO_{(n)} = n(n-1).
\end{align*}
By \cite[(2.2)]{Le15}, $\text{St}_n$ is supported on semisimple elements and its value on $X \in \g'_\lambda(\BF_q)$ is
\[\text{St}_n(X) = \epsilon_G \epsilon_{C_G(X)} |C_G(X)|_p = q^{\frac{1}{2}\sum_i \lambda_i(\lambda_i-1)}.\]
Thus in total we obtain 
\begin{align*}
\F(1_\NN)(X) = q^{\frac{1}{2} ( n^2-2n + N(\lambda))}.
\end{align*}
Comparing this with equation \eqref{ffft} and \eqref{fttame2} we finally see
\[ \FF(\#_C) = \FF(1_{C_1}) \F(1_\NN)^{m-1}. \qedhere \]
\end{proof}

Next we recall some combinatorics; for more details see \cite[\S 2.3]{HLV11} and \cite[\S 4.2]{HMW16}. For a partition $\lambda\in\calP_n$ we let \begin{equation*}
{\cal H}_\lambda(z,w)=\prod \frac{1}{(z^{2a+2}-w^{2l})(z^{2a}-w^{2l+2})}\in \Q(z,w),\end{equation*} 
where the product is over the boxes in the Young diagram of $\lambda$ and $a$ and $l$ are the arm length and the leg length of the given box.  

We fix integers $g\geq 0$ and $k>0$. Let
$\x_1=\{x_{1,1},x_{1,2},\dots\}, \dots,
\x_k=\{x_{k,1},x_{k,2},\dots\}$ be $k$  sets of infinitely many independent
variables and let $\Lambda(\x_1,\ldots,\x_k)$ be the ring of functions
separately symmetric in each of the set of variables.
Then we define the Cauchy kernel
$$\Omega_k(z,w):=\sum_{\lambda\in \calP}{\cal H}_\lambda(z,w)\prod_{i=1}^k \tilde{H}_\lambda(z^2,w^2;{\bf x}_i)\in \Lambda({\bf x}_1,\dots,{\bf x}_k)\otimes_\Z \Q(z,w),$$
where we have the Macdonald polynomials of \cite{garsia-haiman}
 $$\tilde{H}_\lambda(q,t;{\bf x})=\sum_{\mu\in \calP_n} \tilde{K}_{\lambda\mu}(q,t) s_\mu(\bf{x})\in \Lambda({\bf x})\otimes_\Z\Q(z,w).$$ Let now $\muhat \in \calP_n^k$, $\r =(r_1,\dots,r_s) \in \Z_{> 0}^s$ and $r:=r_1+\dots +r_s$ be as in Definition \ref{genodr}.  We let \small
\begin{align*}&\BH_{\muhat,r}(z,w)=\\ & (-1)^{rn} (z^2-1)(1-w^2)\left\langle {\rm Log} \left(\Omega_{k+r}\right),h_{\mu_1}({\bf x}_1)\otimes \cdots \otimes h_{\mu_k}({\bf x}_k)\otimes s_{(1^n)}({\bf x}_{k+1})\otimes \cdots \otimes s_{(1^n)}({\bf x}_{k+s})\right\rangle.\end{align*}
\normalsize

Let now $\mathcal{M}_{\rm B}^{\muhat,\r}$ denote a generic wild character variety of type ${\muhat,\r}$ and $g=0$ as studied in \cite{HMW16}.  The main conjecture \cite[Conjecture 0.2.2]{HMW16} in the form of \cite[Lemma 5.2.1]{HMW16} is that the mixed Hodge polynomial of $\mathcal{M}_{\rm B}^{\muhat,\r}$ satisfies \begin{equation}\label{mhpconj}WH(\mathcal{M}_{\rm B}^{\muhat,\r};q,t)= (qt^2)^{d_{\mu,\r}/2} \BH_{\tilde{\muhat},r}\left(q^{1/2},-q^{-1/2} t^{-1}\right),\end{equation} where  $\tilde{\muhat}=(\mu^1,\dots,\mu^k,(1^n),\dots,(1^n))\in \calP_n^{k+s}$.
The purity conjecture then is saying that the pure part of 
$WH(\mathcal{M}_{\rm B}^{\muhat,\r};q,t)$ should equal the Poincar\'e polynomial of the corresponding open de Rham space. 

The following is our main Theorem~\ref{main}. 
\begin{thm} \label{maincor} In the notation of Theorem \ref{finalcount}, we have $E(\FM_{\muhat,\bfr}^*;q)=q^{d_{\muhat,\r}/2} \BH_{\tilde{\muhat},r}(0,q^{1/2})$.
\end{thm}

\begin{proof} By Theorem \ref{Katzthm} we have $E(\FM_{\muhat,\bfr}^*;q)=|\FM_{\muhat,\bfr}^*(\BF_q)|$ and by Lemma \ref{fourier} we can compute \small
\begin{equation} \label{finalfourier}|\FM_{\muhat,\bfr}^*(\BF_q)|=\frac{q^{-n^2}}{|\text{P\GL}_n(\BF_q)|}\FF\left(\prod_{i=1}^k \FF(1_{C^{i}})\prod_{j=1}^s\FF(1_{C^{k+j}_{1}})\FF(1_\NN)^{r_j}\right)(0)=q^{d_{\muhat,\r}/2}\BH_{\tilde{\muhat},r}(0,q^{1/2}).\end{equation} \normalsize
The last equation follows from combining \cite[Theorem 7.3.3]{Le13} and \cite[Corollary 7.3.5]{Le13} with the observation in \cite[\S 6.10.3]{Le13} that the symmetric function corresponding to a split semisimple adjoint orbit  of type $\mu^i$ is $h_{\mu_i}$ and in \cite[\S 6.10.4]{Le13} the symmetric function corresponding to a regular nilpotent adjoint orbit is $ s_{(1^n)}$.  The result follows in the usual way, as in the proof of \cite[Theorem 7.1.1]{HLV11}.
\end{proof}

\begin{rem} The formula for $ |\FM^*_{\muhat,\bfr}(\BF_q)| =E(\FM_{\muhat,\bfr}^*;q)$ in \eqref{finalcountf} was obtained by computing the same Fourier transform as in \eqref{finalfourier}. Thus we can consider the result of Thm~\ref{finalcount} as the explicit form of Theorem~\ref{maincor}.
\end{rem}

We finish this section by observing that a combination of \cite[Theorem 6.10.1, Theorem 7.4.1]{Le13} implies that $ \BH_{\tilde{\muhat},r}(0,q^{1/2})$ has non-negative coefficients, thus we have
 \begin{corollary} The weight polynomial $E(\FM_{\muhat,\bfr}^*;q)$ has non-negative coefficients. \end{corollary}

This motivates the following

\begin{conj}\label{purcon} The mixed Hodge structure on the cohomology of  $\FM_{\muhat,\bfr}^*$ is pure.
\end{conj}

If all poles in $\mathbf{C}$ are of order one this is proven in \cite[Theorem 2.2.6]{HLV11} using the description of $\FM_{\muhat,\bfr}^*$ as a quiver variety. Moreover if there is only one higher order pole, \cite[Theorem 9.11]{boalch-simply} also identifies $\FM_{\muhat,\bfr}^*$ with a quiver variety, thus Conjecture~\ref{purcon} follows in this case too.   In Section \ref{s:qmv}, we obtain a quiver like description of $\FM_{\muhat,\bfr}^*$ for poles of any order, giving more evidence and possible strategy for Conjecture \ref{purcon}.

\section{Quivers with multiplicities and their associated varieties} \label{s:qmv}

The greater part of this section is devoted to a generalization of the theorem of Crawley-Boevey \cite[Theorem 1]{CB03}, which realizes an additive fusion product of coadjoint orbits for $\GL_n(\K)$, where $\K$ is a fixed perfect ground field (the primary examples the reader should have in mind are where $\K = \C$ or $\K$ is a finite field), as a quiver variety associated to a star-shaped quiver.  The generalization we will achieve is to replace the $\GL_n(\K)$-coadjoint orbits with those for the non-reductive group $\GL_n(\RR_m)$ of Section \ref{s:jd} and to replace quivers with ``quivers with multiplicities'' or ``weighted quivers'' which are defined in Section \ref{s:qmvdefnot} below.  The main theorem is stated as Theorem \ref{mstar=q}. It states that an open de Rham space as defined in Section \ref{s:addfusprod} may be realized as a variety associated to a weighted star-shaped quiver for an appropriate choice of multiplicities.

The main novelty that enters when one introduces multiplicities is that the groups one obtains are no longer reductive.  As in the usual quiver case, one wants to define a variety associated to a quiver with multiplicities as an algebraic symplectic quotient.  In the approach we take here, we simply define this quotient as the spectrum of the appropriate ring of invariants.  Of course, this makes good sense as an affine scheme, however, without the reductivity hypothesis, we are not guaranteed that this ring of invariants is a finitely generated $\K$-algebra.  For the star-shaped quivers described in Section \ref{s:star}, we show that we do, in fact, get finite generation.  Essentially, what we show is that the preimage of the moment map is a trivial principal bundle for the unipotent radical of the group (Proposition \ref{trivbundleq}), the proof of which occupies Section \ref{s:redstages}. 

Let us also mention, that there seem to be several difficulties in applying the general theory of non-reductive GIT as developed by B\'erczy, Doran, Hawes and Kirwan \cite{BHFD16} directly to our situation: The varieties we consider are affine and the unipotent radical does not seem to admit a suitable $\mathbb{G}_m$-grading.

In Section \ref{s:coadjorb}, we explain how a certain symplectic quotient of a ``leg''-shaped quiver (out of which one builds a star-shaped quiver) yields a coadjoint orbit for the group $\GL_n(\RR_m)$, generalizing Boalch's explanation of Crawley-Boevey's theorem \cite[Lemma 9.10]{boalch-simply}.  This was done in the holomorphic category for ``short legs'' in \cite[Lemma 3.7]{YA10}; ours is an algebraic version, along the lines described above.

Quivers with multiplicities  have been introduced in \cite{YA10} for purposes quite similar to ours (namely, the description of moduli spaces of irregular meromorphic connections), and indeed, the quivers of interest in that paper are a special case of those considered here \cite[\S6.2]{YA10}.  We would like to point out that the quotients referred to there are either taken in the holomorphic category (when restricted to the stable locus, in the sense defined there) or simply set-theoretic \cite[Definition 3.3]{YA10}.  The results of Sections \ref{s:redstages} and \ref{s:coadjorb} show that these quotients indeed make sense algebraically.

Finally, we put together the results in Section \ref{s:qmvodr}, relating the varieties constructed in this section to the open de Rham spaces of Section \ref{s:odr}.  Quivers with multiplicities are also discussed by \cite{GeissLeclercSchroeer2017} in order to generalize certain classical results known for finite-dimensional algebras associated to symmetric Cartan matrices to the case where the Cartan matrix is only symmetrizable.  It is well-known that given a simply laced affine Dynkin diagram, one can associate a $2$-dimensional quiver variety (depending on some parameters), which will in fact be a gravitational instanton (i.e., carry a complete hyperk\"ahler metric).  In a parallel generalization, the construction above allows us to construct quiver varieties for non-simply laced Dynkin diagrams; this is explained in Section \ref{s:nonslDD}.  In the next section, we will see that these also admit complete hyperk\"ahler metrics (see Theorem \ref{t:odRhK} and Remark \ref{r:hKrmks}\eqref{r:higherorderhK}).

\subsection{Definitions, notation and a statement for star-shaped quivers} \label{s:qmvdefnot}
 
\subsubsection{Definitions, notation and a finite generation criterion}

Here, we discuss a generalization of the definition of quiver representations to allow for multiplicities at the vertices.  This was first done in \cite{YA10}. We follow the approach of \cite[\S1.4]{GeissLeclercSchroeer2017} and \cite{yamakawanotes}, using the greatest common divisor of neighbouring multiplicities in the definition of a representation.

Let $\K$ be a field.  For a positive integer $m \in \Z_{\geq 1}$, we will denote by $\RR_m$ the truncated polynomial ring
\begin{align*}
\RR_m := \K[z]/(z^m),
\end{align*}
so that $R_1 = \K$.  Observe that if $m | \ell$, then there is a natural inclusion of $\K$-algebras $\RR_m \hookrightarrow R_\ell$, which makes $R_\ell$ a free $\RR_m$-module of rank $\ell/m$.

As in Section \ref{s:jd}, we will use the groups $\GL_n(\RR_m)$ and $\GL_n^1(\RR_m)$ and their respective Lie algebras $\gl_n(\RR_m)$ and $\gl_n^1(\RR_m)$.  However, since the value of $n$ will vary, we will not shorten these to $\G$ or $\G_m$ etc.  We do adopt the identifications
\begin{align*}
\gl_n(\RR_m)^\vee = z^{-m} \gl_n(\RR_m)
\end{align*}
via the trace-residue pairing \eqref{trp}.

As usual, a \emph{quiver} $Q = (Q_0, Q_1, h, t)$ is a finite directed graph, i.e., $Q_0$ and $Q_1$ are finite sets and one has head and tail maps $h$, $t : Q_1 \to Q_0$.  By a \emph{set $\m$ of multiplicities for $Q$} we will mean an element $\m \in \Z_{\geq 1}^{Q_0}$.  The pair $(Q, \m)$ of a quiver and a set of multiplicities $\m$ can be referred to as a \emph{quiver with multiplicities} or a \emph{weighted quiver}.

A \emph{dimension vector} is also an element $\n \in \Z_{\geq 1}^{Q_0}$.  The space of representations of $(Q, \m)$ for a given dimension vector $\n$ is defined as
\begin{align*}
\Rep(Q, \m, \n) := \bigoplus_{\alpha \in Q_1} \Hom_{\RR_\alpha} \left( \RR_{m_{t(\alpha)}}^{\oplus n_{t(\alpha)}}, \RR_{m_{h(\alpha)}}^{\oplus n_{h(\alpha)}} \right)
\end{align*}
where we have made the abbreviation $\RR_\alpha := \RR_{(m_{t(\alpha)}, m_{h(\alpha)})}$. Here $(m_{t(\alpha)}, m_{h(\alpha)})$ denotes the greatest common divisor of $m_{t(\alpha)}$ and of $m_{h(\alpha)}$.  In the case that $m_i = 1$ for all $i \in Q_0$, this is the usual definition of the space of representations $\Rep(Q, \n)$ for $Q$ with dimension vector $\n$.

For a triple $(Q, \m, \n)$ we define the group
\begin{align*}
G_{Q, \m, \n} = G_{\m, \n} := \prod_{i \in Q_0} \GL_{n_i}(\RR_{m_i}).
\end{align*}
We will denote its Lie algebra by $\g_{\m, \n}$.  One has an action of $G_{\m, \n}$ on $\Rep(Q, \m, \n)$ as in the usual case:  for $(g_i) \in G_{\m, \n}$ and $(\varphi_a) \in \Rep(Q, \m, \n)$, 
\begin{align*}
(g_i) \cdot (\varphi_\alpha) = \left( g_{h(\alpha)} \varphi_a g_{t(\alpha)}^{-1} \right).
\end{align*}
The main difference between this and the case without multiplicities is that $G_{\m, \n}$ is not reductive if there is some $i \in Q_0$ with $m_i \geq 2$.

As usual, we denote by $\overline{Q}$ the doubled quiver:  i.e., $\overline{Q}_1 := Q_1 \sqcup Q_1'$, where $Q_1' := \{ \alpha' \, : \, \alpha \in Q_1 \}$ and $h(\alpha') = t(\alpha)$, $t(\alpha') = h(\alpha)$ for all $\alpha \in Q_1$.  Then
\begin{align} \label{e:RepQbar}
\Rep(\overline{Q}, \m, \n) & = \bigoplus_{\alpha \in Q_1} \Hom_{\RR_\alpha} \left( \RR_{m_{t(\alpha)}}^{\oplus n_{t(\alpha)}}, \RR_{m_{h(\alpha)}}^{\oplus n_{h(\alpha)}} \right) \oplus \Hom_{\RR_\alpha} \left(  \RR_{m_{h(\alpha)}}^{\oplus n_{h(\alpha)}}, \RR_{m_{t(\alpha)}}^{\oplus n_{t(\alpha)}} \right)\\
& = T^* \Rep(Q, \m, \n). \nonumber
\end{align}
Thus, we have the usual situation of a cotangent bundle, which has a canonical symplectic form, and an action of a group on the base, for which the induced action on the cotangent bundle is Hamiltonian.  We will be working with the corresponding moment map 
\begin{align*}
\mu : \Rep(\overline{Q},\m, \n) \to \g_{\m, \n}^\vee
\end{align*}
for the $G_{\m, \n}$-action. An explicit formula can be derived as in the case without multiplicities (see for example \cite[Lemma 3.1]{Lo14}) and is given by
\begin{align} \label{e:QMmm}
\mu(p,q) = \left( z^{-m_i}\sum_{\alpha \in h^{-1}(i)} p_\alpha q_\alpha -  z^{-m_i}\sum_{\beta \in t^{-1}(i)} q_\beta p_\beta \right)_{i \in Q_0},
\end{align} 
where we identify for each $i\in Q_0$ the dual $\g_{m_i,n_i}^\vee$ with $z^{-m_i}\g_{m_i,n_i}$ as in \eqref{e:gmvee}. We will usually drop the factor $z^{-m_i}$ in the computations below if no confusion arises. 


\begin{defn} \label{d:qmv}
Let us fix an element $\bm{\gamma} \in \g_{\m, \n}^\vee$ whose $G_{\m, \n}$-coadjoint orbit is a singleton.  One defines the \emph{reduced quiver scheme associated to $Q(\m, \n)$ at $\bm{\gamma}$} as
\begin{align} \label{e:qmv}
\mathcal{Q}_{\bm{\gamma}} := \spec \left( \K[\mu^{-1}(\bm{\gamma})]^{G_{\m, \n}} \right).
\end{align}
In this formula $\K[\mu^{-1}(\bm{\gamma})]$ denotes the coordinate ring of the affine variety $\mu^{-1}(\bm{\gamma})$, and $\K[\mu^{-1}(\bm{\gamma})]^{G_{\m, \n}}$ its subring of $G_{\m, \n}$-invariants.
\end{defn}

\begin{rmk} 
By definition $\mathcal{Q}_{\bm{\gamma}}$ is an affine scheme over $\K$.  However, as $G_{\m, \n}$ is not reductive, $\K[\mu^{-1}(\bm{\gamma})]^{G_{\m, \n}}$ is not a priori a finitely generated $\K$-algebra, therefore we cannot say that $\mathcal{Q}_{\bm{\gamma}}$ is an affine variety (possibly non-reduced) without further justification.
\end{rmk}

In view of the preceding remark, we now give a criterion which is sufficient to conclude that $\mathcal{Q}_{\bm{\gamma}}$ is an affine variety.  Then in Section \ref{s:star} we will describe a class of quivers with multiplicities and conditions on $\bm{\gamma}$ for which this criterion is fulfilled.

Observe that the group $G_{\m, \n}$ is a semi-direct product:  there is a surjective morphism of algebraic groups
\begin{align*}
G_{\m, \n} \to G_{\n},
\end{align*}
where 
\begin{align*}
G_\n := \prod_{i \in Q_0} \GL_{n_i}(\K)
\end{align*}
is the usual group associated to the quiver $Q$ and dimension vector $\n$. We will call the kernel $G_{\m, \n}^1$.  This surjection splits, with $\GL_{n_i}(\K)$ being the subgroup of ``constant'' group elements in $\GL_{n_i}(\RR_{m_i})$, and then taking the product of these inclusions (cf.\ Section \ref{s:jd}).  We may thus write
\begin{align*}
G_{\m, \n} = G_{\n} \ltimes G_{\m, \n}^1.
\end{align*}

With this decomposition, we may similarly decompose the Lie algebra $\g_{\m, \n}$ and its dual as direct sums:
\begin{align} \label{e:DS}
\g_{\m, \n} & = \g_\n \oplus \g_{\m, \n}^1 &  \g_{\m, \n}^\vee & = \g_\n^\vee \oplus (\g_{\m, \n}^1)^\vee. 
\end{align}
For $\bm{\gamma} \in \g_{\m, \n}^\vee$, we will write $\bm{\gamma}_{\res} \in \g_{\n}^\vee$ and $\bm{\gamma}_{\irr} \in (\g_{\m, \n}^1)^\vee$ for its components.  Likewise, the moment map \eqref{e:QMmm} will also have components $\mu_{\res}$ and $\mu_{\irr}$ corresponding to $\g_\n^\vee$ and $(\g_{\m, \n}^1)^\vee$, respectively.  Of course, each of these components is a moment map for the restriction of the action to the respective subgroup.

\begin{lem} \label{l:muirrtriv}
Suppose that there is a $G_\n$-invariant closed subvariety $M_{\bm{\gamma}} \subseteq \mu_{\irr}^{-1}( \bm{\gamma}_{\irr})$ and a $G_{\m,  \n}^1$-equivariant isomorphism
\begin{align*}
\mu_{\irr}^{-1}( \bm{\gamma}_{\irr}) \arsim G_{\m, \n}^1 \times M_{\bm{\gamma}} ,
\end{align*}
where $G_{\m, \n}^1$ acts by left multiplication on the first factor (and trivially on $M_{\bm{\gamma}}$); in other words, $\mu_{\irr}^{-1}( \bm{\gamma}_{\irr})$ is a trivial (left) principal $G_{\m, \n}^1$-bundle over $M_{\bm{\gamma}}$.  Then $\mathcal{Q}_{\bm{\gamma}}$ is an affine algebraic symplectic variety, and hence we will often refer to it as a \emph{quiver variety}.
\end{lem}

\begin{rmk}
By ``symplectic variety'', we mean that it is an algebraic symplectic manifold along its smooth locus.  We are making no claims about the possible singular locus.
\end{rmk}

\begin{proof}
We deal only with the finite generation of $\K[\mu^{-1}(\bm{\gamma})]^{G_{\m, \n}}$; the symplectic structure arises as usual on the smooth locus \cite[Theorem 1]{MW74}.  Since $G_{\m, \n}^1$ is normal in $G_{\m, \n}$, it is not hard to see that $\mu_{\res}$ is constant on $G_{\m, \n}^1$-orbits in $\Rep(\overline{Q}, \m, \n)$.  We have that
\begin{align*}
\mu^{-1}( \bm{\gamma}) = \mu_{\res}^{-1}( \bm{\gamma}_{\res}) \cap \mu_{\irr}^{-1}( \bm{\gamma}_{\irr}) \cong \mu_{\res}^{-1}( \bm{\gamma}_{\res}) \cap \left( G_{\m, \n}^1 \times M_{\bm{\gamma}} \right).
\end{align*}
Using the fact just mentioned, we obtain an isomorphism
\begin{align*}
\mu^{-1}( \bm{\gamma}) \cong G_{\m, \n}^1 \times \left( \mu_{\res}^{-1}( \bm{\gamma}_{\res}) \cap M_{\bm{\gamma}} \right).
\end{align*}
Now write $X := \mu_{\res}^{-1}( \bm{\gamma}_{\res}) \cap M_{\bm{\gamma}}$ for the second factor and note that this is an affine algebraic $G_\n$-variety.  From this,
\begin{align*}
\K[ \mu^{-1}( \bm{\gamma}) ]^{G_{\m, \n}} = \left( \left( \K[G_{\m, \n}^1] \otimes_\K \K[X] \right)^{G_{\m, \n}^1} \right)^{G_{\n}} = \K[X]^{G_{\n}},
\end{align*}
and since $G_\n$ is reductive, this is a finitely generated $\K$-algebra.
\end{proof}

\begin{rmk}
The idea in the proof comes from the procedure in the general theory of Hamiltonian reduction known as ``reduction in stages''. When taking a symplectic quotient by a group with a semi-direct product decomposition, one obtains the same quotient by first reducing by the normal subgroup and then by the quotient group. See, for example, \cite[\S4.2]{MMOPR}.
\end{rmk}

\subsubsection{Star-shaped quivers with multiplicities constant on the legs} \label{s:star}

We now describe a class of quivers with multiplicities and give conditions on the choice of $\bm{\gamma} \in \g_{\m, \n}^\vee$ for which the hypotheses of Lemma \ref{l:muirrtriv} will hold.  Let $Q$ be a star-shaped quiver as on \cite[p.340]{CB03},\footnote{The same diagram also appears at \cite[p.27]{Yamakawa2008} and at \cite[p.348, Figure 1]{HLV11}.}:

\vspace{10pt}

\begin{center}
\unitlength 0.1in
\begin{picture}( 45.0000, 15.4500)( 11.1000,-17.0000)
%
\special{pn 8}%
\special{ar 1376 1010 70 70  0.0000000 6.2831853}%
%
\special{pn 8}%
\special{ar 1946 410 70 70  0.0000000 6.2831853}%
%
\special{pn 8}%
\special{ar 2946 410 70 70  0.0000000 6.2831853}%
%
\special{pn 8}%
\special{ar 5540 410 70 70  0.0000000 6.2831853}%
%
\special{pn 8}%
\special{ar 1946 810 70 70  0.0000000 6.2831853}%
%
\special{pn 8}%
\special{ar 2946 810 70 70  0.0000000 6.2831853}%
%
\special{pn 8}%
\special{ar 5540 810 70 70  0.0000000 6.2831853}%
%
\special{pn 8}%
\special{ar 1946 1610 70 70  0.0000000 6.2831853}%
%
\special{pn 8}%
\special{ar 2946 1610 70 70  0.0000000 6.2831853}%
%
\special{pn 8}%
\special{ar 5540 1610 70 70  0.0000000 6.2831853}%
%
\special{pn 8}%
\special{pa 1890 1560}%
\special{pa 1440 1050}%
\special{fp}%
\special{sh 1}%
\special{pa 1440 1050}%
\special{pa 1470 1114}%
\special{pa 1476 1090}%
\special{pa 1500 1088}%
\special{pa 1440 1050}%
\special{fp}%
%
\special{pn 8}%
\special{pa 2870 410}%
\special{pa 2020 410}%
\special{fp}%
\special{sh 1}%
\special{pa 2020 410}%
\special{pa 2088 430}%
\special{pa 2074 410}%
\special{pa 2088 390}%
\special{pa 2020 410}%
\special{fp}%
%
\special{pn 8}%
\special{pa 3720 410}%
\special{pa 3010 410}%
\special{fp}%
\special{sh 1}%
\special{pa 3010 410}%
\special{pa 3078 430}%
\special{pa 3064 410}%
\special{pa 3078 390}%
\special{pa 3010 410}%
\special{fp}%
\special{pa 3730 410}%
\special{pa 3010 410}%
\special{fp}%
\special{sh 1}%
\special{pa 3010 410}%
\special{pa 3078 430}%
\special{pa 3064 410}%
\special{pa 3078 390}%
\special{pa 3010 410}%
\special{fp}%
%
\special{pn 8}%
\special{pa 2870 810}%
\special{pa 2020 810}%
\special{fp}%
\special{sh 1}%
\special{pa 2020 810}%
\special{pa 2088 830}%
\special{pa 2074 810}%
\special{pa 2088 790}%
\special{pa 2020 810}%
\special{fp}%
%
\special{pn 8}%
\special{pa 2870 1610}%
\special{pa 2020 1610}%
\special{fp}%
\special{sh 1}%
\special{pa 2020 1610}%
\special{pa 2088 1630}%
\special{pa 2074 1610}%
\special{pa 2088 1590}%
\special{pa 2020 1610}%
\special{fp}%
%
\special{pn 8}%
\special{pa 3730 810}%
\special{pa 3020 810}%
\special{fp}%
\special{sh 1}%
\special{pa 3020 810}%
\special{pa 3088 830}%
\special{pa 3074 810}%
\special{pa 3088 790}%
\special{pa 3020 810}%
\special{fp}%
\special{pa 3740 810}%
\special{pa 3020 810}%
\special{fp}%
\special{sh 1}%
\special{pa 3020 810}%
\special{pa 3088 830}%
\special{pa 3074 810}%
\special{pa 3088 790}%
\special{pa 3020 810}%
\special{fp}%
%
\special{pn 8}%
\special{pa 3730 1610}%
\special{pa 3020 1610}%
\special{fp}%
\special{sh 1}%
\special{pa 3020 1610}%
\special{pa 3088 1630}%
\special{pa 3074 1610}%
\special{pa 3088 1590}%
\special{pa 3020 1610}%
\special{fp}%
\special{pa 3740 1610}%
\special{pa 3020 1610}%
\special{fp}%
\special{sh 1}%
\special{pa 3020 1610}%
\special{pa 3088 1630}%
\special{pa 3074 1610}%
\special{pa 3088 1590}%
\special{pa 3020 1610}%
\special{fp}%
%
\special{pn 8}%
\special{pa 5466 410}%
\special{pa 4746 410}%
\special{fp}%
\special{sh 1}%
\special{pa 4746 410}%
\special{pa 4812 430}%
\special{pa 4798 410}%
\special{pa 4812 390}%
\special{pa 4746 410}%
\special{fp}%
%
\special{pn 8}%
\special{pa 5466 810}%
\special{pa 4746 810}%
\special{fp}%
\special{sh 1}%
\special{pa 4746 810}%
\special{pa 4812 830}%
\special{pa 4798 810}%
\special{pa 4812 790}%
\special{pa 4746 810}%
\special{fp}%
%
\special{pn 8}%
\special{pa 5466 1610}%
\special{pa 4746 1610}%
\special{fp}%
\special{sh 1}%
\special{pa 4746 1610}%
\special{pa 4812 1630}%
\special{pa 4798 1610}%
\special{pa 4812 1590}%
\special{pa 4746 1610}%
\special{fp}%
%
\special{pn 8}%
\special{pa 1880 840}%
\special{pa 1450 990}%
\special{fp}%
\special{sh 1}%
\special{pa 1450 990}%
\special{pa 1520 988}%
\special{pa 1500 972}%
\special{pa 1506 950}%
\special{pa 1450 990}%
\special{fp}%
%
\special{pn 8}%
\special{pa 1900 460}%
\special{pa 1430 960}%
\special{fp}%
\special{sh 1}%
\special{pa 1430 960}%
\special{pa 1490 926}%
\special{pa 1468 922}%
\special{pa 1462 898}%
\special{pa 1430 960}%
\special{fp}%
%
\special{pn 8}%
\special{sh 1}%
\special{ar 1946 1010 10 10 0  6.28318530717959E+0000}%
\special{sh 1}%
\special{ar 1946 1210 10 10 0  6.28318530717959E+0000}%
\special{sh 1}%
\special{ar 1946 1410 10 10 0  6.28318530717959E+0000}%
\special{sh 1}%
\special{ar 1946 1410 10 10 0  6.28318530717959E+0000}%
%
\special{pn 8}%
\special{sh 1}%
\special{ar 4056 410 10 10 0  6.28318530717959E+0000}%
\special{sh 1}%
\special{ar 4266 410 10 10 0  6.28318530717959E+0000}%
\special{sh 1}%
\special{ar 4456 410 10 10 0  6.28318530717959E+0000}%
\special{sh 1}%
\special{ar 4456 410 10 10 0  6.28318530717959E+0000}%
%
\special{pn 8}%
\special{sh 1}%
\special{ar 4056 810 10 10 0  6.28318530717959E+0000}%
\special{sh 1}%
\special{ar 4266 810 10 10 0  6.28318530717959E+0000}%
\special{sh 1}%
\special{ar 4456 810 10 10 0  6.28318530717959E+0000}%
\special{sh 1}%
\special{ar 4456 810 10 10 0  6.28318530717959E+0000}%
%
\special{pn 8}%
\special{sh 1}%
\special{ar 4056 1610 10 10 0  6.28318530717959E+0000}%
\special{sh 1}%
\special{ar 4266 1610 10 10 0  6.28318530717959E+0000}%
\special{sh 1}%
\special{ar 4456 1610 10 10 0  6.28318530717959E+0000}%
\special{sh 1}%
\special{ar 4456 1610 10 10 0  6.28318530717959E+0000}%
\put(19.7000,-2.4500){\makebox(0,0){$[1,1]$}}%
\put(29.7000,-2.4000){\makebox(0,0){$[1,2]$}}%
\put(55.7000,-2.5000){\makebox(0,0){$[1,l_1]$}}%
\put(19.7000,-6.5500){\makebox(0,0){$[2,1]$}}%
\put(29.7000,-6.4500){\makebox(0,0){$[2,2]$}}%
\put(55.7000,-6.5500){\makebox(0,0){$[2,l_2]$}}%
\put(19.7000,-17.8500){\makebox(0,0){$[n,1]$}}%
\put(29.7000,-17.8500){\makebox(0,0){$[n,2]$}}%
\put(55.7000,-17.8500){\makebox(0,0){$[n,l_n]$}}%
\put(12.4500,-10.1000){\makebox(0,0){$0$}}%
%
\special{pn 8}%
\special{sh 1}%
\special{ar 2950 1010 10 10 0  6.28318530717959E+0000}%
\special{sh 1}%
\special{ar 2950 1210 10 10 0  6.28318530717959E+0000}%
\special{sh 1}%
\special{ar 2950 1410 10 10 0  6.28318530717959E+0000}%
\special{sh 1}%
\special{ar 2950 1410 10 10 0  6.28318530717959E+0000}%
\end{picture}%
\end{center}
with $d$ legs and the $i$-th leg of length $l_i$.  Thus, the vertex set is
\begin{align} \label{e:starvert}
Q_0 = \{ 0 \} \cup \bigcup_{i=1}^d \{ [i,1], \ldots, [i,l_i] \}.
\end{align}
For notational convenience, for $1 \leq i \leq d$, we will often write $[i, 0]$ for $0$.  Then for $1 \leq i \leq d$ and $1 \leq j \leq l_i$, the doubled quiver $\overline{Q}$ will have one arrow from $[i,j-1]$ to $[i,j]$ and one from $[i,j]$ to $[i,j-1]$.

For each $1 \leq i \leq d$, we fix $m_i \in \Z_{>0}$ and choose the multiplicity vector $\m$ with
\begin{align} \label{e:starmult}
\m_0 & = 1 & \m_{[i,j]} & = m_i;
\end{align}
thus, the multiplicity is fixed in each leg (away from the central vertex $0$).  The dimension vector $\n$ will be
\begin{align} \label{e:stardim}
\n_0 & = n & \n_{[i,j]} & = n_{i,j}
\end{align}
with $n > n_{i,1} > \ldots > n_{i,l_i} > 0$ for $1 \leq i \leq d$, so that the dimensions are decreasing as one moves away from the central vertex on each leg.

Observe that with this,
\begin{align*}
G_{\m, \n} = \GL_n(\K) \times \prod_{i=1}^d \prod_{j=1}^{l_i} \GL_{n_{i,j}}(\RR_{m_i});
\end{align*}
hence
\begin{align*}
\g_{\m, \n} & = \gl_n(\K) \oplus \bigoplus_{i=1}^d \bigoplus_{j=1}^{l_i} \gl_{n_{i,j}}(\RR_{m_i}) & \g_{\m, \n}^\vee & = z^{-1} \gl_n(\K) \oplus \bigoplus_{i=1}^d \bigoplus_{j=1}^{l_i} z^{-m_i} \gl_{n_{i,j}}(\RR_{m_i}).
\end{align*}
We choose $\bm{\gamma} \in \g_{\m, \n}$ of the form
\begin{align*}
\bm{\gamma} & = (\gamma^0 \1_n, \gamma^{[i,j]} \1_{n_{i,j}})_{ \substack{1 \leq i \leq d \\ 1 \leq j \leq l_i}} & & \textnormal{ with } & \gamma^0 & \in z^{-1} \cdot \K, \quad \gamma^{[i,j]} \in z^{-m_i} \RR_{m_i}.
\end{align*}
We further assume that for each $1 \leq i \leq d$ and for $1 \leq j \leq k \leq l_i$,
\begin{align} \label{e:cassum2}
 z^{m_i} \left( \gamma^{[i,j]} + \cdots + \gamma^{[i,k]} \right) \in \RR_{m_i}^\times  .
\end{align}

\begin{prop} \label{p:qmvstar}
With $(Q, \m, \n)$ and $\bm{\gamma}$ as above, $\mathcal{Q}_{\bm{\gamma}}$ is an affine algebraic symplectic variety.
\end{prop}

In the next subsection, we will show that the hypotheses of Proposition \ref{p:qmvstar} imply those of Lemma \ref{l:muirrtriv}, which then provides a proof of the proposition.

\subsection{Proof of Proposition \ref{p:qmvstar}: construction of a section} \label{s:redstages}

Here we prove the following

\begin{prop}\label{trivbundleq}
Let $(Q, \m, \n)$ and $\bm{\gamma} \in \g_{\m, \n}^\vee$ be as in Proposition \ref{p:qmvstar}.  Then there is a $G_\n$-invariant closed subvariety $M_{\bm{\gamma}} \subseteq \mu_{\irr}^{-1}( \bm{\gamma}_{\irr})$ and a $G_{\m, \n}^1$-equivariant isomorphism
\begin{align} \label{e:G1section}
\mu_{\irr}^{-1}( \bm{\gamma}_{\irr}) \arsim G_{\m, \n}^1 \times M_{\bm{\gamma}},
 \end{align}
 as in the hypothesis of Lemma \ref{l:muirrtriv}.
\end{prop}

\paragraph{Simplified case: \texorpdfstring{$d = 1$}{Lg}.}

For notational simplicity, we will first assume that $d = 1$, so that we have a quiver with a single leg.  This allows us to drop the index $i$ from the notation in \eqref{e:starvert}, \eqref{e:starmult} and \eqref{e:stardim}.  Thus, the doubled quiver $\overline{Q}$ is the following
\begin{align} \label{e:oneleg}
\xymatrix{ 
\circ \ar@/^/[r]^{p^1} \ar@{-}[r] & \bullet \ar@/^/[l]^{q^1} \ar@/^/[r]^{p^2} \ar@{-}[r] & \bullet \ar@/^/[l]^{q^2} & \cdots & \bullet \ar@/^/[r]^{p^{l-1}} \ar@{-}[r] & \bullet \ar@/^/[l]^{q^{l-1}} \ar@/^/[r]^{p^l} \ar@{-}[r] & \bullet \ar@/^/[l]^{q^l}
},
\end{align}
with the vertices are labeled $0, \ldots, l$ from left to right, multiplicity vector $\m = (1, m, \ldots, m)$ and dimension vector $\n = (n, n_1, \ldots, n_l)$ with $n > n_1 > \cdots > n_l$.  To avoid having to write out things separately for the vertex $0$, we will often write $n_0 := n$.

Also, 
\begin{align*}
\bm{\gamma} & = \left( \gamma^0 \1_n, \gamma^j \1_{n_j} \right)_{1 \leq j \leq l} & & \textnormal{ with } & \gamma^0 \in z^{-1} \cdot \K, \quad \gamma^j \in z^{-m} \RR_m.
\end{align*}
The condition \eqref{e:cassum2} on $\bm{\gamma}$ reads
\begin{align} \label{e:cassum1}
z^m \left( \gamma^j + \cdots + \gamma^k \right) \in \RR_m^\times
\end{align}
for $1 \leq j \leq k \leq l$.

\paragraph{Explicit description of the groups, Lie algebras and their duals.}  Let us now be explicit about the groups that are involved.  One has
\begin{align} \label{e:Gmn}
G_{\m, \n} & = \GL_n(\K) \times \prod_{i=1}^l \GL_{n_i}(\RR_m) & G_{\m, \n}^1 & = \prod_{i=1}^l \GL_{n_i}^1(\RR_m),
\end{align}
with respective Lie algebras
\begin{align*}
\g_{\m, \n} & = \gl_n(\K) \oplus \prod_{i=1}^l \gl_{n_i}(\RR_m) & \g_{\m, \n}^1 & = \bigoplus_{i=1}^l \gl_{n_i}^1(\RR_m)
\end{align*}
and duals
\begin{align*}
\g_{\m, \n}^\vee & = z^{-1} \gl_n(\K) \oplus \prod_{i=1}^l z^{-m} \gl_{n_i}(\RR_m) & \left( \g_{\m, \n}^1 \right)^\vee & = \bigoplus_{i=1}^l z^{-m} \gl_{n_i}(\RR_m) \big/ z^{-1} \gl_{n_i}(\RR_m).
\end{align*}
One has direct sum decompositions as in \eqref{e:DS} and the projection maps to each factor simply omits the ``irregular'' part or the ``residue'' term, respectively, as in \eqref{pis}, \eqref{e:irrproj}.  As before, we write $\bm{\gamma}_{\res}$ and $\bm{\gamma}_\irr$ for the images of $\bm{\gamma}$ under the respective projections to $\g_{\n}$ and to $\g_{\m, \n}^1$.

\paragraph{Explicit description of the space of representations $\Rep(\overline{Q}, \m, \n)$ and group actions.}

Explicitly, the space $\Rep(\overline{Q}, \m, \n)$ is given by
\begin{align*}
\Hom_\K(\K^n, \RR_m^{\oplus n_1}) \oplus \Hom_\K(\RR_m^{\oplus n_1}, \K^n) \oplus  \bigoplus_{i=2}^l \left( \Hom_{\RR_m}(\RR_m^{n_{i-1}}, \RR_m^{n_i}) \oplus \Hom_{\RR_m}( \RR_m^{n_i}, \RR_m^{n_{i-1}}) \right).
\end{align*}
We will write elements of $\Rep(\overline{Q}, \m, \n)$ as pairs $(p,q)$, where $p = (p^1, \ldots, p^l)$ and $q = (q^1, \ldots, q^l)$ are each themselves $s$-tuples, where
\begin{align*}
p^1 & \in \Hom_\K(\K^n, \RR_m^{\oplus n_1}) & p^i & \in \Hom_{\RR_m}(\RR_m^{n_{i-1}}, \RR_m^{n_i}), \quad 2 \leq i \leq l \\
q^1 & \in \Hom_\K(\RR_m^{\oplus n_1}, \K^n) & q^i & \in \Hom_{\RR_m}(\RR_m^{n_i}, \RR_m^{n_{i-1}}), \quad 2 \leq i \leq l.
\end{align*}

Since $ \Hom_{\RR_m}(\RR_m^{r}, \RR_m^{s}) \cong  \Hom_{\K}(\K^{r}, \K^{s}) \otimes_{\K} \RR_m$, we will think of the $p^i$ as elements
\begin{align*}
p^i \in M_{n_i \times n_{i-1}}(\RR_m), \quad 1 \leq i \leq l
\end{align*}
and dually of the $q^i$ as Laurent polynomials
\begin{align*}
q^i \in z^{-m} M_{n_{i-1} \times n_i}(\RR_m), \quad 1 \leq i \leq l.
\end{align*}

In this notation, an element $g = (g^0, \ldots, g^l) \in G_{\m, \n}$ acts on $(p,q) \in \Rep(\overline{Q}, \m, \n)$ by
\begin{align} \label{e:Gmnact} 
g \cdot (p,q) & = \big( g^1 p^1 (g^0)^{-1}, g^2 p^2 (g^1)^{-1}, \ldots, g^l p^l (g^{l-1})^{-1}, \nonumber \\
& \quad \quad g^0 q^1 (g^1)^{-1}, g^1 q^2 (g^2)^{-1}, \ldots, g^{l-1} q^l (g^l)^{-1}  \big).
\end{align}
Of course, here $g^0 \in \GL_n(\K)$ and $g^i \in \GL_{n_i}(\RR_m)$.

Later we will need to be even more explicit, and so we will write 
\begin{align*}
p^i & = p_0^i + z p_1^i + \cdots + z^{m-1} p_{m-1}^i & q^i & = \frac{q_m^i}{z^m} + \cdots + \frac{q_1^i}{z}
\end{align*}
with $p_j^i \in M_{n_i \times n_{i-1}}(\K)$, $q_j^i \in M_{n_{i-1} \times n_i}(\K)$.  

When we will need to evaluate moment maps or group actions, we will need to regard, for example, the product $p^i q^i$ as an element of $\gl_{n_i}(\RR_m)^\vee = z^{-m} \gl_{n_i}(\RR_m)$.  We do this by  multiplying $p^i$ and $q^i$ as matrices of Laurent polynomials and truncating the terms of degree $\geq 0$. Explicitly,
\begin{align} \label{e:pqex}
p^i q^i = \frac{p_0^i q_m^i}{z^m} + \frac{p_0^i q_{m-1}^i + p_1^i q_m^i}{z^{m-1}} + \cdots + \frac{p_0^i q_1^i + p_1^i q_2^i + \cdots + p_{m-1}^i q_m^i}{z} \in \gl_{n_i}(\RR_m)^\vee.
\end{align}
Of course, products of the form $q^i p^i$ or $g^{i-1} q^i (g^i)^{-1}$ as in \eqref{e:Gmnact} are written similarly.  Products such as $g^i p^i (g^{i-1})^{-1}$ in \eqref{e:Gmnact} will be considered as multiplication of the relevant matrices with entries in $\RR_m$.

We will also write
\begin{align*}
\gamma^i = \frac{\gamma_m^i}{z^m} + \cdots + \frac{\gamma_1^i}{z}, \quad 1 \leq i \leq l
\end{align*}
with $\gamma_j^i \in \K$.  In particular, the assumption \eqref{e:cassum1} implies that $\gamma_m^i \in \K^\times$ for all $1 \leq i \leq l$.

\paragraph{Explicit description of the moment maps.}

The moment map $\mu : \Rep(\overline{Q}, \m, \n) \to \g_{\m, \n}^\vee$ has the explicit expression
\begin{align} \label{e:mu0}
\mu(p,q) = (-q^1 p^1, p^1 q^1 - q^2 p^2, \ldots, p^{l-1} q^{l-1} - q^l p^l, p^l q^l) \in \g_{\m, \n}^\vee.
\end{align}
Composing with the projections arising from the direct sum decomposition \eqref{e:DS}, we get moment maps 
\begin{align*}
& \mu_{\res} : \Rep(\overline{Q}, \m, \n) \to \g_{\n}^\vee & & \mu_{\irr} : \Rep(\overline{Q}, \m, \n) \to (\g_{\m, \n}^1)^\vee 
\end{align*}
which are, in fact, the moment maps for the restricted $G_{\n}$- and $G_{\m, \n}^1$-actions, respectively.

We are primarily interested in the preimage $\mu_{\irr}^{-1}(\bm{\gamma}_{\irr})$. The point $(p,q)$ lies in this preimage if and only if (in other words the moment map condition translates to)
\begin{align} \label{e:mu0inv}
p^i q^i & = q^{i+1} p^{i+1} + \gamma^i \1_{n_i}, \quad 1 \leq i \leq l-1 & p^l q^l & = \gamma^l \1_{n_l}.
\end{align}
We recall that $(\g_{\m, \n}^1)^\vee$ has no component corresponding to the vertex $0$. Thus if we write out an explicit expression as in \eqref{e:pqex}, we should ignore any contributions coming from the residue terms, since we are considering these equations in $(\g_{\m, \n}^1)^\vee$.

\paragraph{Definition of \texorpdfstring{$M_{\bm{\gamma}}$}{Lg} and its \texorpdfstring{$G_{\n}$}{Lg}-invariance.}

We now proceed with defining the subvariety $M_{\bm{\gamma}}$ appearing in Lemma \ref{l:muirrtriv}, where $(Q, \m, \n)$ is as in Proposition \ref{p:qmvstar} in the special case where $d = 1$.  We start by defining 
\begin{align*}
M_{\bm{\gamma}}^0 := \mu_{\irr}^{-1}(\bm{\gamma}_\irr)
\end{align*}
and define inductively a nested sequence of subvarieties of $M_{\bm{\gamma}}^0$ by
\begin{align} \label{Mlambdai}
M_{\bm{\gamma}}^i & := \left\{ (p,q) \in M_{\bm{\gamma}}^{i-1} \, : \, p^i q_m^i = \gamma_m^i \1_{n_i} + q_m^{i+1} p_0^{i+1} \right\} \nonumber \\
M_{\bm{\gamma}} = M_{\bm{\gamma}}^l & := \left\{ (p,q) \in M_{\bm{\gamma}}^{l-1} \, : \, p^l q_m^l = \gamma_m^l \1_{n_l} \right\}.
\end{align}

\begin{rmk} \label{r:Mlambdai}
On the left side of the defining equation for $M_{\bm{\gamma}}^i$, the matrix $p^i q_m^i$ is a priori a $\gl_{n_i}(\K)$-valued polynomial in $z$ of degree $\leq m-1$, but the right hand side is, in fact, a constant matrix.  So the defining condition is equivalent to the further conditions
\begin{align} \label{e:Mlamex}
p_1^i q_m^i = \cdots = p_{m-1}^i q_m^i = 0 \in \gl_{n_i}(\K).
\end{align} 
The fact that $M_{\bm{\gamma}}$ is an affine variety is thus clear.

Furthermore, from this description, it is easy to see that $M_{\bm{\gamma}}$ is a $G_{\n}$-invariant subvariety of $\Rep(\overline{Q}, \m, \n)$ using \eqref{e:Gmnact} and the appropriate expressions in \eqref{e:pqex}:  the vanishing conditions \eqref{e:Mlamex} imposed by \eqref{Mlambdai} are left unchanged.
\end{rmk}

\begin{rmk} \label{C}
 Equation \eqref{e:pq} below implies that $z^m p^i q^i \in \GL_{n_i}(\RR_m)$ for $1 \leq i \leq l$.  In particular, the constant term, namely $p_0^i q_m^i$, must lie in $\GL_{n_i}(\K)$.  This fact is crucial in the construction of the isomorphism \eqref{e:G1section}; see Lemma \ref{F} below.
\end{rmk}

We now establish some properties of the $M_{\bm{\gamma}}^i$ with respect to the action of $G_{\m, \n}^1$ and its subgroups.  Let us consider each $\GL_{n_i}^1(\RR_m)$ for $1 \leq i \leq l$ as a subgroup of $G_{\m, \n}^1$ via the obvious inclusion in \eqref{e:Gmn}. We will also use the subgroups 
\begin{align*}
(G_{\m, \n}^1)^i & := \prod_{j=i}^l \GL_{n_j}^1(\RR_m) \subset G_{\m, \n}^1, 
\end{align*}
noting that we have a (decreasing) chain of inclusions
\begin{align*}
(G_{\m, \n}^1)^l \subset (G_{\m, \n}^1)^{l-1} \subset \cdots \subset (G_{\m, \n}^1)^2 \subset (G_{\m, \n}^1)^1 = G_{\m, \n}^1.
\end{align*}

\begin{lem} \label{D}
$M_{\bm{\gamma}}^i$ is invariant under $(G_{\m, \n}^1)^{i+1}$.
\end{lem}

\begin{proof}
Let $(p,q) \in M_{\bm{\gamma}}^i$ and $r = (1, \ldots, 1, r^{i+1}, \ldots, r^l) \in (G_{\m, \n}^1)^{i+1}$.  Then 
\begin{align*}
r \cdot (p,q) = & \big(p^1, q^1, \ldots, p^i, q^i, r^{i+1} p^{i+1}, q^{i+1} (r^{i+1})^{-1}, r^{i+2} p^{i+2} (r^{i+1})^{-1}, r^{i+1} q^{i+2} (r^{i+2})^{-1}, \\
& \ldots, r^l p^l (r^{l-1})^{-1} , r^{l-1} q^l (r^l)^{-1} \big).
\end{align*}
Observe that $r \cdot (p,q) \in M_{\bm{\gamma}}^{i-1}$ since the condition for this to hold depends only on the components $p^1, q^1, \ldots, p^i, q^i$, which are unchanged under the action of $r$.  We only need to check the condition for $M_{\bm{\gamma}}^i$, which involves the terms $q_m^{i+1}$, $p_0^{i+1}$. It suffices to see that
\begin{align*}
(r^{i+1} p^{i+1})_0 & = p_0^{i+1} & \big(q^{i+1} (r^{i+1})^{-1} \big)_m & = q_m^{i+1}.
 \end{align*}
But since $r^{i+1}\in \GL_{n_{i+1}}^1(\RR_m)$, the constant term of $r^{i+1}$ and $(r^{i+1})^{-1}$ are both $\1_{n_{i+1}}$.
\end{proof}

\begin{lem} \label{E}
Let $r^i \in \GL_{n_i}^1(\RR_m)$ and $(p, q) \in M_{\bm{\gamma}}^i$.  If $r^i \cdot (p,q) \in M_{\bm{\gamma}}^i$ then $r^i = \1_{n_i}$.
\end{lem}

\begin{proof}
Write $r^i = \1_{n_i} + z r_1^i + \cdots + z^{m-1} r_{m-1}^i$.  Then 
\begin{align} \label{e:ripq}
r^i \cdot (p,q) = \left( p^1, q^1, \ldots, r^i p^i, q^i(r^i)^{-1}, p^{i+1} (r^i)^{-1}, r^i q^{i+1}, \ldots, p^l, q^l \right).
\end{align}
The condition for $r^i \cdot (p,q) \in M_{\bm{\gamma}}^i$ is
\begin{align*}
r^i p^i \left( q^i(r^i)^{-1} \right)_m = \gamma_m^i \1_{n_i} + (r^i q^{i+1})_m \left( p^{i+1} (r^i)^{-1} \right)_0,
\end{align*}
but 
\begin{align*}
\left( q^i (r^i)^{-1} \right)_m & = q_m^i & (r^i q^{i+1})_m & = q_m^{i+1} & \left( p^{i+1} (r^i)^{-1} \right)_0 & = p_0^{i+1},
\end{align*}
so we need
\begin{align*}
r^i p^i q_m^i = \gamma_m^i \1_{n_i} + q_m^{i+1} p_0^{i+1} = p_0^i q_m^i,
\end{align*}
where the last equality follows from the definition of $M_{\bm{\gamma}}^i$. But by Remark \ref{C}, $p_0^i q_m^i \in \GL_{n_i}(\K) \subset \GL_{n_i}(\RR_m)$ and hence this may be regarded as an equation in $\GL_{n_i}^1(\RR_m)$ and hence $r^i = \1_{n_i}$.
\end{proof}

\paragraph{Construction of the isomorphism \eqref{e:G1section}.}

\begin{lem} \label{F}
For $1 \leq i \leq l$, there exists a morphism $\varphi_i : M_{\bm{\gamma}}^{i-1} \to \GL_{n_i}^1(\RR_m)$ such that for all $(p,q) \in M_{\bm{\gamma}}^{i-1}$ we have $\varphi_i(p,q) \cdot (p, q) \in M_{\bm{\gamma}}^i$.
\end{lem}

\begin{proof}
Let $(p,q) \in M_{\bm{\gamma}}^{i-1}$ and $r^i \in \GL_{n_i}^1(\RR_m)$.  Since $r^i \in (G_{\m, \n}^1)^i$, $r^i \cdot (p,q) \in M_{\bm{\gamma}}^{i-1}$ and so we need only check the condition for $M_{\bm{\gamma}}^i$.  We recall the expression for $r^i \cdot (p,q)$ from \eqref{e:ripq} and observe that
\begin{align*}
r^i p^i q_m^i = & \ p_0^i q_m^i + z (p_1^i + r_1^i p_0^i) q_m^i + z^2 ( p_2^i + r_1^i p_1^i + r_2^i p_0^i) q_m^i + \cdots + z^{m-1} ( p_{m-1}^i + r_1^i p_{m-2}^i + \cdots \\
& \ + r_{m-1}^i p_0^i) q_m^i.
\end{align*}
The condition we want is that all the non-constant (with respect to $z$) terms  vanish (Remark \ref{r:Mlambdai}).  But now, by Remark \ref{C}, $p_0^i q_m^i$ is invertible, and so starting with the coefficient of $z$ above, we may solve for $r_1^i = - p_1^i (p_0^i q_m^i)^{-1}$ so that this coefficient vanishes.  It is then clear that we may successively solve for $r_2^i$, $\ldots$, $r_{m-1}^i$ algebraically as functions of $p$ and $q$ to eliminate the remaining non-constant terms.  This produces $\varphi_i$ with the stated property.
\end{proof}

\begin{cor} \label{G}
\begin{enumerate}[(a)]
\item For $r^i \in \GL_{n_i}^1(\RR_m)$ and $x \in M_{\bm{\gamma}}^{i-1}$ we have $\varphi_i(r^i \cdot x) = \varphi_i(x) (r^i)^{-1}$.
\item For $\bar{r} \in (G_{\m, \n}^1)^{i+1}$ and $x \in M_{\bm{\gamma}}^{i-1}$ we have $\varphi_i(\bar{r} \cdot x) = \varphi_i(x)$.
\end{enumerate}
\end{cor}

\begin{proof}
\begin{enumerate}[(a)]
\item By definition, $\varphi_i(x) \cdot x \in M_{\bm{\gamma}}^i$, so also $\varphi_i(r^i \cdot x) \cdot (r^i \cdot x) \in M_{\bm{\gamma}}^i$, but the latter is equal to $\varphi_i( r^i \cdot x) \cdot r^i \cdot \varphi_i(x)^{-1} \cdot \big( \varphi_i(x) \cdot x \big)$.  Since $\varphi_i(x) \cdot x \in M_{\bm{\gamma}}^i$, by Lemma \ref{E}, $\varphi_i(r^i \cdot x) \cdot r^i \cdot \varphi_i(x)^{-1} = \1_{n_i}$.
\item We have $\varphi_i(\bar{r} \cdot x) \cdot (\bar{r} \cdot x) \in M_{\bm{\gamma}}^i$ but since $\varphi_i( \bar{r} \cdot x) \in \GL_{n_i}^1(\RR_m)$ and $\bar{r} \in (G_{\m, \n}^1)^{i+1}$ we have
\[\varphi_i(\bar{r} \cdot x) \bar{r} = \bar{r}\varphi_i(\bar{r} \cdot x)\]
and hence $\bar{r} \cdot \big( \varphi_i(\bar{r} \cdot x) \cdot x \big) \in M_{\bm{\gamma}}^i$, but then also $\varphi_i(\bar{r} \cdot x) \cdot x \in M_{\bm{\gamma}}^i$ by Lemma \ref{D}.  So again Lemma \ref{E} yields $\varphi_i(\bar{r} \cdot x) = \varphi_i(x)$. \qedhere
\end{enumerate}
\end{proof}

\begin{cor} \label{H}
There is a $(G_{\m, \n}^1)^i$-equivariant isomorphism $\sigma_i : M_{\bm{\gamma}}^{i-1} \to \GL_{n_i}^1(\RR_m) \times M_{\bm{\gamma}}^i$, where the action of $(r^i, \bar{r}) \in \GL_{n_i}^1(\RR_m) \times (G_{\m, \n}^1)^{i+1} = (G_{\m, \n}^1)^i$ on $(s,y) \in \GL_{n_i}^1(\RR_m) \times M_{\bm{\gamma}}^i$ is
\begin{align*}
(r^i, \bar{r}) \cdot (s, x) = (r^i s, \bar{r} \cdot x).
\end{align*}
\end{cor}

\begin{proof}
We define $\sigma_i : M_{\bm{\gamma}}^{i-1} \to \GL_{n_i}^1(\RR_m) \times M_{\bm{\gamma}}^i$ by $\sigma_i(x) = \big( \varphi_i(x)^{-1}, \varphi_i(x) \cdot x \big)$, which is well-defined by Lemma \ref{F}.  The inverse $\tau_i : \GL_{n_i}^1(\RR_m) \times M_{\bm{\gamma}}^i \to M_{\bm{\gamma}}^{i-1}$ is simply given by the action $\tau_i(s,y) = s \cdot y$.  This is well-defined since $\GL_{n_i}^1(\RR_m) \subset (G_{\m, \n}^1)^i$ and $M_{\bm{\gamma}}^i \subseteq M_{\bm{\gamma}}^{i-1}$ and $(G_{\m, \n}^1)^i$ acts on $M_{\bm{\gamma}}^{i-1}$ by Lemma \ref{D}.  It is clear that $\tau_i \circ \sigma_i = \1_{M_{\bm{\gamma}}^{i-1}}$.  Now, if $(s,y) \in \GL_{n_i}^1(\RR_m) \times M_{\bm{\gamma}}^i$, then $\varphi_i(s \cdot y) = \varphi_i(y) s^{-1} = s^{-1}$ since $\varphi_i(y) \cdot y \in M_{\bm{\gamma}}^i$ but already $y \in M_{\bm{\gamma}}^i$, so one uses Lemma \ref{E} to see that $\varphi_i(y) = \1_{n_i}$.  Then
\begin{align*}
\sigma_i \circ \tau_i(s,y) = \sigma_i(s \cdot y) = \left( \varphi_i(s \cdot y)^{-1}, \varphi_i(s \cdot y) \cdot (s \cdot y) \right) = \left( s, s^{-1} \cdot (s \cdot y) \right) = (s,y).
\end{align*}

Using Corollary \ref{G}(a), it is easy to see that $\sigma_i$ is $\GL_{n_i}^1(\RR_m)$-equivariant.  Thus, it suffices to show that it is $(G_{\m, \n}^1)^{i+1}$-equivariant.  Let $\bar{r} \in (G_{\m, \n}^1)^{i+1}$, $x \in M_{\bm{\gamma}}^{i-1}$:
\begin{align*}
\sigma_i( \bar{r} \cdot x) & = \left( \varphi_i( \bar{r} \cdot x)^{-1},\varphi_i( \bar{r} \cdot x) \cdot (\bar{r} \cdot x) \right) = \left( \varphi_i(x)^{-1}, \varphi_i(x) \cdot \bar{r} \cdot x \right) = \bar{r} \cdot \left( \varphi_i(x)^{-1}, \varphi_i(x) \cdot x \right) \\
& = \bar{r} \cdot \sigma_i(x). & \qedhere
\end{align*}
\end{proof}

\paragraph{Conclusion in the case \texorpdfstring{$d =1$}{Lg}.}
We can now put the $\sigma_i$ of Corollary \ref{H} into a $G_{\m, \n}^1$-equivariant isomorphism
\begin{align*}
\mu_{\irr}^{-1}(\bm{\gamma}_{\irr}) = & \ M_{\bm{\gamma}}^0 \arsim \GL_{n_1}^1(\RR_m) \times M_{\bm{\gamma}}^1 \arsim \GL_{n_1}^1(\RR_m) \times \GL_{n_2}^1(\RR_m) \times M_{\bm{\gamma}}^2 \arsim \cdots \\
& \arsim \GL_{n_1}^1(\RR_m) \times \cdots \times \GL_{n_l}^1(\RR_m) \times M_{\bm{\gamma}}^l = G_{\m, \n}^1 \times M_{\bm{\gamma}}. 
\end{align*}
It was already noted in Remark \ref{r:Mlambdai} that $M_{\bm{\gamma}}$ is $G_{\n}$-invariant.  These are the hypotheses of Lemma \ref{l:muirrtriv}, and so we may conclude in the case $d = 1$.

\paragraph{The general case of \texorpdfstring{$d$}{Lg} legs.}

Now, consider the situation of Proposition \ref{p:qmvstar}, where the number of legs $d \in \Z_{>0}$ in the quiver $Q$ is arbitrary.  Since the multiplicity at the central vertex $0$ is $m_0 = 1$, $\mu_{\irr}$ has no component at $0$ and thus 
\[\mu_{\irr}^{-1}( \bm{\gamma}_{\irr}) \cong \prod_{i=1}^d \mu_{\irr}^{-1}( \bm{\gamma}_{\irr})_i,\]
where $\mu_{\irr}^{-1}( \bm{\gamma}_{\irr})_i$ denotes the moment map preimage for the $i$-th leg.
Likewise, $G_{\m, \n}^1$ factors as a product of the groups \eqref{e:Gmn} for each leg and the action on $\mu_{\irr}^{-1}( \bm{\gamma}_{\irr})$ is just the product action. Thus if we set 
\[ M_{\bm{\gamma}} = \prod_{i=1}^d M_{\bm{\gamma},i},\]
with $M_{\bm{\gamma},i}$ is in \eqref{Mlambdai} for the $i$-th leg, we get a $G_{\m,  \n}^1$-equivariant isomorphism
\[
\mu_{\irr}^{-1}( \bm{\gamma}_{\irr}) \arsim G_{\m, \n}^1 \times M_{\bm{\gamma}}.\]

Finally the $G_{\n}$-invariance of $M_{\bm{\gamma}}$ follows from the corresponding statement for each leg (see Remark \ref{r:Mlambdai}) since there we already included the action of $\GL_n(\K)$ at the central vertex. Applying Lemma \ref{l:muirrtriv} completes the proof of Proposition \ref{p:qmvstar}.

\subsection{Coadjoint orbits} 
\label{s:coadjorb}

Here, we discuss the relationship between coadjoint orbits for the group $\GL_n(\RR_m)$ for a fixed $m \geq 1$ and varieties associated to quivers with multiplicities, where the underlying quiver is a single leg.  It may thus help the reader to refer back to the diagram \eqref{e:oneleg}.  What will be true is that coadjoint orbits of certain diagonal elements $C \in \t_m^\vee$ in $\gl_n(\RR_m)^\vee$ \eqref{e:diagelt} can be realized as symplectic reductions of the spaces $\Rep(\overline{Q}, \m, \n)$ that we considered in the case $d = 1$ in Section \ref{s:redstages}.  To be able to make a precise statement, we will first need to explain the conditions on the coadjoint orbits and set some notation.

We will take $C \in \t_m^\vee$, and suppose that it is written in the form \eqref{e:C}. We will make the further assumption that \eqref{e:cassum} holds.  For such a $C$, we wish to describe a quiver $Q$, which will be a leg, as well as some data on it, from which we will recover $\O(C)$.  The quiver will be the same as in \eqref{e:oneleg}, having $l+1$ nodes and $2l$ arrows, and will have the same multiplicity vector $\m$, with multiplicity $1$ at the vertex $0$ and all other vertices receiving multiplicity $m$.  The dimension vector $\n$ will be defined by taking $n_l := \lambda_l, \ n_{l-1} := \lambda_l + \lambda_{l-1}$ and so on, with $(\lambda_0, \ldots, \lambda_l)$ given as in \eqref{e:C}.  

\begin{rmk} \label{r:regcassum}
Observe that if $C$ is a regular formal type (recall Definition \ref{d:formaltype}), then any representative of $C$ in any local coordinate $z$ satisfies \eqref{e:cassum}.
\end{rmk}

From the data in \eqref{e:C}, we set
\begin{align} \label{e:lambdadef}
\gamma^0 & := c^0 & \gamma^i & := c^i - c^{i-1}, \quad 1 \leq i \leq l.
\end{align}
Then \eqref{e:cassum} implies that \eqref{e:cassum1} holds.

The statement that we want is that the $\GL_n(\RR_m)$-coadjoint orbit of $C$ as above is given by a symplectic reduction of $\Rep(\overline{Q}, \m, \n)$ by a subgroup of $G_{\m, \n}$.  The subgroup in question is that we obtain by leaving out the group $\GL_n(\K)$ corresponding to the vertex $0$, namely,
\begin{align*}
G_{\m, \n, 0} := \prod_{i=1}^l \GL_{n_i}(\RR_m).
\end{align*}
We write $\g_{\m, \n, 0}$ for its Lie algebra.  The reason the vertex $0$ in \eqref{e:oneleg} was drawn empty is because we want to consider only the symplectic quotient by $G_{\m, \n, 0}$.

Of course, $G_{\m, \n}^1$ is a normal subgroup of $G_{\m,\n,0}$ with quotient
\begin{align*}
G_{\n, 0} = \prod_{i=1}^l \GL_{n_i}(\K)
\end{align*}
which is precisely the group associated to the underlying quiver with dimension vector $\n$, ignoring the multiplicities, where again we are omitting the group $\GL_n(\K)$ corresponding to the vertex $0$.  Thus, we want to take a symplectic quotient by the semi-direct product
\begin{align} \label{e:semidirprod}
G_{\m, \n, 0} = G_{\n, 0} \ltimes G_{\m, \n}^1.
\end{align}

From the inclusion $\g_{\m, \n, 0} \subseteq \g_{\m, \n}$ we have a natural surjection of the duals
\begin{align*}
\g_{\m, \n}^\vee \to \g_{\m, \n, 0}^\vee,
\end{align*}
and the moment map for the $G_{\m, \n, 0}$-action on $\Rep(\overline{Q}, \m, \n)$ is given by the composition
\begin{align*}
\mu_0 : \Rep(\overline{Q}, \m, \n) \xrightarrow{\mu} \g_{\m, \n}^\vee \to \g_{\m, \n, 0}^\vee.
\end{align*}
We will consider the element
\begin{align*}
\bm{\gamma}_0 := (\gamma^1 \1_{n_1}, \ldots, \gamma^l \1_{n_l}) \in \g_{\m,\n,0}^\vee,
\end{align*}
and define the symplectic quotient 
\begin{align*}
\Rep(\overline{Q}, \m, \n) /\!\!/_{\bm{\gamma}_0} G_{\m, \n, 0}  := \spec \left( \K[\mu_0^{-1}( \bm{\gamma}_0)]^{G_{\m, \n,0}} \right) .
\end{align*}
For less burdensome notation, we will often abbreviate the left hand side of the above to $\Rep /\!\!/ G_{\m, \n, 0}$.  Observe that the assumption \eqref{e:cassum}, the arguments of Section \ref{s:redstages} and Lemma \ref{l:muirrtriv} already show that $\Rep /\!\!/ G_{\m, \n, 0}$ is an affine symplectic variety.  We will write 
\begin{align*}
\pi : \mu_0^{-1}(\bm{\gamma}_0) \to \Rep /\!\!/ G_{\m, \n, 0}
\end{align*}
for the quotient map; since this is defined as a GIT quotient, this is a categorical quotient.

\begin{prop} \label{p:coadjqmv}
Suppose $C \in z^{-m} \t(\RR_m) \subseteq \gl_n(\RR_m)^\vee$ is written in the form \eqref{e:C} and satisfies \eqref{e:cassum} and that $Q$, $\m$, and $\n$ are given as above.  Then $\Rep(\overline{Q}, \m, \n) /\!\!/_{\bm{\gamma}_0} G_{\m, \n, 0}$ admits a $\GL_n(\RR_m)$-action and there is a $\GL_n(\RR_m)$-equivariant isomorphism
\begin{align} \label{e:coadjqmv}
\Rep(\overline{Q}, \m, \n) /\!\!/_{\bm{\gamma}_0} G_{\m, \n, 0} \arsim \O(C).
\end{align}
\end{prop}

The following is a slight generalization of \cite[Proposition D.1]{boalch-simply} which will be important in the proof of the proposition.

\begin{lem}  \label{l:A1}
Let $R$ be a commutative local ring and let $m \leq n \in \Z_{>0}$, $p \in M_{m \times n}(R)$, $q \in M_{n \times m}(R)$ be such that $pq \in \GL_m(R)$. Then 
\begin{align} \label{e:simmat}
qp \quad \text{and} \quad \mat{ 0_{n-m}}{}{}{ pq }
\end{align}
are conjugate in $\GL_n(R)$.
\end{lem}

\begin{proof}
Observe that $R^n = \ker p \oplus \im q$. Indeed, given $v \in R^n$, let $w := (pq)^{-1} pv \in R^m$.  Then
\begin{align*}
p ( v - qw) = 0
\end{align*}
hence $v = (v - qw) + qw \in \ker p + \im q$, i.e., $R^n = \ker p + \im q$.  Furthermore, the sum is direct, for if $v \in \ker p \cap \im q$, say $v = qw$ with $w \in R^m$ and $pv = pqw = 0$, then $w = 0$ and hence $v = 0$.

Furthermore, the assumption that $pq$ is invertible also implies that $\ker p$ and $\im q$ are free $R$-modules.  This can be seen via the Cauchy--Binet formula:  for a subset $I \subseteq \{ 1, \ldots, n\}$ of size $m$, one sets $\det_I p$ to be the determinant of the $m \times m$ submatrix of $p$ taking the columns with indices in $I$; one defines $\det_I q$ the same way, using columns instead of rows; then the formula states that
\begin{align*}
\mathrm{det} \ pq = \sum_I (\mathrm{det}_{I}  p) (\mathrm{det}_{I} q),
\end{align*}
where the sum is over all subsets of size $m$.  Since $\det pq \in R^\times$ and $R$ is local, there must be some $I$ with $\det_I p \in R^\times$ (otherwise, all the terms in the sum would lie in the maximal ideal and hence $\det pq$ could not be a unit).  Hence there exists $r \in \GL_m(R)$ for which the submatrix of $rp$ corresponding to $I$ is the identity matrix.  That is, the matrix $rp$ is in reduced row echelon form and one can find a basis of $\ker p = \ker rp$ as one does in a first-year linear algebra class.  A similar argument shows that the columns of $q$ are linearly independent over $R$ and hence already give a basis of $\im q$.

Now, we choose a basis of $R^n$ by taking the first $n-m$ vectors as a basis of $\ker p$ and the last $m$ vectors as the columns of $q$.  Then the fact that $R^n = \ker p \oplus \im q$ implies that the matrix obtained in this way lies in $\GL_n(R)$.  Writing $qp$ with respect to this basis gives the second matrix in \eqref{e:simmat}.
\end{proof}

With this, the proof of Proposition \ref{p:coadjqmv} follows the idea of \cite[Lemma 9.10]{boalch-simply}, which explains the proof of \cite[\S3]{CB03}.  However, there one can rely on usual results of linear algebra over fields, while in our case working with the orbits of the unipotent groups involved requires a little care, which makes the arguments somewhat longer.

\begin{proof}[Proof of Proposition \ref{p:coadjqmv}] 
For the reader's convenience, we recall the explicit expressions for the moment map \eqref{e:mu0}
\begin{align} \label{e:mu0actual}
\mu_0(p,q) = (p^1 q^1 - q^2 p^2, \ldots, p^{l-1} q^{l-1} - q^l p^l, p^l q^l) \in \g_{\m, \n, 0}^\vee = \bigoplus_{i=1}^l \gl_{n_i}(\RR_m)^\vee
\end{align}
and the $G_{\m, \n, 0}$-action: for $h = (h_1, \ldots, h_l) \in G_{\m, \n, 0}$, one has
\begin{align} \label{e:Gmn0act}
h \cdot (p,q) = ( h_1 p^1, q^1 h_1^{-1}, h_2 p^2 h_1^{-1}, h_1 q^2 h_2^{-1}, \ldots, h_l p^l h_{l-1}^{-1}, h_{l-1} q^l h_l^{-1}).
\end{align}

\paragraph{\texorpdfstring{$\GL_n(\RR_m)$}{Lg}-action on \texorpdfstring{$\Rep/\!\!/ G_{\m, \n, 0}$}{Lg}.} Of course $\Rep(\overline{Q}, \m, \n)$ admits a $\GL_n(\RR_m)$-action:  for $g \in \GL_n(\RR_m)$, one has
\begin{align} \label{e:GnRmact}
g \cdot (p,q) = (p^1g^{-1}, gq^1, p^2, q^2, \ldots, p^l, q^l).
\end{align}
From \eqref{e:mu0actual}, it is easy to check that $\mu_0^{-1}(\bm{\gamma}_0)$ is invariant under this action.  It is likewise easy to see that it commutes with the $G_{\m, \n, 0}$-action \eqref{e:Gmn0act}.  Thus, the action descends to the quotient $\Rep/\!\!/ G_{\m, \n, 0}$.

\paragraph{Definition of the isomorphism \texorpdfstring{$\Phi : \Rep/\!\!/ G_{\m, \n, 0} \to \O(C)$}{Lg}.}
We begin by defining a morphism $\widetilde{\Phi} : \mu_0^{-1}(\bm{\gamma}_0) \to \O(C)$ by
\begin{align} \label{e:quivertocoadjorbx}
(p,q) \mapsto q^1 p^1 + \gamma^0 \1_n.
\end{align}
For this to define a $\GL_n(\RR_m)$-equivariant morphism $\Phi : \Rep/\!\!/ G_{\m, \n, 0} \to \O(C)$, we need to verify three things:  first, a priori, $\widetilde{\Phi}$ takes values only in $\gl_n(\RR_m)^\vee$, so we need to see that it indeed takes values in $\O(C)$; second, we need to check that $\widetilde{\Phi}$ is $G_{\m, \n, 0}$-invariant; finally, one wants to see that $\widetilde{\Phi}$ is $\GL_n(\RR_m)$-equivariant.  The latter two statements are easy to check simply from their definitions:  \eqref{e:Gmn0act} for $G_{\m, \n, 0}$-invariance and \eqref{e:GnRmact} for $\GL_n(\RR_m)$-equivariance.  The first statement is a bit longer and so we will justify it in the next paragraph.

\paragraph{\texorpdfstring{$\widetilde{\Phi}$}{Lg} takes values in \texorpdfstring{$\O(C)$}{Lg}.}
For $0 \leq i \leq l$, we define the diagonal matrix
\begin{align*}
t^i & := \left[ \begin{array}{ccccc}
\gamma^i \1_{\lambda_i} & & & & \\
& \gamma^{i,i+1} \1_{\lambda_{i+1}} & & & \\
& & \ddots & & \\
& & & \gamma^{i,l-1} \1_{\lambda_{l-1}} & \\
& & & & \gamma^{i,l} \1_{\lambda_l} 
\end{array} \right] \in z^{-m} \gl_{n_i}(\RR_m) = \gl_{n_i}(\RR_m)^\vee,  
\end{align*}
where the $\gamma^i$ were defined in \eqref{e:lambdadef} and
\begin{align*}
\gamma^{i,j} & := \gamma^i + \cdots + \gamma^j, \quad 1 \leq i < j \leq l.
\end{align*}
Since \eqref{e:cassum} and hence \eqref{e:cassum1} hold, $z^m t^i \in \GL_{n_i}(\RR_m)$ for $1 \leq i \leq l$.
Also, we have
\begin{align} \label{e:tind}
\mat{ 0_{\lambda_i} }{}{}{ t^{i+1} } + \gamma^i \1_{n_i} & = t^i, \quad 0 \leq i \leq l-1 & t^0 & = C.
\end{align}

We observe that if $(p,q) \in \mu_0^{-1}(\bm{\gamma}_0)$, then for $1 \leq i \leq l$,
\begin{align} \label{e:pq}
p^i q^i \sim_{\GL_{n_i}(\RR_m)} t^i, 
\end{align}
where $\sim_{\GL_{n_i}(\RR_m)}$ means in the same $\GL_{n_i}(\RR_m)$ coadjoint orbit in $\gl_{n_i}(\RR_m)^\vee$, and
\begin{align} \label{e:qp}
q^i p^i \sim_{\GL_{n_{i-1}}(\RR_m)} \mat{ 0_{\lambda_{i-1}} }{}{}{ t^i }.
\end{align}

First, note that \eqref{e:qp} follows from \eqref{e:pq} and Lemma \ref{l:A1}.  Then \eqref{e:pq} is easy to see by decreasing induction on $i$.  For $i = l$, \eqref{e:pq} follows from the last component of the moment map condition $\mu_0(p,q) = \bm{\gamma}_0$, see \eqref{e:mu0actual}.  Now, for the inductive step, one has
\begin{align*}
p^i q^i = q^{i+1} p^{i+1} + \gamma^i \1_{n_i} \sim \mat{ 0_{\lambda_i} }{}{}{ t^{i+1} } + \gamma^i \1_{n_i} = t^i,
\end{align*}
the first equality being the $i$th component of the moment map \eqref{e:mu0actual}, the similarity \eqref{e:qp} and the last equality following directly from the definition of the $t^i$.

Finally, to show that $\widetilde{\Phi}$ takes values in $\O(C)$, we wish to show that for $(p,q) \in \mu_0^{-1}(\bm{\gamma})$, one has $q^1 p^1 + \gamma^0 \1_n \in \O(C)$.  This now follows from \eqref{e:qp} for $i = 1$ and \eqref{e:tind}.

\paragraph{Definition of the inverse \texorpdfstring{$\Psi : \O(C) \to \Rep/\!\!/ G_{\m, \n, 0}$}{Lg}.}

We start by defining a morphism $\Psi' : \GL_n(\RR_m) \to \mu_0^{-1}(\bm{\gamma}_0)$.  Of course, we can compose this with the projection $\pi : \mu_0^{-1}(\bm{\gamma}_0) \to \Rep /\!\!/ G_{\m, \n, 0}$ to obtain a map $\widetilde{\Psi} : \GL_n(\RR_m) \to \Rep /\!\!/ G_{\m, \n, 0}$.  Then, since the map $\eta : \GL_n(\RR_m) \to \O(C)$ of Lemma \ref{l:coadjorb}\eqref{l:coadjhomog} is a categorical quotient, in order to define $\Psi : \O(C) \to \Rep /\!\!/ G_{\m, \n, 0}$, it suffices to show that $\widetilde{\Psi}$ is $\L_{\lambda, m}$-invariant, by Lemma \ref{intor}\eqref{l:centralizers}.

We first define a tuple $(p,q)_C$ by
\begin{align} \label{e:pCqC}
p_C^i & := \left[ \begin{array}{cc}
0_{n_i \times \lambda_{i-1}} & \1_{n_i} \end{array} \right]
& q_C^i & := \left[ \begin{array}{c} 0_{\lambda_{i-1} \times n_i} \\ t^i 
\end{array} \right], \quad 1 \leq i \leq l.
\end{align}
With \eqref{e:tind} it is easy to check that 
\begin{align} \label{e:pqC}
(p,q)_C & \in \mu_0^{-1}(\bm{\gamma}_0) & & \textnormal{and} & \widetilde{\Phi}(p,q)_C & = C.
\end{align}
We now define $\Psi' : \GL_n(\RR_m) \to \mu_0^{-1}(\bm{\gamma}_0)$ using the action \eqref{e:GnRmact}
\begin{align*} 
g \mapsto g \cdot (p, q)_C.
\end{align*}

We will show that the resulting $\widetilde{\Psi}$ is $\L_{\lambda, m}$-invariant.  Let $f \in \L_{\lambda, m}$.  We may write
\begin{align*}
f = \diag(f_0, \ldots, f_l)
\end{align*}
with $f_i \in \GL_{\lambda_i}(\RR_m)$.  Furthermore, for $0 \leq i \leq l$, we will set 
\begin{align*}
f^i := \diag(f_i, \ldots, f_l) \in \GL_{n_i}(\RR_m),
\end{align*}
noting that $f^0 = f$.  Then it is easy to check that for $1 \leq i \leq l$,
\begin{align} \label{e:fswitch}
p_C^i (f^{i-1})^{-1} & = (f^i)^{-1} p_C^i & f^{i-1} q_C^i & = q_C^i f^i.
\end{align}
By an inductive argument using \eqref{e:fswitch}, it is straightforward to show that for $g \in \GL_n(\RR_m)$ and $f \in \L_{\lambda, m}$ as above,
\begin{align*}
(gf) \cdot (p,q)_C & = \left( f^1, \ldots, f^l\right)^{-1} \cdot \left( g \cdot ( p, q)_C \right),
\end{align*}
where the right hand side is the action of $G_{\m, \n, 0}$; in other words, $\Psi'(gf)$ and $\Psi'(f)$ lie in the same $G_{\m, \n, 0}$-orbit.  It follows that $\widetilde{\Psi}$ is $\L_{\lambda, m}$-invariant, and hence induces $\Psi : \O(C) \to \Rep/\!\!/G_{\m, \n, 0}$ with
\begin{align*}
\Psi \circ \eta = \widetilde{\Psi}.
\end{align*}

\paragraph{Verification that \texorpdfstring{$\Phi$}{Lg} and \texorpdfstring{$\Psi$}{Lg} are mutually inverse.}
We first check that $\Phi \circ \Psi = \1_{\O(C)}$.  Now, $\Phi \circ \Psi$ is the morphism induced via $\L_{\lambda, m}$-invariance from the map $\widetilde{\Phi} \circ \Psi' : \GL_n(\RR_m) \to \O(C)$, which is, by \eqref{e:pqC},
\begin{align*}
g \mapsto g \cdot (p,q)_C \mapsto \Ad_g C.
\end{align*}
But this is precisely $\eta$, as in Lemma \ref{l:coadjorb}\eqref{l:coadjhomog}, hence the induced map on the quotient must be the identity.

Furthermore, for $(p,q)_C$, the similarity relations in \eqref{e:pq} and \eqref{e:qp} are, in fact, equalities.  In particular, $q_C^1 p_C^1 + \gamma^0 \1_n = t^0 = C$ and hence it is easy to check that 
\begin{align*}
\Phi \circ \Psi(A) = A
\end{align*}
for all $A \in \O(C)$.

Finally, we show that $\Psi \circ \Phi = \1_{\Rep/\!\!/G_{\m, \n, 0}}$.  Let $(p,q) \in \mu_0^{-1}( \bm{\gamma}_0)$.  Then $\widetilde{\Phi}(p,q) = q^1 p^1 + \gamma^0 \1_n$; if this is $\Ad_g C$, for $g \in \GL_n(\RR_m)$, then 
\begin{align*}
\Psi' \circ \widetilde{\Phi}(p,q) = g \cdot (p,q)_C.
\end{align*}
Thus, to show that $\Psi \circ \Phi \circ \pi(p,q) = \pi(p,q)$, we will show that $(p,q)$ and $g\cdot (p,q)_C$ are in the same $G_{\m, \n, 0}$-orbit.  This is again an (increasing) induction.  Using \eqref{e:pqC}, we have
\begin{align*}
q^1 p^1 + \gamma_0 \1_n = \Ad_g C = g(q_C^1 p_C^1 + \gamma_0 \1_n)g^{-1}
\end{align*}
and hence we find
\begin{align} \label{e:qpC1}
q^1 p^1 = g q_C^1 p_C^1 g^{-1}
\end{align}
and multiplying by $p^1$ on the left and by $g$ on the right, we obtain
\begin{align} \label{e:gqp1}
p^1 q^1 p^1 g  = p^1 g q_C^1 p_C^1.
\end{align}
Let us now write
\begin{align*}
p^1 g =: \left[ \begin{array}{cc}
d_1 & h_1
\end{array} \right], 
\end{align*}
for some $d_1 \in M_{n_1 \times \lambda_0}(\RR_m)$ and $h_1 \in \gl_{n_1}(\RR_m)$.  Substituting this into \eqref{e:gqp1}, and using the explicit expressions for $p_C^1$ and $q_C^1$ \eqref{e:pCqC}, we get
\begin{align*}
\left[ \begin{array}{cc}
(p^1 q^1) d_1 & (p^1 q^1) h_1
\end{array} \right] = \left[ \begin{array}{cc}
0_{n_1 \times \lambda_0} & h_1 t^1 
\end{array} \right]
\end{align*}
Now, \eqref{e:pq} tells us that $z^m p^1 q^1$ is invertible and hence $d_1 = 0$.  Therefore
\[ p^1 g  = \left[ \begin{array}{cc}
0 & h_1
\end{array} \right] = h_1 p_C^1\]
or equivalently
\begin{align} \label{e:p1}
 p^1 = h_1 p_C^1 g^{-1}.
\end{align}

Since $z^m p^1 q^1$ is invertible, $p^1$ is of rank $n_1$ and multiplication by $g$ does not change this, so $h_1$ must also be of rank $n_1$ and hence $h_1 \in \GL_{n_1}(\RR_m)$.  Using this and substituting \eqref{e:p1} into \eqref{e:qpC1}, we can conclude that
\begin{align*}
q^1 = g q_C^1 h_1^{-1}.
\end{align*}
Therefore, 
\begin{align*}
(p,q) & = (h_1 p_C^1 g^{-1}, g q_C^1 h_1^{-1}, p^2, q^2, \ldots, p^l, q^l) \\
& = (h^1,1, \ldots, 1) \cdot (p_C^1 g^{-1}, gq_C^1, p^2 h_1, h_1^{-1} q^2, p^3, q^3, \ldots, p^l, q^l),
\end{align*}
thus, after relabeling $p^2$, $q^2$, $(p,q)$ is in the same $G_{\m, \n, 0}$-orbit as an element of the form $(p_C^1 g^{-1}, gq_C^1, p^2 , q^2, p^3, q^3, \ldots, p^l, q^l)$.

By induction, we may assume that $(p,q)$ is of the form
\begin{align*}
(p,q) = (p_C^1 g^{-1}, gq_C^1, p_C^2, q_C^2, \ldots, p_C^i, q_C^i, p^{i+1}, q^{i+1}, \ldots, p^l, q^l).
\end{align*}
Then the $i$th component of the moment map \eqref{e:mu0actual} gives
\begin{align*}
q^{i+1} p^{i+1} = p_C^i q_C^i - \gamma^i \1_{n_i} = \mat{ 0_{\lambda_i}}{}{}{ t^{i+1} };
\end{align*}
(in case $i = 1$, we have the product $(p_C^i g^{-1})(g q_C^1) = p_C^1 q_C^1$, and so this case yields the same equation).  Using the same argument as above, we write $p^{i+1}$ as a block matrix with two blocks and using the fact that $z^m p^{i+1} q^{i+1}$ is invertible, show that the square block is an invertible matrix $h_{i+1}$ and the other block is zero.  We then conclude that $p^{i+1} = h_{i+1} p_C^{i+1}$ and then that $q^{i+1} = q_C^{i+1} h_{i+1}^{-1}$.  Hence
\begin{align*}
(p,q) & = (p_C^1 g^{-1}, gq_C^1, p_C^2, q_C^2, \ldots, p_C^i, q_C^i, h_{i+1} p_C^{i+1}, q_C^{i+1} h_{i+1}^{-1}, \ldots, p^l, q^l) \\
& = (1, \ldots, h_{i+1}, \ldots, 1) \cdot (p_C^1 g^{-1}, gq_C^1, p_C^2, q_C^2, \ldots, p_C^{i+1}, q_C^{i+1}, p_C^{i+2} h_{i+1} , h_{i+1}^{-1} q_C^{i+2}, \ldots, p^l, q^l),
\end{align*}
and the induction hypothesis is satisfied for $i+1$.  Continuing in this fashion, we see that our original $(p,q)$ is in the $G_{\m, \n, 0}$-orbit of $g \cdot (p, q)_C$ and hence $\Psi \circ \Phi \circ \pi(p,q) = \pi(p,q)$.
\end{proof}

\subsection{Relation between open de Rham spaces and non-simply laced affine Dynkin diagrams} \label{s:qmvodr}

\subsubsection{Additive fusion product of coadjoint orbits and open de Rham spaces}

The reason for the emphasis on the quiver with multiplicities described in Section \ref{s:star} is to relate the corresponding variety to an additive fusion product of coadjoint orbits and hence open de Rham spaces. Suppose we are given a $d$-tuple $\m := (m_i)_{i=1}^d$ of positive integers and coadjoint orbits $\O(C^i)$, $1 \leq i \leq d$, for some diagonal elements $C^i \in \gl_n(R_{m_i})^\vee$, which we will take to be written in the form \eqref{e:C}.  We use this to define the data for a quiver with multiplicities:
\begin{enumerate}
\item For $1 \leq i \leq d$, the integer $l_i$ is defined as in \eqref{e:C} and the quiver $Q$  as in \eqref{e:starvert} with arrows defined immediately thereafter.
\item The tuple $\m$ of multiplicities can then be chosen as in \eqref{e:starmult}.
\item We define the dimension vector $\n$ as follows.  We set $n_0 := n$.  Again from \eqref{e:C}, for each $1 \leq i \leq d$, one gets a series of positive integers $\lambda_{[i,0]}, \ldots, \lambda_{[i,l_i-1]}$, and we set $n_{[i,k+1]} := n_{[i,k]} - \lambda_{[i,k]}$ with the convention $n_{[i,0]} = n$.  This defines $(n_{[i,1]}, \ldots, n_{[i,l_i]})$ for $1 \leq i \leq d$ and we use these to define the remaining entries of $\n$.  We have now defined $Q(\m, \n)$, and hence also $G_{\m, \n}$, $\g_{\m, \n}$, etc.
\item Finally, we define an element $\bm{\gamma} \in \g_{\m, \n}^\vee$.  Once again from \eqref{e:C} and \eqref{e:lambdadef}, for each $1 \leq i \leq d$, we obtain elements $\gamma^{[i,0]}, \ldots, \gamma^{[i, l_i]} \in z^{-m_i} R_{m_i}$; we take 
\begin{align*}
\gamma^{[i,j]} \1_{n_{[i,j]}} \in z^{-m_i} \gl_{n_{[i,j]}}(R_{m_i})
\end{align*}
to be the component of $\bm{\gamma}$ at all vertices except the central vertex $0$.  There, we take $\gamma^0 z^{-1} \1_n \in \gl_n(\K)^\vee$, where
\begin{align} \label{e:vertex0value}
\gamma^0 := \sum_{i=1}^d \res_{z=0} \gamma^{[i,0]}.
\end{align}
\item Furthermore, we will assume that for $1 \leq i \leq d$, \eqref{e:cassum} is satisfied.
\end{enumerate}

\begin{thm}  \label{mstar=q}
\label{t:qmvafs} With $Q(\m, \n)$ and $\bm{\gamma} \in \g_{\m, \n}$ chosen as above, one has an isomorphism of the associated quiver variety with the additive fusion product of coadjoint orbits
\begin{align} \label{e:quivertoorbits}
\mathcal{Q}_{\bm{\gamma}} \cong \left( \prod_{i=1}^d \O(C^i) \right) \bigg/\!\!\!\!\bigg/_0 \GL_n(\K).
\end{align}
In particular, if $C^1,\dots,C^d$ are regular generic we have $\mathcal{Q}_{\bm{\gamma}} \cong \FM^*_{\muhat,\bfr}$.  Furthermore, in terms of the quiver data, the dimension of $\mathcal{Q}_{\bm{\gamma}}$ is given by the formula 
\begin{align} \label{e:quivdimform}
\dim \mathcal{Q}_{\bm{\gamma}} = 2 \left( \sum_{i=1}^d m_i \sum_{k=1}^l n_{[i,k]} ( n_{[i,k-1]} - n_{[i,k]}) - n_0^2 + 1 \right).
\end{align}
\end{thm}

\begin{proof}
Taking each leg one at a time, Proposition \ref{p:coadjqmv} takes $(p,q) \in \mu^{-1}(\bm{\gamma})$ and gives us a $d$-tuple $(A^1, \ldots, A^d)$ with $A^i \in \O(C^i)$, $1 \leq i \leq d$; more explicitly \eqref{e:quivertocoadjorbx}, one has 
\begin{align*}
A^i = q^{[i,1]} p^{[i,1]} + \gamma^{[i,0]} \1_n.
\end{align*}
The moment map condition for the quiver at the vertex $0$ is
\begin{align*}
\pi_{\res} \left( -\sum_{i=1}^d q^{[i,1]} p^{[i,1]} \right) = \gamma^0 \1_n,
\end{align*}
and that for the additive fusion product, that is the right hand side of \eqref{e:quivertoorbits}, is
\begin{align*}
\pi_{\res} \left( \sum_{i=1}^d A^i \right) = \pi_{\res} \left( \sum_{i=1}^d q^{[i,1]} p^{[i,1]} + \gamma^{[i,0]} \1_n \right) = 0.
\end{align*}
So by definition \eqref{e:vertex0value} it is clear that one moment map condition is satisfied if and only if the other one is.  Finally, we remark that the remaining group action is the diagonal action of $\GL_n(\K)$, with $g \in \GL_n(\K)$ acting on $q^{[i,1]}$ as $g q^{[i,1]}$ and on $p^{[i,1]}$ by $p^{[i,1]} g^{-1}$; this clearly translates into conjugation on the $A^i$.

For the dimension formula, we can use the expression \eqref{e:quivertoorbits} and compute the total dimension by summing those of the coadjoint orbits $\O(C^i)$.  To obtain these in terms of the quiver data, we use the expression \eqref{e:coadjqmv}.  For the $i$th leg, the dimension of the space of representations for the arrows (going in opposite directions) joining the $(k-1)$th and $k$th nodes is $2m_i n_{[i,k]} n_{[i,k-1]}$; the dimension of the group at the node $[i,k]$ is $m_i n_{[i,k]}^2$.  Summing over the nodes on the $i$th leg and accounting for the preimage of central elements in the dual of the Lie algebra, we get
\begin{align*}
\dim \O(C^i) = 2 \left( m_i \sum_{k=1}^{l_i} n_{[i,k]} n_{[i,k-1]} - m_i \sum_{k=1}^{l_i} n_{[i,k]}^2 \right) = 2m_i \sum_{k=1}^{l_i} n_{[i,k]} (n_{[i,k-1]} - n_{[i,k]}).
\end{align*}
The term $-2(n_0^2-1) = -2(n^2-1)$ of course comes from the quotient by $P\GL_n = \GL_n/Z$.
\end{proof}

\subsubsection{Surfaces associated to non-simply laced affine Dynkin diagrams} \label{s:nonslDD}

It is well-known how to attach a smooth algebraic surface to a simply laced affine Dynkin diagram. Namely, given a simply laced affine root system of type $\mathcal{T}$ (where $\mathcal{T}$ is one of $\{\{\tilde{A}_i\}_{i\geq 1},\{\tilde{D}_i\}_{i\geq 4},\tilde{E}_6,\tilde{E}_7,\tilde{E}_8\}$), the type $\mathcal{T}$ ALE space can be constructed as a Nakajima quiver variety for the quiver, the affine Dynkin diagram, and with a suitable choice of dimension vectors \cite[\S 2]{nak94}. It is isomorphic to a resolution of a Kleinian singularity $\C^2/\Gamma$ for a finite $\Gamma\subset \SL_2$ corresponding to $\mathcal{T}$ via the McKay correspondence. It carries  natural {\em Asymptotically Locally Euclidean} hyperk\"ahler metrics \cite{Kr89} - hence the abbreviation ALE. In this section we will study
open de Rham spaces corresponding to non-simply laced affine Dynkin diagrams. Many of them will turn out to be isomorphic to ALE spaces, thanks to  Boalch's \cite[Theorem 9.11]{boalch-simply}. 

Let $Q=(Q_0,Q_1,h,t)$ be a quiver, $\n \in \Z_{>0}^{Q_0}$ a dimension vector and $\m\in \Z_{>0}^{Q_0}$ the multiplicity vector. As explained in \cite{yamakawanotes} this data is equivalent to the following symmetrizable generalized Cartan matrix $C=(c_{ij})_{i,j\in Q_0}$ defined by $c_{ii}=2$ and for $i\neq j$ $$c_{i,j}:=-\frac{n_i}{{\rm gcd}(n_i,n_j)} a_{i,j}$$ where $$a_{i,j}:=|\{a\in Q_1 \ |\ h(a)=i, t(a)=j \mbox{ or }  h(a)=j, t(a)=i\}|.$$ This we can record by a not necessarily simply-laced Dynkin diagram. 

Let now $\bm{\gamma}\in \g_{\m, \n}^\vee$. Then the quiver variety with multiplicity $Q_\gamma$ has dimension given by formula \eqref{e:quivdimform}.  Thus $\mathcal{Q}_{\bm{\gamma}}$ is a surface if and only if
\begin{align*}
\sum_{i=1}^d m_i \sum_{k=1}^l n_{[i,k]} ( n_{[i,k-1]} - n_{[i,k]}) =  n_0^2.
\end{align*}
For instance, in the example of $\bf F_4^{(1)}$ below, one has 
\begin{align*}
\mathbf{m} & = (m_0, m_{[1,1]}, m_{[1,2]}, m_{[1,3]}, m_{[2,1]}) = (1,2,2,2,1), \\
\mathbf{n} & = (n_0, n_{[1,1]}, n_{[1,2]}, n_{[1,3]}, n_{[2,1]}) = (4,3,2,1,2)
\end{align*}
and thus $m_1=2$ and $m_2=1$ and the condition is readily verified.

Below we will list the star-shaped non-simply laced affine Dynkin diagrams, which correspond to open de Rham spaces in Theorem~\ref{mstar=q} of dimension $2$. The simply-laced star-shaped ones $\tilde{D}_4, \tilde{E}_6,  \tilde{E}_7,\tilde{E}_8$ correspond to open de Rham spaces with logarithmic singularities.

\pagebreak
{\bf Example: }$\bf {A}_5^{(2)}$

 \begin{center}
\begin{tikzpicture}
\begin{scope}[start chain]
\dnodea{1}{\!1}
\node[chj,label={below:\dlabel{\! 2}},label={[inner sep=1pt]above right:\mlabel{1}}] {};
\dnodeanj{2}{\!1}
\path (chain-2) -- node[anchor=mid] {\(\Leftarrow\)} (chain-3);
\end{scope}
\begin{scope}[start chain=br going above]
\chainin(chain-2);
\node[chj,label={below left:\dlabel{1}},label={[inner sep=1pt]above:\mlabel{1}}] {};
\end{scope}
\end{tikzpicture}
\end{center}

The diagram depicts the non-simply laced Dynkin diagram. The integers written below each node show the dimension vector, while the ones above the node give the multiplicity vector. This corresponds to the open de Rham space $\MM^*_{(1^2,1^2),(1)}$ of type $((1^2,1^2),(1))$. By \cite[Theorem 9.11]{boalch-simply} $\MM^*_{(1^2,1^2),(1)}$ is isomorphic to an $A_3$ ALE space. In particular, the mixed Hodge structure is pure on $H^*(\MM^*_{(1^2,1^2),(1)})$ and $$WH(\MM^*_{(1^2,1^2),(1)};q,t)=1+3qt^2,$$ which is compatible with $|\MM^*_{(1^2,1^2),(1)}(\BF_q)|=q^2+3q$ from \eqref{formn=2} with $d=3,r=1$.

{\bf Example: }$\bf {C}^{(1)}_2$

\begin{center} \begin{tikzpicture}[start chain]
\dnodeanj{2}{\!1}
\dnodeanj{1}{\!2}
\dnodeanj{2}{\!1}
\path (chain-1) -- node{\(\Rightarrow\)} (chain-2);
\path (chain-2) -- node{\(\Leftarrow\)} (chain-3);
\end{tikzpicture}
\end{center}
This corresponds to the open de Rham space $\MM^*_{2,(1,1)}$ of  irregular type $(1,1)$. This is the only example from the list of star-shaped non-simply-laced Dynkin diagrams which is not isomorphic to a Nakajima quiver variety \cite{boalch08}. One can see this using the explicit equation for $\MM^*_{2,(1,1)}$ in \cite[(3.1)]{Bielawski2017}. One can deduce it has isolated singularities at infinity by \cite[Remarks 2.5.(a)]{siersma-tibar}. In turn, this implies  that it is homotopic to a wedge of spheres by \cite[Theorem 3.1]{siersma-tibar}. In order to match the virtual weight polynomial computation $WH_c(\MM^*_{2,(1,1)},q,-1)=q^2+2q$ from \eqref{formn=2} with $d=2,r=2$,  we must have that the mixed Hodge structure is pure and
$$WH(\MM^*_{2,(1,1)};q,t)=1+2qt^2.$$
 This case is special in that Boalch's \cite[Theorem 9.11]{boalch-simply} identification with a quiver variety does not apply, as we have two irregular poles. In fact there is no ALE space which is isomorphic with $\MM^*_{2,(1,1)}$ as the intersection form on $H^2_c(\MM^*_{2,(1,1)})$ is divisible\footnote{Because after a hyperk\"ahler rotation the manifold becomes a blow-up of $(\C^\times \times \C)/\Z_2$ at the two $A_1$ singularities, where $\Z_2$ acts by the inverse (see Example~\ref{ex:ALF} and \cite{Hitchin1984,Dancer1993}). Thus $H^2_c(\MM^*_{2,(1,1)},\Q)$ has a basis represented by the two disjoint exceptional divisors of self-intersection $-2$.} by $2$ , while the intersection form of the $A_2$ ALE space is not divisible by $2$.

{\bf Examples: }$\bf D_4^{(3)}, A_2^{(2)}, G_2^{(1)},F_4^{(1)},E_6^{(2)}$

In these cases we will only have one irregular pole, and thus \cite[Theorem 9.11]{boalch-simply} will apply. The arguments are identical to the  ${A}_5^{(2)}$ case above. We collect the results in the following table. 

\begin{center}
\begin{tabular}{||c c c c c||} 
 \hline
 type & Dynkin & open de Rham space&ALE type & $WH(q,t)$-polynomial \\ [0.5ex] 
 \hline\hline
$ {A}_5^{(2)}$& \begin{tikzpicture}
\begin{scope}[start chain]
\dnodea{1}{\!1}
\node[chj,label={below:\dlabel{\! 2}},label={[inner sep=1pt]above right:\mlabel{1}}] {};
\dnodeanj{2}{\!1}
\path (chain-2) -- node[anchor=mid] {\(\Leftarrow\)} (chain-3);
\end{scope}
\begin{scope}[start chain=br going above]
\chainin(chain-2);
\node[chj,label={below left:\dlabel{1}},label={[inner sep=1pt]above:\mlabel{1}}] {};
\end{scope}
\end{tikzpicture}&$\MM^*_{(1^2,1^2),(1)}$ & $\tilde{A}_3$ & $1+3qt^2$ \\
\hline
 $ D_4^{(3)}$ &\begin{tikzpicture}[start chain]
\dnodea{1}{\!1}
\dnodea{1}{\!2}
\dnodeanj{3}{\! 1}
\path (chain-2) -- node {\(\Lleftarrow\)} (chain-3);
\end{tikzpicture}&$\MM^*_{(1^2),(2)}$ & $\tilde{A}_2$ & $1+2qt^2$ \\ 
 \hline
 $ A_2^{(2)}$ & \begin{tikzpicture}[start chain]
\dnodeanj{1}{\!2}
\dnodeanj{4}{\!1}
\path (chain-1) -- node {\QLeftarrow} (chain-2);
\end{tikzpicture}&$\MM^*_{2,(3)}$ & $\tilde{A}_1$ & $1+qt^2$ \\
 \hline
 $ G_2^{(1)}$ & \begin{tikzpicture}[start chain]
\dnodea{3}{\!1}
\dnodea{3}{\!2}
\dnodeanj{1}{\!3}
\path (chain-2) -- node{\(\Rrightarrow\)} (chain-3);
\end{tikzpicture}&$\MM^*_{3,(2)}$ &$\tilde{A}_2$ &$1+2qt^2$ \\
 \hline
 $F_4^{(1)}$ & \begin{tikzpicture}[start chain]
\dnodea{2}{\!1}
\dnodea{2}{\!2}
\dnodea{2}{\!3}
\dnodeanj{1}{\!4}
\dnodea{1}{\!2}
\path (chain-3) -- node[anchor=mid]{\(\Rightarrow\)} (chain-4);
\end{tikzpicture}&$\MM^*_{(2^2),(1)}$ & $\tilde{D}_4$ & $1+4qt^2$ \\
 \hline
 $E_6^{(2)}$ & \begin{tikzpicture}[start chain]
\dnodea{1}{\!1}
\dnodea{1}{\!2}
\dnodea{1}{\!3}
\dnodeanj{2}{\!2}
\dnodea{2}{\!1}
\path (chain-3) -- node[anchor=mid] {\(\Leftarrow\)} (chain-4);
\end{tikzpicture}&$\MM^*_{(1^3),(1)}$ & $\tilde{D}_4$ & $1+4qt^2.$ \\ [1ex] 
 \hline
\end{tabular}
\end{center}

In all these cases the mixed Hodge polynomial is compatible with the weight polynomial computed from \eqref{formn=2} in the rank $2$ cases and  \eqref{formn=3} in the rank $3$ cases. The only example of rank $4$ is  $F_4^{(1)}$ where one can compute the weight polynomial $q^2+4q$  directly from \eqref{finalcountf}.

\newpage
\section{Hyperk\"ahler considerations} \label{s:HK} 

Our purpose in this section is to prove Theorem \ref{t:odRhK}, which says that some of the open de Rham spaces $\MM^*(\mathbf{C})$ that we have been discussing, namely those for which all the formal types are of order $\leq 2$, admit canonical complete hyperk\"ahler metrics.  

Let us first give some motivation for and review what is already known about this problem.  In the tame case (i.e., when $m_i = 1$ for all $i$ in Definition \ref{d:odr}), the corresponding open de Rham is known to be a Nakajima quiver variety \cite[Theorem 1]{CB03}, and these possess complete hyperk\"ahler metrics \cite[Theorem 2.8]{nak94}.  On the other hand, in the case where we have two poles of order $2$ (cf., \cite[discussion after Corollary 1]{Boalch2007}), then $\MM^*(\mathbf{C})$ will be a holomorphic symplectic quotient of $T^*G$, which is in fact realizable as a hyperk\"ahler reduction of $T^*G$; see Example \ref{ex:ALF} for further details.  The existence of such metrics in more general irregular cases was discussed in \cite[\S3.1]{Boalch2012}, but as we know of no precise reference for this fact, we give a construction here.

Now, to give a rough explanation of this construction, let us recall that an $\MM^*(\mathbf{C})$ with poles of order $\leq 2$ is an additive fusion product of coadjoint orbits in $\g^\vee$ and those for the group $\G_2 = \GL_n(\RR_2)$ (using the notation of Section \ref{s:jd}).  Each coadjoint orbit of the latter type may be realized as an algebraic symplectic quotient of $T^*\G$, the cotangent bundle of $\G=\GL_n(\C)$, by left multiplication by a maximal torus (Lemma \ref{l:decoup}).  It is well known that $T^*\G$ admits a hyperk\"ahler metric, by an infinite-dimensional hyperk\"ahler quotient via Nahm's equations \cite{Kr88}.  Furthermore, if $K \subset \G$ is a maximal compact subgroup then $K \times K$ acts by left and right multiplication \eqref{e:TGact} and these actions admit hyperk\"ahler moment maps \cite{DS96}.

For coadjoint orbits in $\g^\vee$, hyperk\"ahler metrics were first constructed by \cite{Kronheimer1990} for regular semisimple orbits and for general semisimple orbits in \cite{Biquard1996} and \cite{Kovalev1996}.  The coadjoint $\G$-action can be restricted to $K$ and it can be shown (Proposition \ref{p:Ohkstmt}), in a manner similar to that for $T^*\G$, that this action also admits a hyperk\"ahler moment map.

Now, if a hyper-hamiltonian action of a compact group extends to a holomorphic action of its complexification, then the hyperk\"ahler reduction can be understood as a holomorphic symplectic reduction \cite[\S3(D)]{HKLR}. Here we will need the opposite direction:  $\MM^*(\mathbf{C})$ is given as an algebraic symplectic quotient; we will  show that it, in fact, arises as a hyperk\"ahler quotient.  A special case of a version of the Kempf--Ness theorem due to Mayrand \cite{Mayrand} gives sufficient conditions for algebraic symplectic quotients of the type we have been considering to be upgraded to hyperk\"ahler quotients.  

\subsection{Holomorphic symplectic quotients to hyperk\"ahler quotients}

Let us begin by incorporating Mayrand's statement into the following, which will give us the criterion we will apply later to obtain the theorem.

\begin{prop} \label{p:Mayrand}
Suppose $(M, g, \mathbf{I}, \mathbf{J}, \mathbf{K})$ is a hyperk\"ahler manifold, with K\"ahler forms $\omega_{\mathbf{I}},\omega_{\mathbf{J}},\omega_{\mathbf{K}}\in \Omega^2(M)$ in the corresponding complex structures.  We suppose that $(M, \mathbf{I})$ is a (smooth) complex affine variety and refer to $M$ as such with the complex structure $\mathbf{I}$ in mind.  Suppose $G$ is a complex reductive group with an algebraic action on $M$ for which the restriction to its maximal compact $K$ admits a hyperk\"ahler moment map $\mu_{\mathbf{I}}, \mu_{\mathbf{J}}, \mu_{\mathbf{K}} : M \to \k^\vee$.  As usual, we will write
\begin{align*}
\mu_{\R} & := \mu_{\mathbf{I}} & \mu_{\C} & := \mu_{\mathbf{J}} + i \mu_{\mathbf{K}}
\end{align*}
for the real and complex components of the moment map; we will assume that $\mu_{\C} : M \to \g^\vee = \k^\vee \oplus i \k^\vee$ is algebraic. Here $\g=\textnormal{Lie}(G)$ is a complex and $\k=\textnormal{Lie}(K)$ a real Lie algebra and $\vee$ means dual vector space over the respective field. Let $\lambda \in (\g^\vee)^G$ be such that $G$ acts freely on the affine variety $\mu_{\C}^{-1}(\lambda)$; thus, the algebraic symplectic quotient
\begin{align*}
M /\!\!/_\lambda G := \spec \C[\mu_{\C}^{-1}(\lambda)]^G
\end{align*}
is smooth.  Then if there exists a $K$-invariant, proper global K\"ahler potential for $\omega_{\mathbf{I}}|_{\mu_{\C}^{-1}(\lambda)}$, which is bounded below, then there exists $\lambda_{\R} \in \k^\vee$ and, for the complex structures induced from $\mathbf{I}$, a natural biholomorphism
\begin{align*}
M /\!\!/\!\!/_{(\lambda_{\R}, \lambda)} K \cong M /\!\!/_\lambda G,
\end{align*} where $M /\!\!/\!\!/_{(\lambda_{\R}, \lambda)}:= \left( \mu_{\R}^{-1}(\lambda_{\R}) \cap \mu_{\C}^{-1}(0) \right) / K$ denotes the hyperk\"ahler quotient. 
\end{prop}

\begin{proof}
The K\"ahler potential produces a moment map $\tilde{\mu}_{\R}$ for the $K$-action with respect to $\omega_{\mathbf{I}}$ \cite[Proposition 4.1]{Mayrand}.  This must differ from the given $\mu_{\R}$ by a constant, i.e., there exists $\lambda_{\R} \in \k^\vee$ such that $\mu_\R = \tilde{\mu}_{\R} + \lambda_{\R}$.  Thus, $\tilde{\mu}_{\R}^{-1}(0) = \mu_{\R}^{-1}(\lambda_{\R})$.  Then \cite[Proposition 4.2]{Mayrand} gives a homeomorphism at the last step of the sequence
\begin{align*}
M /\!\!/\!\!/_{(\lambda_{\R}, \lambda)} K = \left( \mu_{\R}^{-1}(\lambda_{\R}) \cap \mu_{\C}^{-1}(0) \right) / K = \left( \tilde{\mu}_{\R}^{-1}(0) \cap \mu_{\C}^{-1}(0) \right) / K \cong M /\!\!/_\lambda G.
\end{align*} 
Furthermore, by the freeness of the action, there is a single orbit type stratum and hence the homeomorphism is, in fact, a biholomorphism, again from \cite[Proposition 4.2]{Mayrand}.
\end{proof}

\subsection{Hyperk\"ahler moment maps on the factors} \label{s:factors}

Here we show that the two kinds of factors that appear in the relevant additive fusion product each admit hyper-hamiltonian group actions.

\subsubsection{Cotangent bundles}

Let $G$ be a complex reductive group with Lie algebra $\g$ and let $K \leq G$ be a maximal compact subgroup.  We recall that for $T^*G = G \times \g^\vee$, there is an algebraic hamiltonian action of $G \times G$ given by
\begin{align} \label{e:TGact}
(g,h) \cdot (a, X) = (gah^{-1}, \Ad_h X),
\end{align}
for which the moment map is
\begin{align} \label{e:TGmm}
(a,X) \mapsto (\Ad_a X, -X).
\end{align}

\begin{prop} \label{p:TGhkstmt}
$T^*G$ admits a hyperk\"ahler metric for which the restriction of the action \eqref{e:TGact} to $K \times K$ admits a hyperk\"ahler moment map.  Furthermore, the complex part of this moment map is given by \eqref{e:TGmm}, and for the natural complex (K\"ahler) structure there exists a $(K \times K)$-invariant, proper and bounded below global K\"ahler potential.
\end{prop}

\begin{proof}
As mentioned, the hyperk\"ahler structure is due to \cite[Proposition 1]{Kr88}.  The existence of the hyperk\"ahler moment map is \cite[\S3 Lemma 2]{DS96}.  The expression for the complex part of the moment map is obtained by comparing \cite[Equations (4), (5)]{DS96} and the expressions in the statement of \cite[\S3 Lemma 2]{DS96}.  Finally, the existence of the global K\"ahler potential is \cite[Proposition 4.6]{Mayrand}.
\end{proof}

\subsubsection{Coadjoint orbits} \label{s:Ohk}

Let $G$, $\g$, $K$ be as above and let $\O$ be a semisimple coadjoint orbit in $\g^\vee$.  Of course, $G$ acts algebraically on $\O$ via the coadjoint action and the moment map is simply the inclusion $\O \hookrightarrow \g^\vee$.  Exactly the same statement as in Proposition \ref{p:TGhkstmt} holds for $\O$.

\begin{prop} \label{p:Ohkstmt}
$\O$ admits a complete hyperk\"ahler metric for which the restriction of the coadjoint action to $K$ admits a hyperk\"ahler moment map.  Furthermore, the complex part of this moment map is given by the inclusion $\O \hookrightarrow \g^\vee$, and for the natural complex (K\"ahler) structure there exists a $K$-invariant, proper and bounded below global K\"ahler potential.
\end{prop}

\begin{proof}
As mentioned earlier, existence of the hyperk\"ahler metrics can be found at \cite[Theorem 1.1]{Kronheimer1990} for regular semisimple orbits and at \cite[Th\'eor\`eme 1]{Biquard1996} and \cite[Theorem 1.1]{Kovalev1996} for general semisimple orbits.  The fact that the conjugation action of $K$ admits a hyperk\"ahler moment map can be proved in the same way as \cite[Lemma 2]{DS96}.  The existence of the K\"ahler potential with the indicated properties uses the same argument as that of \cite[Proposition 4.6]{Mayrand}.  Further details can be found in Subsection~\ref{detailed}.
\end{proof}

\begin{rmk} \label{r:hKfam}
Let $S := T \cap K$ be a maximal torus in the maximal compact group $K \leq G$ and let $\mathfrak{s}$ be its Lie algebra.  Then, as in the references \cite{Kronheimer1990, Biquard1996, Kovalev1996}, once the coadjoint orbit is fixed, the family of such hyperk\"ahler metrics is parametrized by an element $\tau_1 \in \s$ (see \ref{detailed}).
\end{rmk}

\subsection{Hyperk\"ahler metrics on open de Rham spaces}

We first give a lemma describing coadjoint orbits for $G_2$ as algebraic symplectic reductions of $T^*G$.

\begin{lem} \label{l:decoup}
Consider the diagonal element \eqref{e:diagelt}
\begin{align*}
C := \frac{C_2}{z^2} + \frac{C_1}{z} \in \t_2^\vee
\end{align*}
with $C_2$ regular (i.e., having distinct eigenvalues).  Then the $G_2$ coadjoint orbit $\O(C) \subseteq \g_2^\vee$ is isomorphic to the algebraic symplectic quotient
\begin{align*}
T^*G /\!\!/_{C_1} T,
\end{align*}
for the $T$-action $t \cdot (a, X) = (ta, X)$.
\end{lem}

\begin{rmk} \label{r:Boalchsp}
This is a special case of \cite[Lemma 2.3(2), see also Lemma 2.4]{Bo01}.
\end{rmk}

\begin{proof}
Recall that the moment map for this action is $\mu : T^* G  \to \t^\vee$
\begin{align*}
(a, X) \mapsto \pi_{\t} \left( \Ad_a X \right).
\end{align*}
The map $\mu^{-1}(C_1) \to \O(C)$
\begin{align*}
(a, X) \mapsto \frac{\Ad_{a^{-1}} C_2 }{z^2 } + \frac{X}{z}
\end{align*}
is readily seen to be $T$-invariant, so descends to a morphism $T^*G/\!\!/_{C_1} T \to \O(C)$.

Now, as in Lemma \ref{l:coadjorb}\eqref{l:coadjhomog}, $\O(C)$ may be realized as the geometric quotient $G_2/T_2$.  If we write an element of $G_2$ in the form $g (I + zH)$ for some $g \in G$, $H \in \g$, then we define $G_2 \to T^*G /\!\!/_{C_1} T$ by
\begin{align*}
g(I + zH) \mapsto [ g^{-1}, \Ad_g ( C_1 + [H, C_2]) ],
\end{align*}
with the square brackets indicating the class mod $T$.  It is straightforward to check that this is well-defined and that it descends to $\O(C) \to T^*G/\!\!/_{C_1} T$ to give the inverse.
\end{proof}

\begin{thm} \label{t:odRhK}
Consider a generic open de Rham space $\MM^*(\mathbf{C})$ for which the orders of all the formal types $C^i$ are $\leq 2$.  Then $\MM^*(\mathbf{C})$ admits a complete hyperk\"ahler metric induced by a  choice of a hyperk\"ahler metric on each coadjoint orbit $\O(C^i)$ as in Proposition~\ref{p:Ohkstmt}.
\end{thm}

\begin{proof}
As in Section \ref{s:addfusprod}, we label the coadjoint orbits so that $\O(C^i) \subseteq \g^\vee$ for $1 \leq i \leq k$ and $\O(C^i)$ is a coadjoint orbit for $G_2$ for $k+1 \leq i \leq d$; we will abbreviate $\O(C^i)$ to $\O^i$ in the following.  Spelling out \cite[Proposition 2.1]{Bo01}, we may use Lemma \ref{l:decoup} to rewrite the holomorphic symplectic quotient of Definition \ref{d:odr} as follows \small
\begin{align} \label{e:odRsympquot}
\MM^*(\mathbf{C}) & = \prod_{i=1}^d \O(C^i) \bigg/\!\!\!\!\bigg/_0 G = \left( \prod_{i=1}^k \O^i \times \prod_{i=k+1}^d \O^i \right) \bigg/\!\!\!\!\bigg/_0 G \cong \left( \prod_{i=1}^k \O^i \times \prod_{i=k+1}^d T^*G /\!\!/_{C_1^i} T \right) \bigg/\!\!\!\!\bigg/_0 G \nonumber \\
& \cong \left( \prod_{i=1}^k \O^i \times \prod_{i=1}^m T^*G \right) \bigg/\!\!\!\!\bigg/_{(\mathbf{C}_1, 0)} T^m \times G,
\end{align} \normalsize
where $\mathbf{C}_1$ denotes the tuple $(C_1^{k+1}, \ldots, C_1^d)$ of residue terms of the formal types.  Each factor of $T$ acts via the left action on the corresponding $T^*G$ factor as in \eqref{e:TGact} and the $G$ factor acts diagonally:  by the coadjoint action on $\O^i$, for $1 \leq i \leq k$ and with the right action in \eqref{e:TGact} for the factors indexed by $k+1 \leq i \leq d$.  We have thus expressed $\MM^*(\mathbf{C})$ as an algebraic symplectic quotient and we can now apply Proposition \ref{p:Mayrand}.  
But now the hypotheses are verified for each factor in Propositions \ref{p:TGhkstmt} and \ref{p:Ohkstmt}, and can therefore easily be verified for the product.
\end{proof}

\begin{ex} \label{ex:ALF}
Consider the case of a rank $2$ open de Rham space $\MM^*(\mathbf{C})$ with two poles each of order $2$, which is a smooth affine algebraic surface (Proposition \ref{odrspa}).  Applying \cite[Proposition 2.1]{Bo01} or \eqref{e:odRsympquot}, if we first take the quotient of $T^* G \times T^*G$ by $G$, one can see that it is a quotient of the form $T^*G /\!\!/ T^2$ (for appropriate values of the moment map), which is precisely the reduction carried out at \cite[pp.88--89]{Dancer1993}.  Hence $\MM^*(\mathbf{C})$ is isometric to the deformation of the $D_2$ singularity, which was already observed at \cite[p.3]{boalch08}, the hyperk\"ahler metric on which was previously constructed via twistor methods in \cite[\S7]{Hitchin1984}, and proved to be ALF in \cite[\S5.3]{CherkisKapustin1999}.  It is worth noting that this space can also be described via a slice construction \cite[(3.1)]{Bielawski2017} which, although an algebraic operation, also yields the metric \cite[\S4]{Bielawski2017}.  Furthermore, we expect the metrics on higher dimensional open de Rham spaces to exhibit ``higher dimensional ALF behaviour''.
\end{ex}

\begin{rmk} \label{r:hKrmks}
\begin{enumerate}[(i)]
\item \label{r:higherorderhK} It is true more generally, and not much harder to prove, that with no restriction on the order of the formal types, the spaces $\MM^*(\mathbf{C})$ always admit complete hyperk\"ahler metrics.  However, there is some extra choice involved. In general, a coadjoint orbit for $\G_m$ may be realized as an algebraic symplectic quotient of $T^*G \times \O(C')$ (\cite[Lemma 2.3(2), Lemma 2.4]{Bo01}, cf.\ Remark \ref{r:Boalchsp}), where $\O(C')$ is a coadjoint orbit for the unipotent group $\G_m^1$.  As such, $\O(C')$ is an even-dimensional complex affine space, hence, upon some choice of coordinates, admits a flat hyperk\"ahler metric.  One can show that the coordinates can be chosen so that $S$ acts on pairs of coordinates with opposite weights, and hence with a hyperk\"ahler moment map. 

\item The spaces $T^* G \times \O(C')$ are (isomorphic to) what are known as ``extended orbits'' and these can be arranged into a moduli space by taking an additive fusion product (this is referred to as the ``extended'' moduli space in \cite[Definition 2.6]{Bo01} \cite[Definition 2.4]{HiroeYamakawa}.  This moduli space will admit an action of $T^l$, where $l$ is the number of irregular poles, and $\MM^*(\mathbf{C})$ can thus be realized as a hyperk\"ahler quotient of the extended moduli space from this action.  Indeed, in the last expression in \eqref{e:odRsympquot}, if we first take the quotient by $G$, then we obtain the extended moduli space.
\end{enumerate}
\end{rmk}

\subsection{Details of proof of Proposition \ref{p:Ohkstmt}} \label{detailed}

Here, we will give the details of the proof of Proposition \ref{p:Ohkstmt}.  The statements that need to be proved are:  the existence of a hyperk\"ahler moment map on the coajdoint orbit $\O$ and its complex part is simply the inclusion map into the dual of the complex Lie algebra, which is Lemma \ref{l:Omm} below; and the existence of a K\"ahler potential with the appropriate properties, which is Lemma \ref{l:OKahpot}.

To proceed, we will need to fix notation, and so we will adopt that of \cite[\S3, \emph{L'espace des modules}]{Biquard1996}.  As such, $G$ will now denote a compact Lie group, which is of course the maximal compact subgroup of its complexification $G_\C$, whereas in Section \ref{s:factors} it denoted the complexification and $K$ a maximal compact subgroup; we hope this will cause the reader no confusion.  Of course, $\g$ will denote the Lie algebra of $G$ and $\g_\C$ its complexification, and $\langle \, , \rangle$ will denote a $\Ad$-invariant inner product on $\g$.

Let us recall how the coadjoint orbit $\O$ is identified with a moduli space of solutions to Nahm's equations.  Let $S \subset G$ be a maximal torus (which is, of course, compact) with Lie algebra $\mathfrak{s}$, and respective complexifications $S_\C$ and $\mathfrak{s}_\C$.  As usual, using the invariant inner product, we identify $\g^\vee = \g$ and $\g_\C^\vee = \g_\C$.  Viewing $\O$ as a subset of $\g_\C$, as it is a semisimple orbit, its intersection with $\mathfrak{s}_\C$ is a singleton; we write this element as $\tau_2 + i \tau_3$ with $\tau_2$, $\tau_3 \in \mathfrak{s}$.  One chooses a third element $\tau_1 \in \mathfrak{s}$ so that we have a triple $\tau = (\tau_1, \tau_2, \tau_3) \in \mathfrak{s}^3$, so that for an appropriate $\varsigma > 0$, we can make sense of the space $\Omega_{\nabla; \varsigma}^1$ as described at \cite[\S3, p.265]{Biquard1996}.  We will write an element $\nabla + a \in \mathscr{A}_\varsigma$ as a quadruple $T = (T_0, T_1, T_2, T_3)$ of smooth maps $T_i : (-\infty, 0] \to \g$ satisfying an asymptotic condition depending on $\tau$: one has $\nabla + a = d + T_0 \, ds + \sum_{i=1}^e T_i \, d\theta^i$.  We consider such $T$ which are solutions to Nahm's equations:
\begin{align} \label{e:Nahmeq}
\frac{dT_i}{ds} + [T_0, T_i] + [T_j, T_k] = 0
\end{align}
for cyclic permutations $(i \ j \ k)$ of $(1 \ 2 \ 3)$.  The quotient of the space of such solutions by the group $\mathscr{G}_\varsigma$, also defined at \cite[\S3, p.265]{Biquard1996}, will be referred to as the moduli space $M = M(\tau)$ of solutions to Nahm's equations.  We will often write $[T]$ for the $\mathscr{G}_\varsigma$-orbit of a solution $T$.  The isomorphism $M \xrightarrow{\sim} \O$ is given by \cite[Corollaire 4.5]{Biquard1996} (see also the definition before Equations (IIIa) and (IIIb)) as
\begin{align} \label{e:ONahmiso}
[T] \mapsto T_2(0) + i T_3(0).
\end{align}

\begin{lem} \label{l:Omm}
The map $\mu : M \to \g^{\oplus 3}$ given by
\begin{align} \label{e:Omm}
[T] \mapsto \left( T_1(0), T_2(0), T_3(0) \right)
\end{align}
yields a hyperk\"ahler moment map for the (co)adjoint $G$-action.  Furthermore, the complex part of the moment map coincides with the inclusion of $\O$ in the dual of the complex Lie algebra.
\end{lem}

\begin{proof}
As mentioned in Section \ref{s:Ohk}, the proof mirrors that of \cite[Lemma 2]{DS96}.  Consider the group 
\begin{align*}
\mathscr{G}_\varsigma^+ := \left\{ g : \R_- \to G \, : \, (\nabla g) g^{-1} \in \Omega_{\nabla; \varsigma}^1 \right\}
\end{align*}
which has Lie algebra
\begin{align*}
\textnormal{Lie}(\mathscr{G}_\varsigma^+) = \left\{ u : \R_- \to \g \, : \, \nabla u \in \Omega_{\nabla; \varsigma}^1 \right\}.
\end{align*}

Consider $\mathscr{G}_\varsigma^+ \to G$ the evaluation map at $s = 0$ and $\mathscr{G}_\varsigma$ its preimage  of $1_\G$. It  is a normal subgroup of $\mathscr{G}_\varsigma^+$ with quotient $G$.  Of course we have a parallel statement for the Lie algebras.  Furthermore, $\mathscr{G}_\varsigma^+$ acts on the space of solutions to Nahm's equations inducing the adjoint action of $G$ on $\O$, as is easily seen via the map \eqref{e:ONahmiso}.

A tangent vector to $M$ at $[T]$ is represented by a quadruple $w = (w_0, w_1, w_2, w_3) \in \Omega_{\nabla, \varsigma}^1$ satisfying
\begin{align} \label{e:Nahmtgtvec}
\frac{dw_i}{ds} = - [T_0, w_i] - [w_0, T_i] - [T_j, w_k ] - [w_j, T_k]
\end{align}
for cyclic permutations $(i \ j \ k)$ of $(1 \ 2 \ 3)$ (cf.\ \cite[Equations (6)-(9)]{DS96}; note that there is a slight difference in the complex structures there and in \cite{Biquard1996} accounting for the sign differences in the equations).  These equations are obtained simply by linearizing Nahm's equations \eqref{e:Nahmeq}.

Let $\xi \in \g$ and choose a lift $u(s) \in \textnormal{Lie}(\mathscr{G}_\varsigma^+)$, so that $u(0) = \xi$.  A representative for the tangent vector $v_\xi([T])$ to $M$ at $[T]$ generated by the infinitesimal action of $\xi$ is given by
\begin{align*}
v_\xi([T]) \leftrightarrow \left( [ u, T_0] - \frac{du}{ds} , [ u, T_1], [ u, T_2], [ u, T_3] \right).
\end{align*}
We evaluate using \eqref{e:Nahmtgtvec}
\begin{align*}
\omega_{\mathbf{I}}( v_\xi, w) & = \int_{-\infty}^0 - \left\langle [u, T_0] -\frac{du}{ds}, w_1 \right\rangle + \langle [u, T_1], w_0 \rangle - \langle [ u, T_2], w_3 \rangle + \langle [ u, T_3], w_2 \rangle \, ds \\
& = \int_{-\infty}^0 \left\langle u, - [T_0, w_1] -[w_0, T_1] - [T_2, w_3] - [w_2, T_3] \right\rangle + \left\langle \frac{du}{ds}, w_1 \right\rangle \, ds \\
& = \int_{-\infty}^0 \left\langle u, \frac{dw_1}{ds} \right\rangle + \left\langle \frac{du}{ds}, w_1 \right\rangle \, ds = \int_{-\infty}^0 \frac{d}{ds} \langle u, w_1 \rangle \, ds = \langle \xi, w_1(0) \rangle,
\end{align*}
since $u \to 0$ as $s \to -\infty$.

On the other hand, pairing $\xi$ with the moment map $\mu_{\mathbf{I}}$ gives the function $\mu_{\mathbf{I}}^\xi : M \to \R$
\begin{align*}
[T] \mapsto \langle \xi, T_1(0) \rangle.
\end{align*}
Hence
\begin{align*}
d\mu_{\mathbf{I}}^\xi(w) = \langle \xi, w_1(0) \rangle = \omega_{\mathbf{I}}( v_\xi, w)
\end{align*}
and this is exactly the moment map condition.  The same computation can be repeated for the complex structures $\mathbf{J}$ and $\mathbf{K}$.

The statement about the complex part of the moment map is obvious from the expressions \eqref{e:ONahmiso} and \eqref{e:Omm}.
\end{proof}

\begin{lem} \label{l:OKahpot}
The semisimple coajdoint orbit $\O$ admits a global $G$-invariant K\"ahler potential (for the complex structure $\mathbf{I}$) which is proper and bounded below.
\end{lem}

As mentioned, the proof here is adapted from that of \cite[Lemma 4.5]{Mayrand}.  

\begin{proof}
A global K\"ahler potential $\varphi_{\mathbf{I}} : M \to \R$ for the K\"ahler form $\omega_{\mathbf{I}}$ is given by (see \cite[\S3(E)]{HKLR}, cf.\ \cite[p.64]{DS96})
\begin{align*}
[T] \mapsto \frac{1}{2} \int_{-\infty}^0 \langle T_2, T_2 \rangle + \langle T_3, T_3 \rangle \, ds.
\end{align*}
It is then sufficient to show that the $\varphi_{\mathbf{I}}$ is $G$-invariant, proper and bounded below.  $G$-invariance follows from that of the bilinear form $\langle \, , \rangle$ and lower-boundedness is obvious, so properness is essentially all that needs to be proved.

Let $\xi = (\xi_1, \xi_2, \xi_3) \in \g^{\oplus 3}$.  Then the existence and uniqueness theorem for systems of ordinary differential equations gives a unique solution $T^\xi = (T_0^\xi \equiv 0, T_1^\xi, T_2^\xi, T_3^\xi)$ to the reduced Nahm's equations 
\begin{align} \label{e:Nahmredeq}
\frac{dT_i}{dt} + [T_j, T_k] = 0
\end{align}
(this is just \eqref{e:Nahmeq} with $T_0 = 0$) with $T_i^\xi(0) = \xi_i$ for $i = 1, 2, 3$.  This allows us to define a function $\widetilde{\varphi}_{\mathbf{I}} : \g^{\oplus 3} \to \R$ by
\begin{align} \label{e:tildephi}
\xi \mapsto \frac{1}{2} \int_{-\infty}^0 \langle T_2^\xi, T_2^\xi \rangle + \langle T_3^\xi, T_3^\xi \rangle \, ds,
\end{align}
which is (at the very least) continuous, again by the assertions of the existence and uniqueness theorem on the dependence on the initial conditions.

Let us explain how this is related to $\varphi_{\mathbf{I}}$.  Consider the map $\ev_\C : \g^{\oplus 3} \to \g_\C$
\begin{align*}
\xi \mapsto \xi_2 + i \xi_3.
\end{align*}
As $\O$ is a semi-simple (co)adjoint orbit, $\O$ is closed in $\g_\C$ and hence so is $\ev_\C^{-1}(\O)$.  We may then identify $\O$ with the subset of $\xi \in \ev_\C^{-1}(0)$ for which $T^\xi$ is gauge equivalent to an element of $\mathscr{A}_\varsigma$.  We observe that this will be closed, as we are imposing an asymptotic condition on $T_1^\xi$.  As the definition of $\varphi_{\mathbf{I}}$ is independent of the gauge equivalence class, one sees that under the identification of $\O$ with the described subset of $\ev_\C^{-1}(\O)$, 
\begin{align*}
\varphi_{\mathbf{I}} = \widetilde{\varphi}_{\mathbf{I}}|_{\O}.
\end{align*}
It now suffices to show that there is a closed subset $V \subseteq \g^{\oplus 3}$ containing $\O$ (using the identification above) for which $\widetilde{\varphi}_{\mathbf{I}}|_V$ is proper, since the restriction of a proper map to a closed subset is still proper.

For this, we will take
\begin{align*}
V := \left\{ r \cdot \xi \, : \, \xi \in \ev_\C^{-1}(\O), \quad r \in \R_{\geq 0} \right\},
\end{align*}
which is closed in $\g^{\oplus 3}$.  Let $S \subseteq \g^{\oplus 3}$ be the unit sphere (here, we may take the invariant inner product on $\g$ in each factor); this is compact.  Then $S \cap V$ is a compact subset of $\g^{\oplus 3}$, and hence $\widetilde{\varphi}_{\mathbf{I}}$ has a minimum value $m > 0$.  It is non-zero, for if $r \cdot \xi \in V$ is such that $\widetilde{\varphi}_{\mathbf{I}}(r \xi) = 0$ then, assuming that $\O \neq \{ 0 \}$ (which we may of course do), then one finds $\ev_\C(r \xi) = 0$, and hence $r = 0$; but this would contradict $r \xi \in S$.

By uniqueness of the solutions $T^\xi$ to \eqref{e:Nahmredeq}, for $r \in \R_{>0}$, it is easy to see that $T^{r \xi}(s) = r T^\xi(rs)$.  From this, we obtain for any $\xi \in \g^{\oplus 3}$,
\begin{align} \label{e:rscal}
\widetilde{\varphi}_{\mathbf{I}} ( r\xi) & = \frac{1}{2} r^2 \int_{-\infty}^0 \langle T_2^\xi(rs), T_2^\xi(rs) \rangle + \langle T_3^\xi(rs), T_3^\xi(rs) \rangle \, ds \nonumber \\
& = \frac{1}{2} r \int_{-\infty}^0 \langle T_2^\xi(t), T_2^\xi(t) \rangle + \langle T_3^\xi(t), T_3^\xi(t) \rangle \, dt = r \widetilde{\varphi}_{\mathbf{I}} ( \xi).
\end{align}

Now, given $\xi \in V$, one has $\tfrac{\xi}{\| \xi \|} \in S \cap V$ and so \eqref{e:rscal} gives
\begin{align*}
\widetilde{\varphi}_{\mathbf{I}} ( \xi)  = \| \xi \|  \widetilde{\varphi}_{\mathbf{I}} \left( \frac{\xi}{\| \xi \|} \right) \geq m \| \xi \|  \end{align*} or equivalently \begin{align*}  \| \xi \|  \leq  \frac{\widetilde{\varphi}_{\mathbf{I}} ( \xi)}{m}.
\end{align*}
From this, one finds that the preimage of a bounded set in $\R$ under $\widetilde{\varphi}_{\mathbf{I}}|_V$ is bounded in $V$.  By continuity, the preimage of a closed set is closed, and so $\widetilde{\varphi}_{\mathbf{I}}|_V$ is proper.
\end{proof}

\appendix

\footnotesize

\bibliographystyle{abbrv}
\bibliography{AAAreferences}

\begin{thebibliography}{10}

\bibitem{Dancer1993}
{A.~Dancer}.
\newblock {Dihedral singularities and gravitational instantons}.
\newblock {\em {Journal of Geometry and Physics}}, 12:77--91, 1993.

\bibitem{Audin}
M.~Audin.
\newblock {\em {The Topology of Torus Actions on Symplectic Manifolds}}.
\newblock Birkh\"auser Verlag, 1991.

\bibitem{BHFD16}
G.~B\'erczi, T.~Hawes, F.~Kirwan, and B.~Doran.
\newblock Constructing quotients of algebraic varieties by linear algebraic
  group actions.
\newblock {\em International Press of Boston, Inc.}, pages (pp. 341 -- 446),
  2016.

\bibitem{Bielawski2017}
R.~Bielawski.
\newblock {Slices to sums of adjoint orbits, the Atiyah--Hitchin manifold, and
  Hilbert schemes of points}.
\newblock {\em Comp. Man.}, 4:16--36, 2017.

\bibitem{Biquard1996}
O.~Biquard.
\newblock {Sur les \'equations de Nahm et la structure de Poisson des
  alg\`ebres de Lie semi-simples complexes}.
\newblock {\em Math. Ann.}, 304:253--276, 1996.

\bibitem{BiquardBoalch}
O.~Biquard and P.~Boalch.
\newblock {Wild non-abelian Hodge theory on curves}.
\newblock {\em Compos. Math.}, 140:179--204, 2004.

\bibitem{Bo01}
P.~Boalch.
\newblock Symplectic manifolds and isomonodromic deformations.
\newblock {\em Adv. in Math.}, 163(2):137--205, 2001.

\bibitem{Boalch2007}
P.~Boalch.
\newblock {Quasi-Hamiltonian Geometry of Meromorphic Connections}.
\newblock {\em Duke Math.l J.}, 139(2):369--405, 2007.

\bibitem{boalch08}
P.~{Boalch}.
\newblock {Irregular connections and Kac-Moody root systems}.
\newblock {\em ArXiv e-prints}, June 2008.

\bibitem{Boalch2012}
P.~Boalch.
\newblock {Hyperkahler manifolds and nonabelian Hodge theory on (irregular)
  curves}.
\newblock {\em ArXiv e-prints}, 2012.

\bibitem{boalch-simply}
P.~Boalch.
\newblock Simply-laced isomonodromy systems.
\newblock {\em Publ. Math. Inst. Hautes \'Etudes Sci.}, 116:1--68, 2012.

\bibitem{CLL}
A.~Chambert-Loir and F.~Loeser.
\newblock Motivic height zeta functions.
\newblock {\em American Journal of Mathematics}, 138(1):1--59, 2016.

\bibitem{CherkisKapustin1999}
S.~A. Cherkis and A.~Kapustin.
\newblock {Singular Monopoles and Gravitational Instantons}.
\newblock {\em Communications in Mathematical Physics}, 203:713--728, 1999.

\bibitem{CL}
R.~Cluckers and F.~Loeser.
\newblock {Constructible exponential functions, motivic Fourier transform and
  transfer principle}.
\newblock {\em Ann. of Math.}, 171(2):1011--1065, 2010.

\bibitem{CB03}
W.~Crawley-Boevey.
\newblock On matrices in prescribed conjugacy classes with no common invariant
  subspace and sum zero.
\newblock {\em Duke Math. J.}, 118(2):339--352, 2003.

\bibitem{DS96}
A.~Dancer and A.~Swann.
\newblock {Hyperk{\"a}hler metrics associated to compact Lie groups}.
\newblock In {\em Math. Proc. of the Cam. Phil. Soc.}, volume 120, pages
  61--69. Cambridge Univ Press, 1996.

\bibitem{De71}
P.~Deligne.
\newblock {Th{\'e}orie de Hodge: II}.
\newblock {\em Publications math{\'e}matiques de l'IH{\'E}S}, 40(1):5--57,
  1971.

\bibitem{De74}
P.~Deligne.
\newblock {Th{\'e}orie de Hodge: III}.
\newblock {\em Publications math{\'e}matiques de l'IH{\'E}S}, 44:5--77, 1974.

\bibitem{DL76}
P.~Deligne and G.~Lusztig.
\newblock Representations of reductive groups over finite fields.
\newblock {\em Ann.of Math.}, pages 103--161, 1976.

\bibitem{Do03}
I.~Dolgachev.
\newblock {\em Lectures on invariant theory}.
\newblock Number 296. Cam. Univ. Press, 2003.

\bibitem{EK09}
T.~Ekedahl.
\newblock The {G}rothendieck group of algebraic stacks.
\newblock {\em arXiv preprint arXiv:0903.3143}, 2009.

\bibitem{garsia-haiman}
A.~M. Garsia and M.~Haiman.
\newblock A remarkable {$q,t$}-{C}atalan sequence and {$q$}-{L}agrange
  inversion.
\newblock {\em J. Algebraic Combin.}, 5(3):191--244, 1996.

\bibitem{GeissLeclercSchroeer2017}
C.~Geiss, B.~Leclerc, and J.~Schr\"oer.
\newblock {Quivers with relations for symmetrizable Cartan matrices I:
  Foundations}.
\newblock {\em Invent. Math.}, 209:61--158, 2017.

\bibitem{HLV11}
T.~Hausel, E.~Letellier, and F.~Rodriguez-Villegas.
\newblock Arithmetic harmonic analysis on character and quiver varieties.
\newblock {\em Duke Math. J.}, 160(2):323--400, 2011.

\bibitem{HLV13}
T.~Hausel, E.~Letellier, and F.~Rodriguez-Villegas.
\newblock {Arithmetic harmonic analysis on character and quiver varieties II}.
\newblock {\em Adv. in Math.}, 234:85--128, 2013.

\bibitem{HMW16}
T.~Hausel, M.~Mereb, and M.~L. Wong.
\newblock Arithmetic and representation theory of wild character varieties.
\newblock {\em Journal of the European Mathematical Society},
  21(10):2995--3052, 2019.

\bibitem{hausel-villegas}
T.~Hausel and F.~Rodriguez-Villegas.
\newblock Mixed {H}odge polynomials of character varieties.
\newblock {\em Invent. Math.}, 174(3):555--624, 2008.
\newblock With an appendix by Nicholas M. Katz.

\bibitem{Hiroe}
K.~Hiroe.
\newblock {Linear differential equations on the Riemann sphere and
  representations of quivers}.
\newblock {\em Duke Math. J.}, 166(5):855--935, 2017.

\bibitem{HiroeYamakawa}
K.~Hiroe and D.~Yamakawa.
\newblock Moduli spaces of meromorphic connections and quiver varieties.
\newblock {\em Adv. in Math.}, 266:120--151, 2014.

\bibitem{HKLR}
N.~Hitchin, A.~Karlhede, U.~Lidstr\"om, and M.~Ro\v{c}ek.
\newblock Hyperk\"ahler metrics and supersymmetry.
\newblock {\em Comm.\ Math.\ Phys.}, 108:535--589, 1987.

\bibitem{InabaSaito2013}
M.~Inaba and M.-H. Saito.
\newblock {Moduli of unramified irregular singular parabolic connections on a
  smooth projective curve}.
\newblock {\em {Kyoto Journal of Mathematics}}, 53(2):433--482, 2013.

\bibitem{Kovalev1996}
A.~Kovalev.
\newblock Nahm's equations and complex adjoint orbits.
\newblock {\em Quart.~J.~Math.~Oxford (2)}, 47:41--58, 1996.

\bibitem{Kr88}
P.~Kronheimer.
\newblock {A hyperk\"ahler structure on the cotangent bundle of a complex Lie
  group}.
\newblock {\em arXiv preprint math/0409253}, 1988.

\bibitem{Kronheimer1990}
P.~Kronheimer.
\newblock A hyper-k\"ahlerian structure on coadjoint orbits of a semsimple
  complex group.
\newblock {\em J.~London Math.~Soc.~(2)}, 42:193--208, 1990.

\bibitem{Kr89}
P.~B. Kronheimer.
\newblock The construction of {ALE} spaces as hyper-{K}\"{a}hler quotients.
\newblock {\em J. Differential Geom.}, 29(3):665--683, 1989.

\bibitem{kronheimer}
P.~B. Kronheimer.
\newblock A {T}orelli-type theorem for gravitational instantons.
\newblock {\em J. Differential Geom.}, 29(3):685--697, 1989.

\bibitem{Le13}
E.~Letellier.
\newblock Quiver varieties and the character ring of general linear groups over
  finite fields.
\newblock {\em J. Eur. Math. Soc. (JEMS)}, 15(4):1375--1455, 2013.

\bibitem{Le15}
E.~Letellier.
\newblock {DT-Invariants of Quivers and the Steinberg Character of GL n}.
\newblock {\em IMRN}, 2015(22):11887--11908, 2015.

\bibitem{Lo14}
I.~Losev.
\newblock Masterclass on quantized quiver varieties, {L}ecture 1.
\newblock \url{https://web.northeastern.edu/iloseu/nak_quiv_var2.pdf}, 2014.

\bibitem{Mac}
I.~G. Macdonald.
\newblock {\em {Symmetric functions and Hall polynomials}}.
\newblock Ox. Univ. Press., 1998.

\bibitem{MW74}
J.~Marsden and A.~Weinstein.
\newblock Reduction of symplectic manifolds with symmetry.
\newblock {\em Reports on mathematical physics}, 5(1):121--130, 1974.

\bibitem{MMOPR}
J.~E. Marsden, G.~Misio\l{}ek, J.-P. Ortega, M.~Perlmutter, and T.~S. Ratiu.
\newblock {\em Hamiltonian Reduction by Stages}, volume 1913 of {\em Springer
  Lecture Notes in Mathematics}.
\newblock Springer-Verlag Publishing Company Inc., 2006.

\bibitem{Ma89}
H.~Matsumura.
\newblock {\em Commutative ring theory}, volume~8.
\newblock Cambridge university press, 1989.

\bibitem{Mayrand}
M.~Mayrand.
\newblock Stratified hyperk{\"a}hler spaces from semisimple lie algebras.
\newblock {\em Transformation Groups}, pages 1--25, 2018.

\bibitem{Mi17}
J.~S. Milne.
\newblock {\em Algebraic Groups}.
\newblock Cam. Univ. Press, 2017.

\bibitem{Hitchin1984}
{N.~Hitchin}.
\newblock {Twistor construction of Einstein metrics}.
\newblock In {\em {Global Riemannian geometry}}. {Chichester: Ellis Horwood
  Ltd.}, 1984.

\bibitem{nak94}
H.~Nakajima.
\newblock Instantons on {ALE} spaces, quiver varieties, and {K}ac-{M}oody
  algebras.
\newblock {\em Duke Math. J.}, 76(2):365--416, 1994.

\bibitem{serre58}
J.-P. Serre.
\newblock Espaces fibr{\'e}s alg{\'e}briques.
\newblock {\em S{\'e}minaire Claude Chevalley}, 3:1--37, 1958.

\bibitem{siersma-tibar}
D.~Siersma and M.~Tibar.
\newblock Singularities at infinity and their vanishing cycles.
\newblock {\em Duke Math. J.}, 80(3):771--783, 1995.

\bibitem{Si94}
C.~T. Simpson.
\newblock {Moduli of representations of the fundamental group of a smooth
  projective variety I}.
\newblock {\em Pub. Math. IHES}, 79(1):47--129, 1994.

\bibitem{stacks-project}
T.~{Stacks project authors}.
\newblock The stacks project.
\newblock \url{https://stacks.math.columbia.edu}, 2022.

\bibitem{WY16}
D.~Wyss.
\newblock {Motivic classes of Nakajima quiver varieties}.
\newblock {\em IMRN}, 2017(22):6961--6976, 2016.

\bibitem{Yamakawa2008}
D.~Yamakawa.
\newblock Geometry of multiplicative preprojective algebra.
\newblock {\em Int.~Math.~Res.~Pap.}, (rpn 008), 2008.

\bibitem{YA10}
D.~Yamakawa.
\newblock {Quiver varieties with multiplicities, Weyl groups of non-symmetric
  Kac--Moody algebras, and Painlev{\'e} equations}.
\newblock {\em SIGMA}, 6:087, 2010.

\bibitem{yamakawanotes}
D.~Yamakawa.
\newblock {Notes on the preprojective algebras of Geiss-Leclerc-Schr\"oer}.
\newblock {\em Letter to B. Leclerc}, 2015.

\end{thebibliography}
\end{document}